\pdfoutput=1
\documentclass[leqno,10pt]{amsart}
\usepackage[normal]{phaine_style}
\usepackage[margin=1.175in]{geometry}

\DeclareSourcemap{
  \maps[datatype=bibtex]{
    \map{
      \step[fieldsource=note, final]
      \step[fieldset=addendum, origfieldval, final]
      \step[fieldset=note, null]
    }
  }
}

\DeclareMathOperator{\image}{im}

\renewcommand{\Equiv}{\mathrm{Equiv}}

\newcommand{\smalltop}{\mathrm{top}}

\newcommand{\Tatep}{\Tup_{p'}}

\DeclareMathOperator{\Kum}{Kum}
\DeclareMathOperator{\can}{can}

\newcommand{\units}{\cross}

\DeclareMathOperator{\Units}{\Ocal^{\units}}
\DeclareMathOperator{\Unitsone}{\Ocal_{1}^{\units}}

\DeclareMathOperator{\Res}{Res}

\DeclareMathOperator{\OTop}{\Ocal_{\smalltop}^{\units}}

\newcommand{\chitop}{\upchi_{\smalltop}}

\newcommand{\eq}{\mathrm{eq}}

\newcommand{\evhat}{\widehat{\ev}}

\newcommand{\NNpprime}{\NN_{p'}}

\newcommand{\ZZell}{\ZZ_{\ell}}

\newcommand{\ZZpprime}{\ZZhat_{p'}}
\newcommand{\QQell}{\QQ_{\ell}}

\newcommand{\mult}{\mathrm{mult}}

\newcommand{\mup}{\mathrm{m}}
\newcommand{\mubf}{\mathbf{\mu}}
\newcommand{\mum}{\mubf_m}
\newcommand{\mun}{\mubf_n}

\newcommand{\GGmet}{\GG_{\mup,\et}}

\newcommand{\ftpts}{\ensuremath{\textup{ft-pts}}}
\newcommand{\Xftpts}{X_{\ftpts}}

\DeclareMathOperator{\res}{res}
\DeclareMathOperator{\UH}{UH}
\newcommand{\UHres}{\UH_{\res}}

\usepackage{xargs}  
\newcommandx{\PHtodo}[2][1=]{\todo[linecolor=blue,backgroundcolor=blue!25,bordercolor=blue,#1]{#2}}

\newtheorem{claim}[equation]{Claim}
\theoremstyle{definition}

\newtheorem*{remark*}{Remark}

\newcommand{\arcdescent}{\xspace{arc-de\-scent}\xspace}

\newcommand{\arctopology}{\xspace{arc-topol\-o\-gy}\xspace}

\newcommand{\hcover}{\xspace{h-cover}\xspace}

\newcommand{\hhypercover}{\xspace{h-hy\-per\-cover}\xspace}

\newcommand{\hdescent}{\xspace{h-de\-scent}\xspace}
\newcommand{\hlocal}{\xspace{h-lo\-cal}\xspace}
\newcommand{\hlocally}{\xspace{h-lo\-cal\-ly}\xspace}
\newcommand{\hsheaf}{\xspace{h-sheaf}\xspace}

\newcommand{\htopology}{\xspace{h-topol\-o\-gy}\xspace}

\newcommand{\RTop}{\categ{RTop}}

\newcommand{\LTop}{\categ{LTop}}

\newcommand{\FFp}{\FF_{p}}
\newcommand{\ZZhat}{\widehat{\ZZ}}

\newcommand{\awn}{\ensuremath{\textup{awn}}}
\DeclareMathOperator{\norm}{norm}
\DeclareMathOperator{\sn}{sn}
\newcommand{\tft}{\ensuremath{\textup{tft}}}

\newcommand{\kawn}{k^{\awn}}
\newcommand{\Aawn}{A^{\awn}}

\newcommand{\Sawn}{S^{\awn}}
\newcommand{\Wawn}{W^{\awn}}
\newcommand{\Xawn}{X^{\awn}}
\newcommand{\Yawn}{Y^{\awn}}

\newcommand{\kperf}{k^{\perf}}

\newcommand{\Xperf}{X^{\perf}}
\newcommand{\Yperf}{Y^{\perf}}

\newcommand{\Xsn}{X^{\sn}}
\newcommand{\Ysn}{Y^{\sn}}

\newcommand{\Schawn}{\Sch^{\awn}}
\newcommand{\Schawntft}{\Sch^{\awn,\tft}}
\newcommand{\Schsn}{\Sch^{\sn}}
\newcommand{\Schperf}{\Sch^{\perf}}
\newcommand{\Schft}{\Sch^{\ft}}

\DeclareMathOperator{\pin}{pin}
\newcommand{\Hompin}{\Hom^{\pin}}
\newcommand{\Hompink}{\Hom_{k}^{\pin}}

\newcommand{\RToppin}{\RTop^{\pin}}

\newcommand{\pprimecomp}{_{p'}^{\wedge}}

\newcommand{\piet}{\uppi^{\et}}

\DeclareMathOperator{\characteristic}{char}

\renewcommand{\Open}{\mathrm{Open}}

\newcommand{\aff}{\ensuremath{\textup{aff}}}
\renewcommand{\ft}{\ensuremath{\textup{ft}}}

\newcommand{\sep}{\mathrm{sep}}
\newcommand{\qcs}{\mathrm{qcs}}

\newcommand{\Et}{\categ{Ét}}
\newcommand{\Etaff}{\Et{}^{\aff}}

\newcommand{\Etsep}{\Et{}^{\sep}}

\newcommand{\Etqcs}{\Et{}^{\qcs}}

\newcommand{\Uet}{U_{\et}}
\newcommand{\Vet}{V_{\et}}
\newcommand{\Wet}{W_{\et}}
\newcommand{\Xet}{X_{\et}}
\newcommand{\Yet}{Y_{\et}}
\newcommand{\Zet}{Z_{\et}}

\newcommand{\Xetawn}{X_{\et}^{\awn}}
\newcommand{\Yetawn}{Y_{\et}^{\awn}}

\newcommand{\Het}{\Hup_{\et}}

\DeclareMathOperator{\cl}{cl}
\newcommand{\Xcl}{X_{\cl}}

\DeclareMathOperator{\map}{map}

\makeatletter
\def\swappedhead#1#2#3{%
  \thmnumber{\@upn{\the\thm@headfont#2\@ifnotempty{#1}{\,~}}}%
  \thmname{#1}%
  \thmnote{ {\the\thm@notefont{#3}}}}
\makeatother

\makeatletter
\def\@seccntformat#1{%
  \protect\textup{\protect\@secnumfont
    \ifnum\pdfstrcmp{subsection}{#1}=0 \bfseries\fi
    \csname the#1\endcsname
    \protect{\,}
  }%
}  
\makeatother

\usepackage{secdot}

\makeatletter
\def\@tocline#1#2#3#4#5#6#7{\relax
  \ifnum #1>\c@tocdepth 
  \else
    \par \addpenalty\@secpenalty\addvspace{#2}%
    \begingroup \hyphenpenalty\@M
    \@ifempty{#4}{%
      \@tempdima\csname r@tocindent\number#1\endcsname\relax
    }{%
      \@tempdima#4\relax
    }%
    \parindent\z@ \leftskip#3\relax \advance\leftskip\@tempdima\relax
    \rightskip\@pnumwidth plus4em \parfillskip-\@pnumwidth
    #5\leavevmode\hskip-\@tempdima
      \ifcase #1
       \or\or \hskip 1em \or \hskip 2em \else \hskip 3em \fi%
      #6\nobreak\relax
    \hfill\hbox to\@pnumwidth{\@tocpagenum{#7}}\par
    \nobreak
    \endgroup
  \fi}
\makeatother

\title{Reconstruction of schemes from their étale topoi}

\author{Magnus Carlson}

\author{Peter J. Haine} 

\author{Sebastian Wolf}

\date{\today}

\begin{document}


\begin{abstract} 
	Let $ k $ be a field that is finitely generated over its prime field.
	In Grothendieck's anabelian letter to Faltings, he conjectured that sending a $ k $-scheme to its étale topos defines a fully faithful functor from the localization of the category of finite type $ k $-schemes at the universal homeomorphisms to a category of topoi.
	We prove Grothendieck's conjecture for infinite fields of arbitrary characteristic.
	In characteristic $ 0 $, this shows that seminormal finite type $ k $-schemes can be reconstructed from their étale topoi, generalizing work of Voevodsky.
	In positive characteristic, this shows that perfections of finite type $ k $-schemes can be reconstructed from their étale topoi.
\end{abstract}

\maketitle

\setcounter{tocdepth}{1}
\tableofcontents


\setcounter{section}{-1}

\section{Introduction}

Let $ k $ be a finitely generated field of characteristic $ 0 $. 
In Grothendieck's influential \textit{anabelian letter} to Faltings \cite{MR1483108}, he made three classes of conjectures regarding the reconstruction of $ k $-schemes.
The first class is about describing isomorphisms between `anabelian' $ k $-schemes in terms of outer isomorphisms between their étale fundamental groups; the second is the \textit{section conjecture} about describing the set of rational points of an `anabelian' $ k $-scheme $ X $ in terms of sections of the natural homomorphism $ \fromto{\piet_1(X)}{\Gal(\kbar/k)} $.
This paper concerns the third conjecture, which is about reconstructing $ k $-schemes from their étale topoi \cite[p. 7]{MR1483108}. 
The main result of this paper is to prove an extension of Grothendieck's conjecture.

Write $ \Schft_{k} $ for the category of finite type $ k $-schemes and $ \RTop_k $ for the category of topoi over the étale topos of $ \Spec(k) $.
Grothendieck's conjecture asks to what extent the functor
\begin{equation*}
	(-)_{\et} \colon \Schft_{k} \to \RTop_{k} \period
\end{equation*}
which takes a scheme $ X $ to its étale topos $ \Xet $ is fully faithful.  
In order to make this plausible, two technical modifications are needed.
For the first, note that every morphism between finite type $ k $-schemes sends closed points to closed points, but this is not true of an arbitrary geometric morphism between étale topoi (see \Cref{ex:a_geometric_morphism_that_isn't_pinned}).
Hence we must add this restriction to the morphisms of topoi considered.
We say that a geometric morphism of topoi $ \fromto{\Xcal}{\Ycal} $ is \textit{pinned} if the induced map on underlying spaces sends closed points to closed points.
Write \smash{$ \RToppin_{k} $} for the category of topoi over $ \Spec(k)_{\et} $ and pinned geometric morphisms. 

The second is that by the \emph{topological invariance} of the étale topos \stacks{03SI},
any universal homeomorphism of schemes induces an equivalence of étale topoi.
Hence the optimal result would be that the functor $ \goesto{X}{\Xet} $ becomes fully faithful after localizing \smash{$ \Schft_k $} by inverting the universal homeomorphisms.
Grothendieck's precise conjecture is that this is true when $ k $ is finitely generated of characteristic $ 0 $.
Our main result is a proof of this conjecture, together with its natural extension to infinite fields of positive characteristic.

Denote the collection of universal homeomorphisms by $ \UH $.  

\begin{theorem}\label{intro_thm:etale_reconstruction}
	Let $ k $ be a field that is finitely generated over its prime field.
	If $ \characteristic(k) > 0 $, assume that $ k $ has transcendence degree at least $ 1 $.
	Then the functor
	\begin{equation*}
		(-)_{\et} \colon \fromto{\Schft_{k}[\UH^{-1}]}{\RToppin_{k}}
	\end{equation*}
	sending a finite type $ k $-scheme $ X $ to its étale topos $ \Xet $ is fully faithful.
\end{theorem}

The first ingredient in our proof of \Cref{intro_thm:etale_reconstruction} is an explicit description of the source category.
For this, recall that a scheme $ X $ is \emph{absolutely weakly normal} if every universal homeomorphism $ X' \to X $ is an isomorphism.
Over fields, this reduces to more well-known notions.
A characteristic $ 0 $ scheme is absolutely weakly normal if and only if it is \emph{seminormal}, i.e., does not have cuspidal cubic singularities.
A positive characteristic scheme is absolutely weakly normal if and only if it is \emph{perfect}.
The inclusion of the full subcategory of absolutely weakly normal schemes into all schemes admits a right adjoint
\begin{equation*}
	(-)^{\awn} \colon \Sch \to \Schawn
\end{equation*}
called \emph{absolute weak normalization}.
Moreover, absolute weak normalization induces an equivalence of categories \smash{$ \equivto{\Sch[\UH^{-1}]}{\Schawn} $}.
As a result, in order to prove \Cref{intro_thm:etale_reconstruction} we equivalently need to show that for all finite type $ k $-schemes $ X $ and $ Y $, the natural map
\begin{equation}\label{equ:etale_reconstr_map}
	\Hom_{k}(\Xawn,Y) \to \Hompink(\Xetawn,\Yet)
\end{equation}
is bijective; here the right-hand side is the set of isomorphism classes of pinned geometric morphisms $ \fromto{\Xetawn}{\Yet} $ over $ \Spec(k)_{\et} $.
By the Mordell--Weil Theorem, \Cref{intro_thm:etale_reconstruction} follows from the following more refined result:

\begin{theorem}[(see \Cref{prop:equivalent_conditions_for_etale_reconstruction,thm:reconstruction_for_amenable_fields})]\label{intro_thm:reconstruction_for_amenable_fields}
	Let $ k $ be a field of exponential characteristic $ p $.
	Assume that:
	\begin{enumerate}[label=\stlabel{intro_thm:reconstruction_for_amenable_fields}, ref=\arabic*]
		\item\label{intro_thm:reconstruction_for_amenable_fields.1} For any torus $ T $ over $ k $, the only element of $T(k)$ divisible by all integers coprime to $ p $ is the identity.

		\item\label{intro_thm:reconstruction_for_amenable_fields.2} For any regular finite type $ k $-scheme $ X $, the only torsion element of $ \Pic(X) $ that is divisible by all integers coprime to $ p $ is the identity.

		\item\label{intro_thm:reconstruction_for_amenable_fields.3} The field $ k $ admits a nontrivial discrete valuation.
	\end{enumerate}
	Then for all finite type $ k $-schemes $ X $ and $ Y $ the natural map 
	\begin{equation*}
		\Hom_{k}(\Xawn,Y) \to \Hompink(\Xetawn,\Yet)
	\end{equation*}
	is bijective.
\end{theorem}


\subsection*{Voevodsky's work}\label{intro_subsec:Voevodskys_work}

We conclude the introduction by noting that \Cref{intro_thm:reconstruction_for_amenable_fields} improves upon a result of Voevodsky \cite{MR1098621}.
Voevodsky proved that if $ k $ is a finitely generated field of characteristic $ 0 $ and $ X $ and $ Y $ are finite type $ k $-schemes with $ X $ normal, then the natural map 
\begin{equation*}
	\Hom_{k}(X,Y) \to \Hompink(\Xet,\Yet)
\end{equation*}
is bijective. 
\Cref{intro_thm:reconstruction_for_amenable_fields} generalizes this result in two directions.
First, \Cref{intro_thm:reconstruction_for_amenable_fields} also holds in positive characteristic.
Second, even in charactistic $ 0 $, \Cref{intro_thm:reconstruction_for_amenable_fields} applies when $ X $ is only seminormal, and seminormal schemes are quite a lot more general than normal schemes. 
For example, the nodal cubic is not normal, but \textit{is} seminormal.

In Voevodsky's proof of \Cref{intro_thm:etale_reconstruction} for normal schemes over a finitely generated field of characteristic $ 0 $, the fact that the group $ \Pic(X) $ is finitely generated is crucial. 
However, for seminormal schemes, $ \Pic(X) $ is in general not finitely generated and it is not clear how to adapt Voevodsky's argument in this more general setting.%
\footnote{More precisely, it suffices to know that $ \Pic(X) $ has no infinitely divisible torsion elements. If $ X $ is seminormal and $ \Sup_2 $, \cite[Theorem 6.5]{MR1399035} implies that $ \Pic(X) $ has no infinitely divisible torsion elements. 
Without the $ \Sup_2 $ assumption, we are neither aware of a proof nor a counterexample.
This is also observed in forthcoming work by Zach Berens.}
Our proof bypasses these issues by proceeding as follows.
We show that both the source and target of \eqref{equ:etale_reconstr_map} are \hlocal in $ X $.
Therefore, using de Jong's alterations, we may reduce to the case that $ X $ is regular.
In particular, in characteristic $ 0 $, we can reduce to Voevodsky's argument.
More generally, the assumption that $ k $ is amenable and admits a nontrivial discrete valuation allows us to adapt Voevodsky's methods to work in arbitrary characteristic.

\begin{remark*}
	Let $ k $ be an infinite finitely generated field of positive characteristic.
	Since this paper was written, we learned that in forthcoming work Zach Berens independently shows that the proof in Voevodsky's paper can be adapted to show that perfections of finite type $ k $-schemes can be reconstructed from their étale topoi.
\end{remark*}


\subsection*{Linear overview}\label{intro_subsec:linear_overview} 

\Cref{sec:background_on_absolute_weak_normalization} provides background on absolute weak normalization, seminormalization, and perfection.
Part of the difficulty in proving the positive characteristic cases of \Cref{intro_thm:etale_reconstruction,intro_thm:reconstruction_for_amenable_fields} is that absolute weak normalizations of finite type schemes are typically no longer of finite type or even noetherian.
To effectively deal with this issue we introduce morphisms \emph{topologically of finite type}, i.e., morphisms that factor as a universal homeomorphism followed by a morphism of finite type, and prove some basic results about them.

In \cref{sec:the_etale_reconstruction_property}, we introduce pinned geometric morphisms as well as what it means for a field $ k $ to satisfy the conclusion of \Cref{intro_thm:etale_reconstruction}.
We call the latter the \textit{étale reconstruction property}.
We give a number of equivalent formulations of étale reconstruction (see \Cref{prop:equivalent_conditions_for_etale_reconstruction}) and discuss some immediate consequences (see \Cref{prop:consequences_of_etale_reconstruction}).
In \cref{sec:proof_of_injectivity}, we prove that the injectivity part of the reconstruction theorem holds over any field, in fact any qcqs base scheme $ S $ (see \Cref{thm:injectivity_part_of_reconstruction,cor:injectivity_part_of_reconstruction_for_Jacobson_schemes}).
In \cref{sec:reduction_to_regular_source_and_target_GGm}, we prove that for any $ S $-scheme $ Y $, the functor $ X \mapsto \Hom_S(\Xawn,Y) $ satisfies \hdescent.
We furthermore show that the étale topos satisfies \hdescent.
Using alterations, this lets us reduce proving \Cref{intro_thm:reconstruction_for_amenable_fields} to the case where $ X $ is an affine, regular, connected, finite type $ k $-scheme that admits a rational point and $ Y = \GGm $ (\Cref{thm:reduction_to_the_regular_case}).
\Cref{sec:amenable_fields} introduces the class of fields satisfying condition \enumref{intro_thm:reconstruction_for_amenable_fields}{1} and a slight generalization of \enumref{intro_thm:reconstruction_for_amenable_fields}{2}; we call these \textit{amenable} fields, see \Cref{def:amenable}.
We also prove some basic results about amenable fields.
In \cref{sec:proof_of_the_reconstruction_theorem}, we prove \Cref{intro_thm:reconstruction_for_amenable_fields}.
The proof given in \cref{sec:proof_of_the_reconstruction_theorem} generalizes the main argument in Voevodsky's paper \cite[\S3]{MR1098621} to work in positive characteristic.
It also streamlines a number of points.

Throughout the paper we use the language of topoi in our results and proofs, in contrast to Voevodsky who chose to work with sites.
In \Cref{app:morphisms_of_etale_sites_and_etale_topoi}, we prove that the two perspectives are equivalent for our purposes: any geometric morphism between étale topoi of topologically noetherian schemes is induced by a morphism of (quasicompact separated) étale sites.


\subsection*{Notational conventions}

Let $ f \colon \fromto{X}{Y} $ be a morphism of schemes.
We write:
\begin{enumerate}
	\item $ |X| $ for the underlying topological space of $ X $.

	\item $ \Sch $ for the category of schemes and  $ \Sch_{X} $ for the category of $ X $-schemes.

	\item $ \Et_{X} \subset \Sch_{X} $ for the full subcategory spanned by the étale $ X $-schemes.
	We regard $ \Et_X $ as a site equipped with the étale topology.

	\item $ \Xet \colonequals \Sh_{\et}(\Et_{X}) $ for the étale topos of $ X $.

	\item $ \flowerstar \colon \fromto{\Xet}{\Yet} $ for the geometric morphism induced by the morphism of sites $ X \cross_Y (-) \colon \fromto{\Et_{Y}}{\Et_{X}} $.
\end{enumerate}


\subsection*{Acknowledgments}

We thank Clark Barwick, Jakob Stix, and Bogdan Zavyolov for many enlightening discussions around the contents of this paper.
We especially thank Jakob for helpful comments on a draft of this work. The first-named author acknowledges support by Deutsche Forschungsgemeinschaft (DFG) through the Collaborative Research Centre TRR 326 `Geometry and Arithmetic of Uniformized Structures', project number 444845124.
The second-named author gratefully acknowledges support from the NSF Mathematical Sciences Postdoctoral Research Fellowship under Grant \#DMS-2102957 and a grant from the Simons Foundation (816048, LC). 
The third-named author acknowledges support from Collaborative Research Centre SFB 1085 `Higher Invariants' funded by the DFG.


\section{Background on absolute weak normalization}\label{sec:background_on_absolute_weak_normalization}

As mentioned in the introduction, if $ f \colon \fromto{X}{Y} $ is a universal homeomorphism of schemes, then by the topological invariance of the étale topology \stacks{03SI}, the induced geometric morphism of étale topoi $ \flowerstar \colon \fromto{\Xet}{\Yet} $ is an equivalence of categories.
Thus, in order to have a reconstruction theorem, we must invert universal homeomorphisms of schemes.
Said differently, we must work with schemes with the property that universal homeomorphisms between them are automatically isomorphisms.

In this section, we explain how to do this.
The solution is to work with \textit{absolutely weakly normal schemes}.
This notion was originally introduced by Rydh \cite[Appendix B]{MR2679038} in order to study effective descent for étale morphisms, and further studied by Barwick \cite{arXiv:1012.1889} for the purpose of inverting universal homeomorphisms.
For schemes over fields, the theory reduces to previously studied notions.
A scheme $ X $ in positive characteristic is absolutely weakly normal if and only if $ X $ is \textit{perfect}, i.e., the absolute Frobenius $ \fromto{X}{X} $ is an isomorphism.
In characteristic $ 0 $, absolute weak normality reduces to \textit{seminormality}.
Seminormality was originally introduced by Traverso \cite{MR0277542} in order to characterize the rings for which the Picard group is homotopy-invariant, and later simplified by Swan \cite{MR595029}.

In \cref{subsec:absolute_weak_normalization_seminormalization_and_perfection}, we recall the definitions of absolute weak normality, seminormality, and perfection.
In \cref{subsec:inverting_universal_homeomorphisms}, we explain why absolute weak normalization inverts universal homeomorphisms.
In \cref{subsec:finiteness_properties}, we discuss the interaction between seminormalization, perfection, and finite type hypotheses on morphisms of schemes.
In particular, the perfection of a finite type $ \FF_p $-scheme is almost never of finite type (see \Cref{prop:the_only_perfect_noetherian_schemes_are_products_of_fields}).
As a result, we introduce the class of schemes \textit{topologically of finite type} (\Cref{def:topologically_of_finite_type}). 
This notion behaves well with respect to perfection and is needed for the positive characteristic case of \Cref{intro_thm:etale_reconstruction}.


\subsection{Seminormalization, perfection, and absolute weak normalization}\label{subsec:absolute_weak_normalization_seminormalization_and_perfection}

An elementary example of a universal homeomorphism is the normalization of the cuspidal cubic.
Seminormal schemes are the schemes that do not have cuspidal cubic singularities:

\begin{definition}[(seminormality)]\label{def:semi_normality}
	Let $ A $ be a ring.
	We say that $ A $ is \defn{seminormal} if for all $ x,y \in A $ such that $ x^3 = y^2 $, there exists an element $ a \in A $ such that $ x = a^2 $ and $ y = a^3 $.
\end{definition}

\begin{proposition}\label{prop:basic_properties_of_seminormal_rings}
	Let $ A $ be a ring.
	\begin{enumerate}[label=\stlabel{prop:basic_properties_of_seminormal_rings}, ref=\arabic*]
		\item\label{prop:basic_properties_of_seminormal_rings.1} Assume that $ A $ is seminormal and let $ x,y \in A $ be such that $ x^3 = y^2 $.
		Then there exists a \emph{unique} element $ a \in A $ such that $ x = a^2 $ and $ y = a^3 $.

		\item\label{prop:basic_properties_of_seminormal_rings.2} If $ A $ is seminormal, then $ A $ is reduced.

		\item\label{prop:basic_properties_of_seminormal_rings.3} The ring $ A $ is seminormal if and only if for each maximal ideal $ \mfrak \subset A $, the local ring $ A_{\mfrak} $ is seminormal.

		\item\label{prop:basic_properties_of_seminormal_rings.4} If $ A $ is a normal ring, then $ A $ is seminormal.
	\end{enumerate}
\end{proposition}

\begin{proof}
	For \enumref{prop:basic_properties_of_seminormal_rings}{1}, see \cite[{\href{https://stacks.math.columbia.edu/tag/0EUK\#footnote-1-mark}{Tag 0EUK, Footnote 1}}]{stacksproject} or \cite{MR662270}. 
	Item \enumref{prop:basic_properties_of_seminormal_rings}{2} is \stacks{0EUQ}.
	Item \enumref{prop:basic_properties_of_seminormal_rings}{3} is \cite[Proposition 3.7]{MR595029}.

	For \enumref{prop:basic_properties_of_seminormal_rings}{4}, note that since $ A $ is normal, for each maximal ideal $ \mfrak \subset A $, the localization $ A_{\mfrak} $ is a normal domain.
	Hence by \enumref{prop:basic_properties_of_seminormal_rings}{3} it suffices to prove the claim in the case where $ A $ is a normal domain.
	In this case, since $ A $ is a domain, it suffices to show that if $ x,y \in A \sminus \{0\} $ and $ x^3 = y^2 $, there exists an element $ a \in A $ such that $ x = a^2 $ and $ y = a^3 $.
	In the fraction field of $ A $, we have
	\begin{equation*}
		\paren{\frac{y}{x}}^2 = \frac{x^3}{x^2} = x \period
	\end{equation*}
	Hence the polynomial $ t^2 - x \in A[t] $ is an integral equation for $ y/x $.
	Since $ A $ is integrally closed, we see that $ y/x \in A $.
	Hence $ a = y/x $ is an element such that $ x = a^2 $ and $ y = a^3 $.
\end{proof}





In positive characteristic, another example of a universal homeomorphism is the Frobenius: 

\begin{recollection}[{(perfect $ \FFp $-schemes; see \cite[\S3]{MR3674218})}]
	Let $ p $ be a prime number, and let $ X $ be an $ \FFp $-scheme.
	Write $ \Frob_X \colon \fromto{X}{X} $ for the \defn{absolute Frobenius}: on underlying spaces, $ \Frob_X $ is the identity, and on structure sheaves $ \Frob_X $ is given by $ \goesto{a}{a^p} $.
	Also recall that $ \Frob_X $ is a universal homeomorphism \stacks{0CC8}.
	We say that $ X $ is \defn{perfect} if $ \Frob_X $ is an isomorphism.
	
	The \defn{perfection} of $ X $ is the cofiltered limit
	\begin{equation*}
		\Xperf \isomorphic \lim \bigg(
		\begin{tikzcd}[sep=2.5em]
			\cdots \arrow[r, "\Frob_X"] & X \arrow[r, "\Frob_X"] & X 
		\end{tikzcd}\bigg) \period
	\end{equation*}
	(Note that since the absolute Frobenius is an affine morphism, this limit exists in the category of schemes.)
	The $ \FFp $-scheme $ \Xperf $ is always perfect.
	Write
	\begin{equation*}
		\Schperf_{\FFp} \subset \Sch_{\FFp}
	\end{equation*}
	for the full subcategory spanned by the perfect $ \FFp $-schemes.
	Then the assignment $ \goesto{X}{\Xperf} $ is right adjoint to the inclusion \smash{$ \Schperf_{\FFp} \subset \Sch_{\FFp} $}.
\end{recollection}

To invert universal homeomorphisms for schemes over $ \Spec(\ZZ) $, it turns out that we just need to make sure that our schemes do not have a simple class of plane singularity types involving prime powers.

\begin{definition}[(absolute weak normality)]\label{def:absolute_weak_normality}
	A ring $ A $ is \defn{absolutely weakly normal} if the following conditions are satisfied:
	\begin{enumerate}[label=\stlabel{def:absolute_weak_normality}, ref=\arabic*]
		\item\label{def:absolute_weak_normality.1} The ring $ A $ is seminormal.

		\item\label{def:absolute_weak_normality.2} For each prime number $ \el $ and all elements $ x,y \in A $ such that $ \el^{\el} x = y^\el $, there is a unique element $ a \in A $ such that $ x = a^\el $ and $ y = \el a $.
	\end{enumerate}
\end{definition}

\begin{observation}\label{obs:absolute_weak_normality_with_p_invertible}
	Let $ A $ be a ring and $ \el $ a prime number.
	If $ \el $ is invertible in $ A $, then given elements $ x,y \in A $ such that $ \el^{\el} x = y^\el $, there is a unique element $ a \in A $ such that $ x = a^\el $ and $ y = \el a $; namely $ a = y/\el $.
\end{observation}

\begin{lemma}[{\stacks{0EUM}}]
	Seminormality and absolute weak normality are Zariski local properties of rings. 
\end{lemma}

\begin{definition}
	A scheme $ X $ is \defn{seminormal} (resp., \defn{absolutely weakly normal}) if for every affine open $ U \subset X $, the ring $ \Gamma(U;\Ocal_X) $ is seminormal (resp., absolutely weakly normal).
\end{definition}

\begin{lemma}[{\stacks{0EUP}}]\label{lem:seminormality_and_absolute_weak_normality_are_affine_local}
	Seminormality and absolutely weakly normality are affine local properties of schemes.
	In particular, if $ X $ is a seminormal (resp., absolutely weakly normal) scheme, then every open subscheme of $ X $ is seminormal (resp., absolutely weakly normal).
\end{lemma}

Over fields, absolute weak normality reduces to seminormality or perfectness:
	
\begin{theorem}[{(absolutely weakly normal schemes over fields)}]\label{thm:absolutely_weakly_normal_schemes_over_fields}
	\hfill
	\begin{enumerate}[label=\stlabel{thm:absolutely_weakly_normal_schemes_over_fields}, ref=\arabic*]
		\item\label{thm:absolutely_weakly_normal_schemes_over_fields.1} A $ \QQ $-scheme $ X $ is absolutely weakly normal if and only if $ X $ is seminormal.

		\item\label{thm:absolutely_weakly_normal_schemes_over_fields.2} 
		An $ \FFp $-scheme $ X $ is absolutely weakly normal if and only if $ X $ is perfect.
	\end{enumerate}
\end{theorem}

\begin{proof}
	For \enumref{thm:absolutely_weakly_normal_schemes_over_fields}{1}, since absolute weak normality and seminormality are affine-local properties, it suffices to check the claim when $ X = \Spec(A) $ is the spectrum of a $ \QQ $-algebra.
	Since absolutely weakly normal rings are seminormal, we just need to show that if $ A $ is seminormal, then $ A $ is absolutely weakly normal. 
	Since $ A $ is a $ \QQ $-algebra, this follows from \Cref{obs:absolute_weak_normality_with_p_invertible}.
	
	Item \enumref{thm:absolutely_weakly_normal_schemes_over_fields}{2} is \stacks{0EVV}.
\end{proof}

\begin{corollary}
	A field $ k $ is perfect if and only if $ k $ is absolutely weakly normal.
\end{corollary}


\subsection{Inverting universal homeomorphisms}\label{subsec:inverting_universal_homeomorphisms}

Let $ \Ccal $ be a category and $ W \subset \Mor(\Ccal) $ a collection of morphisms.
Recall that the \textit{localization} of $ \Ccal $ at $ W $ is a category $ \Ccal[W^{-1}] $ equipped with a functor $ L \colon \fromto{\Ccal}{\Ccal[W^{-1}]} $ that sends morphisms in $ W $ to isomorphisms satisfying the following universal property: for any category $ \Dcal $, precomposition with $ L $ defines a fully faithful functor
\begin{equation*}
	-\of L \colon \incto{\Fun(\Ccal[W^{-1}],\Dcal)}{\Fun(\Ccal,\Dcal)}
\end{equation*}
with essential image those functors $ F \colon \fromto{\Ccal}{\Dcal} $ that send morphisms in $ W $ to isomorphisms in $ \Dcal $.

\begin{example}\label{ex:Bousfield_(co)localizations_are_Dwyer-Kan_localizations}
	Let $ F \colon \fromto{\Ccal}{\Ccal_0} $ be a functor.
	If $ F $ admits a fully faithful left or right adjoint, then $ F $ exhibits $ \Ccal_0 $ at the localization at those morphisms $ w $ in $ \Ccal $ such that $ F(w) $ is an isomorphism in $ \Ccal_0 $.
	See \cite[Proposition 5.3.1]{MR1291599} or \cite[Tags \kerodonlink{04JM} \& \kerodonlink{04JR}]{Kerodon}.
\end{example}

We now state the main theorem characterizing the category of absolutely weakly normal schemes as the localization of $ \Sch $ at the universal homeomorphisms.

\begin{notation}
	Write $ \Schsn \subset \Sch $ for the full subcategory spanned by the seminormal schemes and $ \Schawn \subset \Schsn $ for the full subcategory spanned by the absolutely weakly normal schemes. 
	We write $ \UH \subset \Mor(\Sch) $ for the collection of universal homeomorphisms of schemes and write $ \UHres \subset \UH $ for those universal homeomorphisms that induce isomorphisms on residue fields.
\end{notation}

\begin{theorem}[{(existence of absolute weak normalization \& seminormalization)}]\label{thm:existence_of_absolute_weak_normalization_and_seminormalization}
	\hfill
	\begin{enumerate}[label=\stlabel{thm:existence_of_absolute_weak_normalization_and_seminormalization}, ref=\arabic*]
		\item\label{thm:existence_of_absolute_weak_normalization_and_seminormalization.2} The inclusions $ \incto{\Schawn}{\Sch} $ and $ \incto{\Schsn}{\Sch} $ admit right adjoints
		\begin{equation*}
			(-)^{\awn} \colon \fromto{\Sch}{\Schawn} \andeq (-)^{\sn} \colon \fromto{\Sch}{\Schsn}
		\end{equation*}
		Moreover, the counit $ \fromto{\Xawn}{X} $ is a universal homeomorphism and the counit $ \fromto{\Xsn}{X} $ is a universal homeomorphism that induces isomorphisms on residue fields.

		\item\label{thm:existence_of_absolute_weak_normalization_and_seminormalization.3} If $ X $ is a reduced scheme, $ Y $ is an absolutely weakly normal (resp., seminormal) scheme, and $ f \colon \fromto{X}{Y} $ is a universal homeomorphism (resp. universal homeomorphism that induces isomorphisms on residue fields), then $ f $ is an isomorphism.

		\item\label{thm:existence_of_absolute_weak_normalization_and_seminormalization.4} A morphism of schemes $ f \colon \fromto{X}{Y} $ is a universal homeomorphism (resp. universal homeomorphism that induces isomorphisms on residue fields) if and only if the morphism $ f^{\awn} $ (resp., $ f^{\sn} $) is an isomorphism. 

		\item\label{thm:existence_of_absolute_weak_normalization_and_seminormalization.5} The functors
		\begin{equation*}
			(-)^{\awn} \colon \fromto{\Sch}{\Schawn} \andeq (-)^{\sn} \colon \fromto{\Sch}{\Schsn}
		\end{equation*}
		identify $ \Schawn $ and $ \Schsn $ with the localizations $ \Sch[\UH^{-1}] $ and $ \Sch[\UHres^{-1}] $, respectively.
	\end{enumerate}
\end{theorem}

\begin{proof}
	Items \enumref{thm:existence_of_absolute_weak_normalization_and_seminormalization}{2}--\enumref{thm:existence_of_absolute_weak_normalization_and_seminormalization}{3} are \cite[Tags \stackslink{0EUS} \& \stackslink{0H3G}]{stacksproject} for absolute weak normalization and \cite[Tags \stackslink{0EUS} \& \stackslink{0H3G}]{stacksproject} for seminormalization.
	Thre rest of the proof is exactly the same in both cases, so we only treat the case of absolute weak normalization.
	For \enumref{thm:existence_of_absolute_weak_normalization_and_seminormalization}{4}, consider the commutative square 
	\begin{equation*}
		\begin{tikzcd}
			\Xawn \arrow[r, "f^{\awn}"] \arrow[d] & \Yawn \arrow[d] \\ 
			X \arrow[r, "f"'] & Y \period
		\end{tikzcd}
	\end{equation*}
	By \enumref{thm:existence_of_absolute_weak_normalization_and_seminormalization}{2}, the vertical morphisms are universal homeomorphisms, so by \enumref{thm:existence_of_absolute_weak_normalization_and_seminormalization}{3} we see that $ f $ is a universal homeomorphism if and only if the morphism $ f^{\awn} $ is an isomorphism.
	To conclude, note that \enumref{thm:existence_of_absolute_weak_normalization_and_seminormalization}{5} follows from items \enumref{thm:existence_of_absolute_weak_normalization_and_seminormalization}{2} and \enumref{thm:existence_of_absolute_weak_normalization_and_seminormalization}{4} combined with \Cref{ex:Bousfield_(co)localizations_are_Dwyer-Kan_localizations}. 
\end{proof}

\begin{definition}
	Given a scheme $ X $, we call $ \Xawn $ the \defn{absolute weak normalization} of $ X $ and call $ \Xsn $ the \defn{seminormalization} of $ X $.
\end{definition}

As an immediate consequence of \Cref{thm:absolutely_weakly_normal_schemes_over_fields,thm:existence_of_absolute_weak_normalization_and_seminormalization}, we deduce:

\begin{corollary}\label{cor:absolute_weak_normalization_over_fields}
	\hfill
	\begin{enumerate}[label=\stlabel{cor:absolute_weak_normalization_over_fields}, ref=\arabic*]
		\item\label{cor:absolute_weak_normalization_over_fields.1} For each $ \QQ $-scheme $ X $, we have $ \Xawn = \Xsn $.
		Moreover, the functor \smash{$ (-)^{\sn} \colon \fromto{\Sch_{\QQ}}{\Schsn_{\QQ}} $} identifies $ \Schsn_{\QQ} $ with the localization $ \Sch_{\QQ}[\UH^{-1}] $.
		
		\item\label{cor:absolute_weak_normalization_over_fields.2} For each $ \FFp $-scheme $ X $, we have $ \Xawn = \Xperf $.
		Moreover, the functor \smash{$ (-)^{\perf} \colon \fromto{\Sch_{\FFp}}{\Schperf_{\FFp}} $} identifies \smash{$ \Schperf_{\FFp} $} with the localization \smash{$ \Sch_{\FFp}[\UH^{-1}] $}.	
	\end{enumerate}
\end{corollary}

\begin{remark}
	For normal varieties, \enumref{cor:absolute_weak_normalization_over_fields}{2} was proven by Stix \cite[Lemma 4.1.2]{MR2012864}.
\end{remark}

\begin{notation}\label{ntn:absolute_weak_normalization_of_affine_schemes}
	Let $ X = \Spec(A) $ be an affine scheme.
	Since universal homeomorphisms are affine \stacks{04DE}, both $ \Xawn $ and $ \Xsn $ are affine.
	We write 
	\begin{equation*}
		A^{\awn} \colonequals \Gamma(\Xawn;\Ocal_{\Xawn}) \andeq A^{\sn} \colonequals \Gamma(\Xsn;\Ocal_{\Xsn}) \period
	\end{equation*}
	If $ A $ is an $ \FFp $-algebra, we write $ A^{\perf} \colonequals A^{\awn} $.
\end{notation}

For future use, we record the fact that absolute weak normalization preserves Zariski covers.

\begin{lemma}\label{lem:absolute_weak_normalization_of_open_immersions}
	Let $ j \colon \incto{U}{X} $ be an open immersion of schemes.
	Then the natural square
	\begin{equation*}
		\begin{tikzcd}
			U^{\awn} \arrow[r, "j^{\awn}"] \arrow[d] & \Xawn \arrow[d] \\ 
			U \arrow[r, hooked, "j"'] & X
		\end{tikzcd}
	\end{equation*}
	is cartesian.
	Hence, $ j^{\awn} $ is an open immersion and the functor $ (-)^{\awn} \colon \fromto{\Sch}{\Sch} $ preserves Zariski covers.
\end{lemma}

\begin{proof}
	Write $ c \colon \fromto{\Xawn}{X} $ for the counit.
	Since $ j $ is an open immersion,
	\begin{equation*}
		U \cross_X \Xawn \isomorphic c^{-1}(U) \comma
	\end{equation*}
	where $ c^{-1}(U) $ is given the unique subscheme structure.
	Since open subschemes of absolutely weakly normal schemes are again absolutely weakly normal (\Cref{lem:seminormality_and_absolute_weak_normality_are_affine_local}), $ U \cross_X \Xawn $ is absolutely weakly normal.
	Moreover, since $ c \colon \fromto{\Xawn}{X} $ is a universal homeomorphism, the natural morphism $ \fromto{U \cross_X \Xawn}{U} $ is also a universal homeomorphism.
	Since the natural morphism $ \fromto{U^{\awn}}{U} $ is a universal homeomorphism and universal homeomorphisms satisfy the $ 2 $-of-$ 3 $ property, the natural morphism
	\begin{equation*}
		\fromto{U^{\awn}}{U \cross_X \Xawn}
	\end{equation*} 
	is a universal homeomorphism.
	Since both $ U^{\awn} $ and $ U \cross_X \Xawn $ are absolutely weakly normal, \Cref{thm:existence_of_absolute_weak_normalization_and_seminormalization} shows that the natural morphism $ \fromto{U^{\awn}}{U \cross_X \Xawn} $ is an isomorphism.
\end{proof}


\subsection{Finiteness properties}\label{subsec:finiteness_properties}

We now analyze the interaction between seminormalization, perfection, and finite type hypotheses.
The first result is that, like normalization, seminormalization preserves the property of a scheme being (locally) of finite type over a field.

\begin{recollection}[(normalization)]
	Let $ X $ be a scheme such that every quasicompact open of $ X $ has finitely many irreducible components (e.g., $ X $ is locally topologically noetherian).
	We write $ \upnu \colon \fromto{X^{\norm}}{X} $ for the normalization of $ X $.
	Recall that the natural morphism $ \upnu \colon \fromto{X^{\norm}}{X} $ is an integral surjection.
	In particular, $ \upnu $ is affine; if $ X = \Spec(A) $ is affine, we write
	\begin{equation*}
		A^{\norm} \colonequals \Gamma(X^{\norm};\Ocal_{X^{\norm}}) \period
	\end{equation*}
	Also recall that the natural closed immersion $ \incto{X_{\red}}{X} $ induces an isomorphism \smash{$ \isomto{(X_{\red})^{\norm}}{X^{\norm}} $}.
\end{recollection}

\begin{nul}
	Notice that since normal schemes are seminormal (\Cref{prop:basic_properties_of_seminormal_rings}), the normalization morphism $ \upnu \colon \fromto{X^{\norm}}{X} $ factors through the seminormalization of $ X $.
	Since the seminormalization $ c \colon \fromto{\Xsn}{X} $ is a universal homeomorphism, $ c $ is affine \stacks{04DE}.
	Since $ \upnu $ is also affine, by cancellation we deduce that the induced morphism $ \fromto{X^{\norm}}{\Xsn} $ is affine.
\end{nul}

\begin{proposition}\label{prop:when_is_seminormalization_finite}
	Let $ X $ be a locally noetherian scheme.
	\begin{enumerate}[label=\stlabel{prop:when_is_seminormalization_finite}, ref=\arabic*]
		\item\label{prop:when_is_seminormalization_finite.1} If $ \upnu \colon \fromto{X^{\norm}}{X} $ is finite, then $ c \colon \fromto{\Xsn}{X} $ is finite.

		\item\label{prop:when_is_seminormalization_finite.2} If $ X $ is Nagata, then both $ \upnu \colon \fromto{X^{\norm}}{X} $ and $ c \colon \fromto{\Xsn}{X} $ are finite.

		\item\label{prop:when_is_seminormalization_finite.3} If $ X $ is (locally) of finite type over a field $ k $, then both $ \upnu \colon \fromto{X^{\norm}}{X} $ and $ c \colon \fromto{\Xsn}{X} $ are finite.
		Hence the seminormalization $ \Xsn $ is also (locally) of finite type over $ k $.
	\end{enumerate}
\end{proposition}

\begin{proof}
	For \enumref{prop:when_is_seminormalization_finite}{1}, first note that since $ X $ is locally noetherian, the natural closed immersion $ \incto{X_{\red}}{X} $ is finite.
	Since the natural morphisms
	\begin{equation*}
		\fromto{(X_{\red})^{\norm}}{X^{\norm}} \andeq \fromto{(X_{\red})^{\sn}}{\Xsn}
	\end{equation*}
	are isomorphisms, it suffices to prove the claim under the additional assumption that $ X $ is reduced.
	In this case, factor the normalization morphism as a composite
	\begin{equation*}
		\begin{tikzcd}
			\upnu \colon X^{\norm} \arrow[r] & \Xsn \arrow[r, "c"] & X \period
		\end{tikzcd}
	\end{equation*}
	Both of morphisms in the composite are affine, hence it suffices to prove the claim under the additional assumption that $ X = \Spec(A) $ is affine, noetherian, and reduced.

	Since $ A $ is reduced, the normalization homomorphism $ \fromto{A}{A^{\norm}} $ is injective.
	Thus \cite[Theorem 4.1]{MR595029} shows that the induced morphism $ \fromto{A^{\sn}}{A^{\norm}} $ is also injective.
	Since $ A^{\norm} $ is a finite $ A $-module and $ A $ is noetherian, we deduce that $ A^{\sn} $ is also a finite $ A $-module.
	Thus the seminormalization morphism $ \incto{A}{A^{\sn}} $ is finite, as desired.

	For \enumref{prop:when_is_seminormalization_finite}{2}, the finiteness of the normalization morphism is the content of \stacks{035S}.
	Since Nagata schemes are locally noetherian, the finiteness of the seminormalization morphism follows from \enumref{prop:when_is_seminormalization_finite}{1}.
	For \enumref{prop:when_is_seminormalization_finite}{3}, note that since $ X $ is (locally) of finite over a field, $ X $ is Nagata \stacks{035B}.
	Thus \enumref{prop:when_is_seminormalization_finite}{2} shows that the counit $ \fromto{\Xsn}{X} $ is a finite morphism.
	Since finite morphisms are of finite type, the composite $ \Xsn \to X \to \Spec(k) $ is of (locally) of finite type.
\end{proof}

Unfortunately, perfection (hence, also absolute weak normalization) does not preserve the property of being finite type.
In fact, a perfect $ \FFp $-scheme is basically never noetherian.
The following precise formulation of this fact is well-known:


\begin{proposition}\label{prop:the_only_perfect_noetherian_schemes_are_products_of_fields}
	Let $ X $ be a perfect $ \FFp $-scheme.
	Then $ X $ is noetherian if and only if $ X $ is isomorphic to the spectrum of a finite product of perfect fields.
\end{proposition}

\begin{proof}
	Clearly a finite product of perfect fields is noetherian.
	So it remains to show that if $ X $ is noetherian, then $ X $ is isomorphic to the spectrum of a finite product of perfect fields.
	See \cite{MO:470168} for a proof of the claim under the additional assumption that $ X $ is affine.
	If $ X $ is not assumed to be affine, note that
	\begin{equation*}
		\dim(X) = \sup_{\substack{U \subset\, X \\ \textup{affine open}}} \dim(U) \period
	\end{equation*}
	Since the spectrum of a product of fields is $ 0 $-dimensional, by the affine case we see that $ \dim(X) = 0 $.
	Since every $ 0 $-dimensional qcqs scheme is affine, $ X $ is also affine.
\end{proof}

\noindent As a result, for the positive characteristic version of our reconstruction results, we need to work with schemes that are only \textit{universally homeomorphic} to finite type schemes.

\begin{definition}\label{def:topologically_noetherian}
	We say that a scheme $ X $ is \defn{(locally) topologically noetherian} if the topological space $ |X| $ is (locally) noetherian.
\end{definition}

\begin{example}
	Let $ X $ be a (locally) noetherian scheme.
	Then the absolute weak normalization $ \Xawn $ is (locally) topologically noetherian.
\end{example}

\begin{definition}\label{def:topologically_of_finite_type}
	We say that a morphism of schemes $ f \colon \fromto{X}{S} $ is \defn{(locally) topologically of finite type} if there exists a universal homeomorphism of $ S $-schemes $ \fromto{X}{Y} $ such that $ Y $ is (locally) of finite type over $ S $.
\end{definition}

\begin{observation}\label{obs:morphism_locally_topologically_of_finite_type_are_stable_under_pullback}
	\hfill
	\begin{enumerate}[label=\stlabel{obs:morphism_locally_topologically_of_finite_type_are_stable_under_pullback}, ref=\arabic*]
		\item\label{obs:morphism_locally_topologically_of_finite_type_are_stable_under_pullback.1} Since universal homeomorphisms and morphisms (locally) of finite type are both stable under pullback, the class of morphisms (locally) topologically of finite type is stable under pullback.

		\item\label{obs:morphism_locally_topologically_of_finite_type_are_stable_under_pullback.2} Since universal homeomorphisms are closed under composition, if $ f \colon \fromto{X}{S} $ is (locally) topologically of finite type, then for any universal homeomorphism $ h \colon \fromto{X'}{X} $, the composite $ fh \colon \fromto{X'}{S} $ is (locally) topologically of finite type.
	\end{enumerate}
\end{observation}

The main example we care about is the following:

\begin{lemma}\label{lem:absolute_weak_normalization_topologically_of_finite_type}
	Let $ f \colon \fromto{X}{S} $ be a morphism of schemes.
	If $ f $ is (locally) topologically of finite type, then the induced morphisms $ \fromto{\Xawn}{S} $ and $ \fromto{\Xawn}{\Sawn} $ are (locally) topologically of finite type.
\end{lemma}

\begin{proof}
	The claim for the morphism $ \fromto{\Xawn}{S} $ follows from \enumref{obs:morphism_locally_topologically_of_finite_type_are_stable_under_pullback}{2} and the fact that the counit $ \fromto{\Xawn}{X} $ is a universal homeomorphism.
	For the second claim, note that $ f^{\awn} $ factors as
	\begin{equation*}
		\Xawn \to X \cross_S \Sawn \to \Sawn \period
	\end{equation*}
	By the $ 2 $-of-$ 3 $ property for universal homeomorphisms, the left-hand morphism is a universal homeomorphism.
	Since $ \fromto{X}{S} $ is (locally) topologically of finite type, \enumref{obs:morphism_locally_topologically_of_finite_type_are_stable_under_pullback}{1} shows that the right-hand morphism is (locally) topologically of finite type.
	Applying \enumref{obs:morphism_locally_topologically_of_finite_type_are_stable_under_pullback}{2} again shows that the composite $ \fromto{\Xawn}{\Sawn} $ is (locally) topologically of finite type.
\end{proof}

\begin{example}\label{ex:perfection_topologically_of_finite_type}
	Let $ k $ be a field of characteristic $ p > 0 $ and $ X $ a scheme (locally) topologically of finite type over $ k $.
	Then $ \Xperf $ is (locally) topologically of finite type over both $ k $ and $ \kperf $.
\end{example}

For seminormal schemes over a field of characteristic $ 0 $, the properties `(locally) of finite type' and `(locally) topologically of finite type' coincide:

\begin{lemma}\label{lem:locally_of_finite_type_for_seminormal_schemes_over_a_field_of_characteristic_0}
	Let $ k $ be a field of characteristic $ 0 $ and let $ X $ be a $ k $-scheme.
	\begin{enumerate}[label=\stlabel{lem:locally_of_finite_type_for_seminormal_schemes_over_a_field_of_characteristic_0}, ref=\arabic*]
		\item\label{lem:locally_of_finite_type_for_seminormal_schemes_over_a_field_of_characteristic_0.1} If $ X $ is (locally) topologically of finite type over $ k $, then $ \Xsn $ is (locally) of finite type over $ k $.

		\item\label{lem:locally_of_finite_type_for_seminormal_schemes_over_a_field_of_characteristic_0.2} If $ X $ is seminormal, then $ X $ is (locally) topologically of finite type over $ k $ if and only if $ X $ is (locally) of finite type over $ k $.
	\end{enumerate}
\end{lemma}

\begin{proof}
	For \enumref{lem:locally_of_finite_type_for_seminormal_schemes_over_a_field_of_characteristic_0}{1}, factor the structure morphism $ \fromto{X}{\Spec(k)} $ as a composite of a universal homeomorphism $ \fromto{X}{Y} $ followed by a morphism $ \fromto{Y}{\Spec(k)} $ (locally) of finite type.
	Since $ \characteristic(k) = 0 $, the universal homeomorphism $ \fromto{X}{Y} $ also induces an isomorphism on residue fields.
	By \enumref{thm:existence_of_absolute_weak_normalization_and_seminormalization}{4}, the induced morphism $ \fromto{\Xsn}{\Ysn} $ is an isomorphism.
	To conclude, note that since $ Y $ is (locally) of finite type over $ k $, by \enumref{prop:when_is_seminormalization_finite}{3}, the seminormalization $ \Ysn \isomorphic \Xsn $ is also locally of finite type over $ k $. 

	For the nontrivial implication of \enumref{lem:locally_of_finite_type_for_seminormal_schemes_over_a_field_of_characteristic_0}{2}, assume that $ X $ is (locally) topologically of finite type over $ k $.
	Then since $ X $ is seminormal, by \enumref{lem:locally_of_finite_type_for_seminormal_schemes_over_a_field_of_characteristic_0}{1} we conclude that $ X = \Xsn $ is (locally) of finite type over $ k $.
\end{proof}


We conclude this section by showing that the localization of the category of finite type $ k $-schemes at the universal homeomorphisms is the category of $ k $-schemes that are absolutely weakly normal and topologically of finite type over $ k $.

\begin{notation}\label{ntn:Schawntft}
	Let $ S $ be a scheme.
	We write \smash{$ \Schft_S \subset \Sch_S $} for the full subcategory spanned by the finite type $ S $-schemes.
	We write \smash{$ \Schawntft_S \subset \Sch_S $} for the full subcategory spanned by the $ S $-schemes that are absolutely weakly normal and topologically of finite type over $ S $.
\end{notation}

\begin{observation}
	The full subcategory $ \Schawntft_S \subset \Sch_S $ is the smallest full subcategory closed under isomorphisms and containing the absolute weak normalizations of finite type $ S $-schemes.
\end{observation}

\begin{lemma}\label{lem:awn_is_localization_for_finite_type}
	Let $ k $ be a field.
	Then the functor $ (-)^{\awn} \colon \Schft_k \to \Schawn_{k} $ induces an equivalence of categories
	\begin{equation*}
		\equivto{\Schft_k[\UH^{-1}]}{\Schawntft_k} \period
	\end{equation*}
\end{lemma}

\begin{proof}
	If $ \characteristic(k) = 0 $, then the claim follows from \Cref{thm:existence_of_absolute_weak_normalization_and_seminormalization} and \Cref{prop:when_is_seminormalization_finite}.

	If $ \characteristic(k) > 0 $, note that if $ f \colon X \to Y $ is a universal homeomorphism of finite type $ k $-schemes, it follows from \Cref{cor:absolute_weak_normalization_over_fields} that $ f^{\perf} \colon \Xperf \to \Yperf $ is an isomorphism of $ k $-schemes.
	Thus there is an inverse map $ g^{\perf} \colon \Yperf \to \Xperf $ over $ k $.
	For $ n \geq 0 $, write $ Y_n $ for the $ k $-scheme defined by
	\begin{equation*}
		\begin{tikzcd}[sep=3em]
			Y \arrow[r, "\Frob_Y^n"] & Y \arrow[r] & \Spec(k) \period
		\end{tikzcd}
	\end{equation*}
	Since $ X $ and $ Y $ are of finite type over $ k $, there is some $ n \gg 0 $ such that composite $ \Yperf \to \Xperf \to X $ factors as
	\begin{equation*}
		\Yperf \to Y_n \to X
	\end{equation*}
	and the composite $ Y_n \to X \xrightarrow{f} Y $ is given by \smash{$ \Frob_Y $}.
	Since \smash{$ \Frob_Y $} is a universal homeomorphism, it follows that the localization \smash{$ \Schft_k[\UH^{-1}] $} is equivalent to the localization \smash{$ \Schft_k[\Frob^{-1}] $} at absolute Frobenii.
	
	It remains to show that \smash{$ \Schawntft_k $} is the localization of \smash{$ \Schft_k $} at at absolute Frobenii.
	For this note that we may also describe the category \smash{$ \Schft_k[\Frob^{-1}] $} as follows.
	The objects are finite type $ k $-schemes and for two objects $ X $ and $ Y $, the set of morphisms $ \fromto{X}{Y} $ is given by
	\begin{equation*}
		\colim \Big\lparen\!
		\begin{tikzcd}[sep=3em]
			\Hom_{k}(X,Y) \arrow[r, "\Frob^*"] & \Hom_{k}(X_1,Y) \arrow[r, "\Frob^*"] & \Hom_{k}(X_2,Y) \arrow[r, "\Frob^*"] & \cdots
		\end{tikzcd}
		\!\Big\rparen \period
	\end{equation*}
	With this description it is easy to see that $ \Schawntft_k $ has the correct universal property.
\end{proof}


\section{The étale reconstruction property}\label{sec:the_etale_reconstruction_property}

The purpose of this section is to formulate what it means for a field $ k $ to have the property that absolutely weakly normal schemes topologically of finite type over $ k $ can be reconstructed from their étale topoi.
In order to do this, in \cref{subsec:notation_and_recollections_on_topoi}, we begin by setting some topos-theoretic notations and recalling a bit of background from topos theory.
Schemes topologically of finite type over a field are necessarily \textit{Jacobson}.
In \cref{subsec:background_on_Jacobson_schemes}, we recall the necessary background on Jacobson schemes. 
A particularly important point is that every morphism between schemes topologically of finite type over a field carries closed points to closed points.
However, not every geometric morphism between the étale topoi of such schemes has this property.

In \cref{subsec:pinned_morphisms_and_the_etale_reconstruction_property}, we study the class of geometric morphisms that send closed points to closed points, which we called \textit{pinned} geometric morphisms.
We show that for schemes $ X $ and $ Y $ topologically of finite type over $ k $, the groupoid of pinned geometric morphisms $ \fromto{\Xet}{\Yet} $ over $ \Spec(k)_{\et} $ is equivalent to a set.
This rigidity property is fundamental; we need it to make descent arguments in \cref{sec:reduction_to_regular_source_and_target_GGm}.
We also use this rigidity property to formulate what it means for a field to \textit{satisfy étale reconstruction}, see \Cref{prop:equivalent_conditions_for_etale_reconstruction,def:etale_reconstruction}.
In \cref{subsec:consequences_of_etale_reconstruction}, we record some immediate consequences of the étale reconstruction property (see \Cref{prop:consequences_of_etale_reconstruction}).


\subsection{Notation and recollections on topoi}\label{subsec:notation_and_recollections_on_topoi}

We make use of the following notation throughout the rest of the paper.

\begin{notation}\label{ntn:topoi}
	We write $ \RTop $ for the $ (2,1) $-category with objects topoi, $ 1 $-morphisms (right adjoints in) geometric morphisms, and $ 2 $-morphisms natural isomorphisms.


\end{notation}

\noindent We also have a relative notion of the $ (2,1) $-category of topoi sliced over a particular topos.

\begin{notation}\label{def:B-topoi}
	Let $ \Bcal $ be a topos. 
	We write $ \RTop_{\Bcal} $ for the $ (2,1) $-category of topoi over $ \Bcal $ defined in the following manner:
	\begin{enumerate}[label=\stlabel{def:B-topoi}, ref=\arabic*]
		\item An object of $\RTop_{\Bcal}$ is a geometric morphism of topoi $ \plowerstar \colon \fromto{\Xcal}{\Bcal} $.

		\item A $ 1 $-morphism from $ \plowerstar \colon \fromto{\Xcal}{\Bcal} $ to $ \qlowerstar \colon \fromto{\Ycal}{\Bcal} $ is a pair $ (\flowerstar,\alpha)$ where $ \flowerstar \colon \fromto{\Xcal}{\Ycal} $ is a geometric morphism and $\alpha$ is a natural isomorphism $ \alpha \colon \isomto{\qlowerstar\flowerstar}{\plowerstar} $.

		\item A $2$-morphism from $ (\flowerstar,\alpha) \colon \fromto{\Xcal}{\Ycal} $ to $ (\glowerstar,\beta) \colon \fromto{\Xcal}{\Ycal} $ is a natural isomorphism $ \gamma \colon \isomto{\flowerstar}{\glowerstar} $ such that the triangle of natural isomorphisms
		\begin{equation*} 
			\begin{tikzcd}
			  & \qlowerstar\glowerstar \arrow[dr," \beta"] \\
			 \qlowerstar\flowerstar \arrow[rr, "\alpha"'] \arrow[ur,"\qlowerstar \gamma"] & & \plowerstar
			\end{tikzcd} 
		\end{equation*}
		commutes.
	\end{enumerate}
\end{notation}

\begin{notation}\label{ntn:HomB}
	Let $ \Bcal $ be a topos.
	Given geometric morphisms $ \plowerstar \colon \fromto{\Xcal}{\Bcal} $ and $ \qlowerstar \colon \fromto{\Ycal}{\Bcal} $, we write $ \Hom_{\Bcal}(\Xcal,\Ycal) $ for the $ \Hom $ \textit{groupoid} in \smash{$ \RTop_{\Bcal} $}.
	If $ \Bcal = S_{\et} $ is the étale topos of a scheme $ S $, we simply write
	\begin{equation*}
		\RTop_{S} \colonequals \RTop_{S_{\et}} \andeq \Hom_{S}(\Xcal,\Ycal) \colonequals \Hom_{S_{\et}}(\Xcal,\Ycal) \period
	\end{equation*}
	Moreover, if $ S = \Spec(A) $ is affine, we write
	\begin{equation*}
		\RTop_{A} \colonequals \RTop_{\Spec(A)_{\et}} \andeq \Hom_{A}(\Xcal,\Ycal) \colonequals \Hom_{\Spec(A)_{\et}}(\Xcal,\Ycal) \period
	\end{equation*}
\end{notation}

\begin{recollection}[(the underlying space of a topos)]
	Let $ \Xcal $ be a topos.
	We write $ \Open(\Xcal) $ for the poset of isomorphism classes of subobjects of the terminal object of $ \Xcal $.
	The poset $ \Open(\Xcal) $ is a locale.
	We write $ |\Xcal| $ for the topological space associated to the locale $ \Open(\Xcal) $, i.e., the topological space of points of $ \Open(\Xcal) $.

	If $ \flowerstar \colon \Xcal \to \Ycal $ is a geometric morphism, the pullback functor $ \fupperstar $ preserves subobjects of the terminal object, thus restricts to a morphism of locales
	\begin{equation*}
		\fupperstar \colon \fromto{\Open(\Xcal)}{\Open(\Ycal)} \period
	\end{equation*}
	We write $ |\flowerstar| \colon \fromto{|\Xcal|}{|\Ycal|} $ for the induced map of underlying spaces.
\end{recollection}

\begin{example}
	Let $ X $ be a scheme.
	Then there is a natural isomorphism from  the locale of open subsets of the underlying space $ |X| $ to the locale $ \Open(\Xet) $.
	Since $ |X| $ is sober, we deduce that there is a natural isomorphism of topological spaces $ \isomto{|\Xet|}{|X|} $.
\end{example}

We also recall some background about categories of points of topoi.

\begin{notation}[(category of points)]\label{ntn:category_of_points}
	Let $ \Xcal $ be a topos.
	We write $ \Pt(\Xcal) $ for the \defn{category of points} of $ \Xcal $.
	Objects of $ \Pt(\Xcal) $ are left exact left adjoints $ \xupperstar \colon \fromto{\Xcal}{\Set} $ and morphisms are natural transformations.
	The assignment $ \goesto{\Xcal}{\Pt(\Xcal)} $ defines a functor
	\begin{equation*}
		\Pt \colon \fromto{\RTop}{\Cat}
	\end{equation*}
	from the $ (2,1) $-category of topoi to the $ (2,1) $-category of categories.
	The functoriality is given by sending a geometric morphism $ \flowerstar \colon \fromto{\Xcal}{\Ycal} $ to the functor
	\begin{equation*}
		\Pt(\flowerstar) \colon \Pt(\Xcal) \to \Pt(\Ycal) \comma \qquad \xupperstar \mapsto \xupperstar \fupperstar 
	\end{equation*}
	with the obvious functoriality in natural transformations.
\end{notation}

\begin{example}[{\cite[Exposé VIII, Théorème 7.9]{MR50:7131}}]\label{ex:Pt_of_Xet}
	Let $ X $ be a qcqs scheme.
	The Grothendieck School  computed the category of points of $ \Xet $.
	The objects of $ \Pt(\Xet) $ are geometric points $ \fromto{\xbar}{X} $.
	Given geometric points $ \fromto{\xbar}{X} $ and $ \fromto{\ybar}{X} $, the set of morphisms $ \fromto{\xbar}{\ybar} $ is the set
	\begin{equation*}
		\Hom_{\Pt(\Xet)}(\xbar,\ybar) = \Hom_{X}(\Spec(\Ocal_{X,y}^{\sh}),\Spec(\Ocal_{X,x}^{\sh}))
	\end{equation*}
	of morphisms of $ X $-schemes from the strictly henselian local scheme associated to $ \ybar $ to that of $ \xbar $. 
	Composition in $ \Pt(\Xet) $ is given by composition of morphisms of $ X $-schemes.
\end{example}

\begin{nul}
	An important consequences of \Cref{ex:Pt_of_Xet} are that for each geometric point $ \fromto{\xbar}{X} $ with image $ x \in X $, the set of endomorphisms of $ \xbar $ in $ \Pt(\Xet) $ is isomorphic to the absolute Galois group $ \Gal(\upkappa(\xbar)/\upkappa(x)) $.
	In particular, every endomorphism of $ \xbar $ in $ \Pt(\Xet) $ is an automorphism.
\end{nul}

A few times throughout this paper, we need the following observation.
In general, the functor
\begin{equation*}
	(-)_{\et} \colon \Sch \to \RTop
\end{equation*}
does not preserve fiber products (see \cite[Remark 2.6]{MR4027830}).
However, it does preserve pullbacks along points:

\begin{lemma}\label{rec:taking_topoi_preserves_fiber_products}
	Let $ f \colon \fromto{X}{S} $ be a morphism between qcqs schemes and $ s \in S $.
	Then the square of étale topoi 
	\begin{equation*}
      \begin{tikzcd}[sep=2.25em]
	       X_{s,\et} \arrow[d] \arrow[r] & X_{\et} \arrow[d, "\flowerstar"] \\ 
	       \Spec(\upkappa(s))_{\et} \arrow[r] & S_{\et}
      \end{tikzcd}
    \end{equation*}
    is a pullback square in $ \RTop $.
\end{lemma}

\begin{proof}
	The same argument as in \cite[Proposition 2.3]{zbMATH07795380}, replacing the strictly henselian local ring by the local ring $ \Ocal_{S,s} $ and using that $ \Spec(\Ocal_{S,s}) $ is isomorphic to the limit of the cofitered diagram of affine open subschemes of $ S $ containing $ s $.
\end{proof}


\subsection{Background on Jacobson schemes}\label{subsec:background_on_Jacobson_schemes}

We now recall the basics about Jacobson schemes.
Being Jacobson is a topological property of schemes.
The idea is that `the closed points see everything about the topology of a Jacobson space'.

\begin{notation}
	Given a topological space $ T $, we write $ T_{\cl} \subset T $ for the subset of closed points.
	For a scheme $ X $, we write $ \Xcl \colonequals |X|_{\cl} $.
\end{notation}

\begin{recollection}
	A topological space $ T $ is \defn{Jacobson} if every closed subset $ Z \subset T $ is the closure of $ Z \intersect T_{\cl} $.
	We say that a scheme $ X $ is \defn{Jacobson} if the underlying topological space $ |X| $ is Jacobson.
	See \cite[Remark 3.34]{MR4225278} for a number of equivalent characterizations of Jacobson spaces.
\end{recollection}

\begin{lemma}\label{lem:schemes_locally_topologically_of_finite_type_over_Jacobson_schemes_are_Jacobson}
	If $ f \colon \fromto{X}{S} $ is a morphism of schemes that is locally topologically of finite type and $ S $ is Jacobson, then $ X $ is also Jacobson.
\end{lemma}

\begin{proof}
	Factor $ f $ as a composite of a universal homeomorphsm $ \fromto{X}{X'} $ followed by a morphism $ \fromto{X'}{S} $ locally of finite type.
	Since $ S $ is Jacobson, by \stacks{02J5}, the scheme $ X' $ is Jacobson.
	Since $ |X| $ is homeomorphic to $ |X'| $, the scheme $ X $ is also Jacobson.
\end{proof}

\begin{example}\label{ex:schemes_locally_topologically_of_finite_type_over_fields_are_Jacobson}
	If $ X $ is a scheme locally topologically of finite type over a field, then $ X $ is Jacobson.
\end{example}

As with many topological properties of schemes, there are more algebro-geometric characterizations of Jacobson schemes.
To state them, we need the following definition.

\begin{definition}
	Let $ X $ be a scheme.
	We say that $ x \in X $ is a \defn{finite type point} if the natural morphism $ \fromto{\Spec(\upkappa(x))}{X} $ is of finite type.
	We write $ \Xftpts \subset |X| $ for the subset of finite type points.
\end{definition}

\begin{proposition}[{\stacks{01TB}}]\label{prop:characterization_of_Jacobson_schemes}
	Let $ X $ be a scheme.
	The following are equivalent:
	\begin{enumerate}[label=\stlabel{prop:characterization_of_Jacobson_schemes}, ref=\arabic*]
		\item The scheme $ X $ is Jacobson.

		\item The set $ \Xftpts $ of finite type points of $ X $ is the set $ \Xcl $ of closed points of $ X $.

		\item Every morphism $ \fromto{X'}{X} $ locally of finite type carries closed points of $ X' $ to closed points of $ X $.

		\item For every morphism $ f \colon \fromto{X'}{X} $ locally of finite type and closed point $ x' \in X' $, the point $ f(x') $ is closed and the field extension $ \upkappa(x') \supset \upkappa(f(x')) $ is finite.
	\end{enumerate}
\end{proposition}

To conclude this subsection, we prove that morphisms from reduced Jacobson schemes are determined by their values on closed points.
For this, we need the following lemma.

\begin{lemma}\label{lem:locally_closed_subspaces_of_Jacobson_spaces_containing_all_closed_points}
	Let $ T $ be a Jacobson topological space and let $ S \subset T $ be a locally closed subspace. 
	If $ S $ contains all of the closed points of $ T $, then $ S = T $.
\end{lemma}

\begin{proof}	
	The claim is immediate from the fact that the map
	\begin{align*}
		\set{\textup{locally closed subsets of } T}
		&\to
		\set{\textup{locally closed subsets of } T_{\cl}}
		\\
		W &\mapsto W \intersect T_{\cl}
	\end{align*}
	is bijective \stacks{005Z}.
\end{proof}

\begin{corollary}\label{cor:locally_closed_subschemes_of_Jacobson_schemes_containing_all_closed_points}
	Let $ X $ be a reduced Jacobson scheme and let $ Z \subset X $ be a locally closed subscheme. 
	If $ Z $ contains all of the closed points of $ X $, then $ Z = X $ as schemes.
\end{corollary}

\begin{proof}
	By \Cref{lem:locally_closed_subspaces_of_Jacobson_spaces_containing_all_closed_points}, we see that $ |Z| = |X| $ as topological spaces.
	Since $ X $ is reduced, we deduce that $ Z = X $ as schemes.
\end{proof}

\begin{lemma}\label{lemma:map_from_points_is_epi_for_Jacobson_schemes}
	Let $ X $ be a reduced Jacobson scheme, and for each closed point $ x \in \Xcl $ choose a field extension $ L_x \supset \upkappa(x) $.
	Then the induced morphism
	\begin{equation*}
		\fromto{\coprod_{x \in \Xcl} \Spec(L_x)}{X}
	\end{equation*}
	is an epimorphism in the category of schemes.
\end{lemma}

\begin{proof}
	Note that for each field extension $ L \supset k $, the induced morphism of schemes $ \fromto{\Spec(L)}{\Spec(k)} $ is an epimorphism.
	So it suffices to prove the claim when $ L_x = \upkappa(x) $ for each $ x \in \Xcl $.
	Let $f,g\colon X \to Y$ be morphisms that become equal after precomposition with the natural morphism
	\begin{equation*}
		c_{X} \colon \fromto{\coprod_{x \in \Xcl} \Spec(\upkappa(x))}{X} \period
	\end{equation*}
	Consider the equalizer subscheme $\eq(f,g) \subset X$, which is locally closed.
	By assumption, $ \eq(f,g) $ contains all of the closed points of $ X $.
	Since $ X $ is Jacobson and reduced, \Cref{cor:locally_closed_subschemes_of_Jacobson_schemes_containing_all_closed_points} shows that $ \eq(f,g) = X $ as schemes.
	Hence $ f = g $.
\end{proof}


\subsection{Pinned morphisms and the étale reconstruction property}\label{subsec:pinned_morphisms_and_the_etale_reconstruction_property}

Let $ k $ be a field, and let $ X $ and $ Y $ be schemes locally topologically of finite type over $ k $.
In light of \Cref{prop:characterization_of_Jacobson_schemes}, every morphism of $ k $-schemes $ \fromto{X}{Y} $ carries closed points to closed points.
Not every geometric morphism $ \fromto{\Xet}{\Yet} $ over $ \Spec(k)_{\et} $ has this property (see \Cref{ex:a_geometric_morphism_that_isn't_pinned}).
Hence, in our reconstruction results, we impose this restriction on the morphisms of topoi.

\begin{definition}[{(pinned morphisms)}]\label{def:pinned_geometric_morphism}
	\hfill
	\begin{enumerate}[label=\stlabel{def:pinned_geometric_morphism}, ref=\arabic*]
		\item\label{def:pinned_geometric_morphism.1} A map of topological spaces $ f \colon \fromto{S}{T} $ is \defn{pinned} if $ f $ carries closed points of $ S $ to closed points of $ T $.

		\item\label{def:pinned_geometric_morphism.2} A morphism of schemes $ f \colon \fromto{X}{Y} $ is \defn{pinned} if the induced map on underlying topological spaces $ |f| \colon \fromto{|X|}{|Y|} $ is pinned.

		\item\label{def:pinned_geometric_morphism.3} A geometric morphism of topoi $ \flowerstar \colon \fromto{\Xcal}{\Ycal} $ is \defn{pinned} if the induced map of underlying topological spaces $ |\flowerstar| \colon \fromto{|\Xcal|}{|\Ycal|} $ is pinned.
	\end{enumerate}
\end{definition}

\begin{notation}\label{ntn:RToppin}
	Let $ \Bcal $ be a topos.
	We write \smash{$ \RToppin_{\Bcal} \subset \RTop_{\Bcal} $} for the subcategory with all objects, $ 1 $-morphisms the \textit{pinned} geometric morphisms, and $ 2 $-mor\-phisms all natural isomorphisms.
	Given topoi $ \Xcal,\Ycal \in \RTop_{\Bcal} $, write
	\begin{equation*}
		\Hompin_{\Bcal}(\Xcal,\Ycal) \subset \Hom_{\Bcal}(\Xcal,\Ycal)
	\end{equation*}
	for the full subgroupoid spanned by the pinned geometric morphisms.
	When $ \Bcal $ is the étale topos of a scheme, we make the same notational simplifications as in \Cref{ntn:HomB}.
\end{notation}

\begin{remark}[(Voevodsky's terminology)]
	Voevodsky uses the term `admissible' rather than `pinned', see \cite[p. 513]{MR1098621}.
	We find the term `pinned' more evocative: the requirement that the geometric morphism carries closed points to closed points `pins down' the morphism.
\end{remark}

\begin{proposition}[(rigidity of pinned morphisms)]\label{prop:rigidity}
	Let $ S $ be a scheme and let $ X $ and $ Y $ be schemes locally topologically of finite type over $ S $. 
	\begin{enumerate}[label=\stlabel{prop:rigidity}, ref=\arabic*]
		\item\label{prop:rigidity.1} Let $ (\flowerstar,\alpha) \colon \Xet \to \Yet $ be a pinned geometric morphism over $ S_{\et} $. 
		Then $(\flowerstar,\alpha)$ has no nontrivial automorphisms in \smash{$ \Hom_{S}(\Xet,\Yet) $}.
		
		\item\label{prop:rigidity.2} The groupoid \smash{$ \Hompin_{S}(\Xet,\Yet) $} is equivalent to a set.
	\end{enumerate}
\end{proposition}

\begin{proof}
	Item \enumref{prop:rigidity}{2} is simply a reformulation of \enumref{prop:rigidity}{1}, so we just need to prove \enumref{prop:rigidity}{1}.
	
	We first prove the claim under the additional assumption that $ S $ is the spectrum of a field $ k $.
	Note that, by descent, we can assume that $ k $ is separably closed. 
	Let $\phi \colon \fupperstar \equivalence \fupperstar$ be an automorphism. 
	The claim follows if we can show that for every \emph{closed} geometric point $x \to \Xet $ the composite $\xupperstar \phi \colon \xupperstar \fupperstar \equivalence \xupperstar \fupperstar$ is the identity. 
	This is clear because a map in $ \Xet $ is the identity if and only if it is the identity on stalks at closed points.
	Since $ \flowerstar $ is pinned, the composite $ \flowerstar \circ \xlowerstar $ also defines a closed geometric point of $ \Yet $.
	By \Cref{ex:Pt_of_Xet}, we may assume that $ \flowerstar \circ \xlowerstar $ is induced by a morphism of $ k $-schemes $ y \colon \Spec(\kbar) \to Y $ and $ \xupperstar \phi  $ defines an automorphism of $ y^* $.
	Let $ y_0 \in Y$ denote the image of $ y $.
	Again by \Cref{ex:Pt_of_Xet}, the group of automorphisms of $ y^* $ is isomorphic to the absolute Galois group $ \Gup_{\upkappa(y_0)} $ of $ \upkappa(y_0) $.
	Since $ y_0 $ is a closed point and $ k $ is separably closed it follows that $ \Gup_{\upkappa(y_0)} =\{ \id{} \} $ and thus $ \xupperstar \phi = \id{} $.
	
	For the case of a general base scheme $ S $, fix an automorphism $ \phi \colon \fupperstar \to \fupperstar  $ in \smash{$ \Hompin_S(\Xet,\Yet) $}.
	Using \cref{rec:taking_topoi_preserves_fiber_products}, we see that pulling back $ \phi $ along any point $s \colon \Spec(\upkappa(s)) \to S$ induces an automorphism $ s^* \phi $ of $ s^* \fupperstar $ in the groupoid $ \Hompin_{\upkappa(s)}(X_{s,\et},Y_{s,\et}) $, which we have already shown to be equivalent to a set.
	In particular $ s^*\phi = \id{} $.
	This in particular implies that $ \phi $ is the identity stalk-wise, so $ \phi = \id{} $.
\end{proof}

\begin{convention}\label{convention:regarding_Hompink_as_a_set}
	Let $ k $ be a field and let $ X $ and $ Y $ be schemes locally topologically of finite type over $ k $.
	Then using \Cref{prop:rigidity}, we tacitly regard the groupoid \smash{$ \Hompink(\Xet,\Yet) $} as a set, i.e., we identify \smash{$ \Hompink(\Xet,\Yet) $} with the set of \textit{isomorphism classes} of geometric morphisms $ \fromto{\Xet}{\Yet} $ over $ \Spec(k)_{\et} $.
\end{convention}

We are almost ready to formulate what it means for a field $ k $ to have the property that absolutely weakly normal schemes topologically of finite type over $ k $ can be reconstructed from their étale topoi.
First we show that every morphism between schemes \textit{topologically} of finite type is a pinned.

\begin{lemma}\label{lem:every_morphism_between_schemes_locally_topologically_of_finite_type_is_pinned}
	Let $ S $ be a Jacobson scheme and let $ X $ and $ Y $ be schemes locally topologically of finite type over $ S $.
	Then every morphism of $ S $-schemes $ f \colon \fromto{X}{Y} $ is pinned.
\end{lemma}

\begin{proof}
	We work locally on the target to assume that $ X $ and $ S $ are affine.
	Let $ s \colon X \to X' $ and $ t \colon Y \to Y' $ be universal homeomorphisms to finite type $ S $-schemes.
	Then $ tf $ factors as
	\begin{equation*}
		\begin{tikzcd}
			X \arrow[r, "\Gamma_f"] & X \times_{S} Y \arrow[r, "s \times t"] & X' \times_{S} Y' \arrow[r, "\pr_2"] & Y' \period
		\end{tikzcd}
	\end{equation*}
	Since $ \fromto{X}{S} $ is separated, the graph $ \Gamma_f $ is a closed immersion; in particular, $ \Gamma_f $ is pinned.
	Since $ s $ and $ t $ are universal homeomorphisms, $ s \times t $ is a universal homeomorphism; in particular, $ s \cross t $ is pinned.
	Since $ \pr_2 $ is a morphism of schemes locally of finite type over $ S $, \Cref{lem:schemes_locally_topologically_of_finite_type_over_Jacobson_schemes_are_Jacobson,prop:characterization_of_Jacobson_schemes} show that $ \pr_2 $ is pinned.
	It follows that $ tf $ is pinned.
	Since $ t $ is a universal homeomorphism, this also implies that $ f $ is pinned. 
\end{proof}

\noindent There are a number of equivalent formulations of what it means for a field $ k $ to have the property that absolutely weakly normal schemes topologically of finite type over $ k $ can be reconstructed from their étale topoi.
We prove their equivalence before stating the definition.

\begin{observation}\label{obs:absolute_weak_normalization_on_slices}
	Let $ S $ be a scheme. 
	Absolute weak normalization induces an adjunction
	\begin{equation*}
		\begin{tikzcd}[sep=4em]
			(\Schawn)_{/\Sawn} \arrow[r, shift left] & \Sch_{/S} \arrow[l, shift left, "(-)^{\awn}"] \period
		\end{tikzcd}
	\end{equation*}
	The left adjoint sends an absolutely weakly normal scheme $ X $ over $ \Sawn $ to the composite
	\begin{equation*}
		X \to \Sawn \to S 
	\end{equation*}
	of the structure morphism with the counit.
	In particular, if $ X $ is an absolutely weakly normal $ S $-scheme and $ Y $ is an $ S $-scheme, there is a natural bijection
	\begin{equation*}
		\Hom_{S}(X,Y) \isomorphic \Hom_{\Sawn}(X,\Yawn) \period
	\end{equation*} 
\end{observation}

\begin{lemma}\label{lem:preliminaries_to_equivalent_considtions_for_reconstruction}
	Let $ k $ be a field and let $ X $ and $ Y $ be schemes locally topologically of finite type over $ k $.
	The following are equivalent:
	\begin{enumerate}[label=\stlabel{lem:preliminaries_to_equivalent_considtions_for_reconstruction}, ref=\arabic*]
		\item The natural map $ \Hom_{k}(\Xawn,Y) \to \Hompink(\Xetawn,\Yet) $ is bijective.

		\item The natural map $ \Hom_{\kawn}(\Xawn,\Yawn) \to \Hompin_{\kawn}(\Xetawn,\Yetawn) $ is bijective.
	\end{enumerate}
\end{lemma}

\begin{proof}
	Let us first comment that since $ X $ and $ Y $ are locally topologically of finite type over $ k $, \Cref{lem:absolute_weak_normalization_topologically_of_finite_type} shows that both $ \Xawn $ and $ \Yawn $ are locally topologically of finite type over $ k $.
	Hence by \Cref{prop:rigidity}, the right-hand sides are sets, so the statements actually make sense.

	To prove the claim, notice that the counits $ \fromto{\Yawn}{Y} $ and $ \fromto{\Spec(\kawn)}{\Spec(k)} $ induce a commutative square
	\begin{equation}\label{sq:comparing_reconstruction_maps}
		\begin{tikzcd}[column sep=3em]
			\Hom_{\kawn}(\Xawn,\Yawn) \arrow[r] \arrow[d] & \Hompin_{\kawn}(\Xetawn,\Yetawn) \arrow[d] \\ 
			\Hom_{k}(\Xawn,Y) \arrow[r] & \Hompink(\Xetawn,\Yet) \period
		\end{tikzcd}
	\end{equation}
	By \Cref{obs:absolute_weak_normalization_on_slices}, the left-hand vertical map is bijective.
	By the topological invariance of the étale topos, for any scheme $ S $, the counit $ \fromto{\Sawn}{S} $ induces an equivalence of étale topoi $ \equivto{S_{\et}^{\awn}}{S_{\et}} $.
	Hence the right-hand vertical map is bijective.
	Thus the top map in \eqref{sq:comparing_reconstruction_maps} is bijective if and only if the bottom map is bijective, as claimed.
\end{proof} 

For the next result, recall \Cref{ntn:Schawntft,ntn:RToppin}.

\begin{proposition}\label{prop:equivalent_conditions_for_etale_reconstruction}
	Let $ k $ be a field.
	The following conditions are equivalent:
	\begin{enumerate}[label=\stlabel{prop:equivalent_conditions_for_etale_reconstruction}, ref=\arabic*]
		\item\label{prop:equivalent_conditions_for_etale_reconstruction.-1} The functor $ (-)_{\et} \colon \Schawntft_{k} \to \RToppin_{k} $ is fully faithful.

		\item\label{prop:equivalent_conditions_for_etale_reconstruction.0} The functor $ (-)_{\et} \colon \Schft_{k}[\UH^{-1}] \to \RToppin_{k} $ is fully faithful.

		\item\label{prop:equivalent_conditions_for_etale_reconstruction.1} For all finite type $ k $-schemes $ X $ and $ Y $, the natural map 
		\begin{equation*}
			\Hom_{k}(\Xawn,Y) \to \Hompink(\Xetawn,\Yet)
		\end{equation*}
		is bijective.

		\item\label{prop:equivalent_conditions_for_etale_reconstruction.2} For all schemes $ X $ and $ Y $ topologically of finite type over $ k $, the natural map 
		\begin{equation*}
			\Hom_{k}(\Xawn,Y) \to \Hompink(\Xetawn,\Yet)
		\end{equation*}
		is bijective.

		\item\label{prop:equivalent_conditions_for_etale_reconstruction.3} For all schemes $ X $ and $ Y $ topologically of finite type over $ k $ with $ X $ absolutely weakly normal, the natural map
		\begin{equation*}
			\Hom_{k}(X,Y) \to \Hompink(\Xet,\Yet)
		\end{equation*}
		is bijective.

		\item\label{prop:equivalent_conditions_for_etale_reconstruction.4} For all schemes $ X $ and $ Y $ topologically of finite type over $ k $, the natural map
		\begin{equation*}
			\Hom_{\kawn}(\Xawn,\Yawn) \to \Hompin_{\kawn}(\Xetawn,\Yetawn)
		\end{equation*}
		is bijective.
	\end{enumerate}
\end{proposition}

\begin{proof}
	It is an immediate consequence of \Cref{lem:awn_is_localization_for_finite_type} that \enumref{prop:equivalent_conditions_for_etale_reconstruction}{-1}, \enumref{prop:equivalent_conditions_for_etale_reconstruction}{0}, and \enumref{prop:equivalent_conditions_for_etale_reconstruction}{1} are equivalent.
	Clearly \enumref{prop:equivalent_conditions_for_etale_reconstruction}{2} $ \Rightarrow $ \enumref{prop:equivalent_conditions_for_etale_reconstruction}{1}.
	To see that \enumref{prop:equivalent_conditions_for_etale_reconstruction}{1} $ \Rightarrow $ \enumref{prop:equivalent_conditions_for_etale_reconstruction}{2}, let $ X $ and $ Y $ be schemes topologically of finite type over $ k $.
	Then there exist universal homeomorphisms $ \fromto{X}{X'} $ and $ \fromto{Y}{Y'} $ where $ X' $ and $ Y' $ are of finite type over $ k $.
	By assumption \enumref{prop:equivalent_conditions_for_etale_reconstruction}{1}, the natural map 
	\begin{equation*}
		\Hom_{k}((X')^{\awn},Y') \to \Hompink((X')_{\et}^{\awn},\Yet')
	\end{equation*}
	is bijective.
	By \Cref{lem:preliminaries_to_equivalent_considtions_for_reconstruction}, this is equivalent to the claim that the natural map
	\begin{equation*}
		\Hom_{\kawn}((X')^{\awn},(Y')^{\awn}) \to \Hompin_{\kawn}((X')_{\et}^{\awn},(Y')_{\et}^{\awn})
	\end{equation*}
	is bijective.
	Since the universal homeomorphisms $ \fromto{X}{X'} $ and $ \fromto{Y}{Y'} $ induce isomorphisms on absolute weak normalizations, we deduce that the natural map 
	\begin{equation*}
		\Hom_{\kawn}(\Xawn,\Yawn) \to \Hompin_{\kawn}(\Xetawn,\Yetawn)
	\end{equation*}
	is bijective.
	Applying \Cref{lem:preliminaries_to_equivalent_considtions_for_reconstruction} again, we deduce that the natural map 
	\begin{equation*}
		\Hom_{k}(\Xawn,Y) \to \Hompink(\Xetawn,\Yet)
	\end{equation*}
	is bijective, as desired.
	
	The equivalence \enumref{prop:equivalent_conditions_for_etale_reconstruction}{2} $ \Leftrightarrow $ \enumref{prop:equivalent_conditions_for_etale_reconstruction}{3} is immediate from the fact that the absolute weak normalization of a scheme topologically of finite type over $ k $ is also topologically of finite type over $ k $ (\Cref{lem:absolute_weak_normalization_topologically_of_finite_type}).
	The equivalence \enumref{prop:equivalent_conditions_for_etale_reconstruction}{2} $ \Leftrightarrow $ \enumref{prop:equivalent_conditions_for_etale_reconstruction}{4} is immediate from \Cref{lem:preliminaries_to_equivalent_considtions_for_reconstruction}.
\end{proof}

The following is the main definition of this section.

\begin{definition}[(étale reconstruction)]\label{def:etale_reconstruction}
	Let $ k $ be a field.
	We say that $ k $ \defn{satisfies étale reconstruction} if the equivalent conditions of \Cref{prop:equivalent_conditions_for_etale_reconstruction} are satisfied.
\end{definition}

%

\begin{example}\label{ex:a_geometric_morphism_that_isn't_pinned}
	The restriction to pinned morphisms in the étale reconstruction property is necessary.
	To see this, note that for a field of characteristic $ 0 $ the map of Galois groups $ \Gup_{k(t)} \to \Gup_k $ admits a continuous section.
	This section induces a morphism $ \fromto{\Spec(k)_{\et}}{\Spec(k(t))_{\et}} $ in $ \RTop_{k} $.
	Composing this with the geometric morphism induced by the morphism \smash{$ \Spec(k(t)) \to \PP^1_k $} picking out the generic point, we get a morphism
	\begin{equation*}
		\Spec(k)_{\et} \to \PP^1_{k,\et}
	\end{equation*}
	in $ \RTop_{k} $ whose underlying map on topological spaces hits the generic point.
	Thus this geometric morphism is not induced be a morphism of schemes \smash{$ \Spec(k) \to \PP^1_k $}.
\end{example}

\begin{example}
	If $ k $ is an algebraically closed field of characteristic $ 0 $, then $ k $ does \textit{not} satisfy étale reconstruction.
	If $ C $ and $ C' $ are smooth proper curves over $ k $ of the same genus, then the étale topoi of $ C $ and $ C' $ are equivalent (over the terminal topos $ \Spec(k)_{\et} \equivalent \Set $).
	See \cite{MO:38466}.
\end{example}

In characteristic $ 0 $, étale reconstruction has another reformulation.

\begin{corollary}\label{cor:etale_reconstruction_for_characteristic_0_fields}
	Let $ k $ be a field of characteristic $ 0 $.
	Then the following are equivalent:	
	\begin{enumerate}[label=\stlabel{cor:etale_reconstruction_for_characteristic_0_fields}, ref=\arabic*]
		\item The field $ k $ satisfies étale reconstruction.

		\item For all schemes finite type $ k $-schemes $ X $ and $ Y $ with $ X $ seminormal, the natural map
		\begin{equation*}
			\fromto{\Hom_k(X,Y)}{\Hompink(\Xet,\Yet)} 
		\end{equation*}
		is bijective.
	\end{enumerate}
\end{corollary}

\begin{proof}
	Immediate from \Cref{lem:locally_of_finite_type_for_seminormal_schemes_over_a_field_of_characteristic_0} and the fact that in characteristic $ 0 $, absolute weak normalization and seminormalization coincide (\Cref{cor:absolute_weak_normalization_over_fields}).
\end{proof}


\subsection{Consequences of étale reconstruction}\label{subsec:consequences_of_etale_reconstruction}

We conclude this section by recording some immediate consequences of the étale reconstruction property (see \Cref{prop:consequences_of_etale_reconstruction}).
First we fix some notation.

\begin{notation}\label{ntn:EquivB}
	Let $ \Bcal $ be a topos.
	Given $ \Xcal,\Ycal \in \RTop_{\Bcal} $, write
	\begin{equation*}
		\Equiv_{\Bcal}(\Xcal,\Ycal) \subset \Hom_{\Bcal}(\Xcal,\Ycal)
	\end{equation*}
	for the full subgroupoid spanned by the equivalences of topoi $ \equivto{\Xcal}{\Ycal} $ over $ \Bcal $.
	When $ \Bcal $ is the étale topos of a scheme, we make the same notational simplifications as in \Cref{ntn:HomB}.
\end{notation}

\begin{convention}\label{convention:regarding_Equivk_as_a_set}
	Let $ k $ be a field and $ X $ and $ Y $ schemes topologically of finite type over $ k $.
	Then
	\begin{equation*}
		\Equiv_{k}(\Xet,\Yet) \subset \Hompink(\Xet,\Yet) \period
	\end{equation*}
	Following \Cref{convention:regarding_Hompink_as_a_set}, we tacitly identify the groupoid \smash{$ \Equiv_{k}(\Xet,\Yet) $} with the set of \textit{isomorphism classes} of equivalences $ \equivto{\Xet}{\Yet} $ over $ \Spec(k)_{\et} $.
\end{convention}


\begin{lemma}\label{lem:consequence_of_conceptual_completeness_for_etale_topoi}
	Let $ X $ and $ Y $ be topologically noetherian schemes.
	Then a geometric morphism $ \flowerstar \colon \fromto{\Xet}{\Yet} $ is an equivalence if and only if $ \Pt(\flowerstar) \colon \fromto{\Pt(\Xet)}{\Pt(\Yet)} $ is an equivalence. 
\end{lemma}

\begin{proof}
	Since $ X $ and $ Y $ are topologically noetherian schemes, the topoi $ \Xet $ and $ \Yet $ are coherent; moreover, \Cref{cor:every_geometric_morphism_between_topologically_noetherian_schemes_is_coherent} shows that every geometric morphism $ \flowerstar \colon \fromto{\Xet}{\Yet} $ is coherent. 
	The claim is now an immediate consequence of the Makkai--Reyes Conceptual Completeness Theorem (see \cites[Corollary 2.3.5]{Ultracategories}[Theorem 7.1.8]{MR0505486} or \SAG{Theorem}{A.9.0.6}).
\end{proof}

We now state the main result of this subsection.

\begin{proposition}\label{prop:consequences_of_etale_reconstruction}
	Let $ k $ be a field that satisfies étale reconstruction and let $ X $ and $ Y $ be absolutely weakly normal schemes topologically of finite type over $ k $.
	Then:
	\begin{enumerate}[label=\stlabel{prop:consequences_of_etale_reconstruction}, ref=\arabic*]
		\item\label{prop:consequences_of_etale_reconstruction.2} The natural map $ \fromto{\Isom_k(X,Y)}{\Equiv_k(\Xet,\Yet)} $ is bijective.

		\item\label{prop:consequences_of_etale_reconstruction.3} A morphism of $ k $-schemes $ f \colon \fromto{X}{Y} $ is an isomorphism if and only if $ \Pt(\flowerstar) $ is an equivalence of categories. 
 	\end{enumerate}
\end{proposition}

\begin{proof}
	Item \enumref{prop:consequences_of_etale_reconstruction}{2} is immediate from formulation \enumref{prop:equivalent_conditions_for_etale_reconstruction}{-1} of the étale reconstruction property.
	For \enumref{prop:consequences_of_etale_reconstruction}{3}, note that our assumptions guarantee that $ X $ and $ Y $ are topologically noetherian.
	Hence the claim follows from \enumref{prop:consequences_of_etale_reconstruction}{2} and \Cref{lem:consequence_of_conceptual_completeness_for_etale_topoi}.
\end{proof}

\begin{remark}[(reconstruction from condensed categories of points)]
	Let $ \Xcal $ be a coherent topos.
	The work of Barwick--Glasman--Haine \cite{arXiv:1807.03281} and Lurie \cite{Ultracategories} provides two natural refinements of the category $ \Pt(\Xcal) $ to a \textit{condensed} category (in the sense of condensed mathematics \cite{Scholze:condensednotes}).
	Let $ k $ be a field satisfying étale reconstruction.
	As a consequence of \Cref{prop:consequences_of_etale_reconstruction}, \cite[Theorem 9.3.1 \& Proposition 13.5.2]{arXiv:1807.03281}, and \cite[Remark 2.2.9 \& Theorem 4.3.3]{Ultracategories}, the functor sending an absolutely weakly normal scheme $ X $ topologically of finite type over $ k $ to either of these condensed categories of points of $ \Xet $ is fully faithful.
\end{remark}


\section{Proof of injectivity}\label{sec:proof_of_injectivity}

Let $ S $ be a qcqs Jacobson scheme and let $ X $ and $ Y $ be schemes topologically of finite type over $ S $.
In this section, we prove that if $ X $ is reduced and all fibers of $ \fromto{X}{S} $ are reduced, then the natural map
\begin{equation}\label{eq:sec_intro_want_to_show_is_injective}
	\fromto{\Hom_{S}(X,Y)}{\Hompin_{S}(\Xet,\Yet)}
\end{equation}
is injective.
See \Cref{cor:injectivity_part_of_reconstruction_for_Jacobson_schemes}.
This is the `easy part' of the reconstruction theorem; in particular, it requires minimal conditions on the base scheme $ S $.

By taking fibers over $ S $, it is straightforward to reduce to the case where $ S $ is a field.
The crux of the argument is that since morphisms from reduced schemes topologically of finite type over a field are determined by their values on closed points (\Cref{lemma:map_from_points_is_epi_for_Jacobson_schemes}), we can reduce the claim that \eqref{eq:sec_intro_want_to_show_is_injective} is injective to the case where $ X $ is the spectrum of a finite extension of $ k $.

We start by observing that reconstructing a morphism of schemes from a morphism of étale topoi is easy if the target is a finite separable field extension of $ k $.

\begin{recollection}\label{recollection:claim_is_obvious_for_domain_etale}
	Let $ k $ be a field, let $ L \supset k $ be a finite separable field extension, and write $ Y = \Spec(L) $.
	We conflate notation and also write $ Y \in \Spec(k)_{\et}$ for the associated representable sheaf.
	Then there is a natural equivalence $ \equivto{\Yet}{(\Spec(k)_{\et})_{/Y}} $ of topoi over $ \Spec(k)_{\et} $.
	Let $ \qlowerstar \colon \fromto{\Spec(L)_{\et}}{\Spec(k)_{\et}} $ denote the natural geometric morphism.
	In this case, the geometric morphisms into $ \Yet $ over $ k $ are easy to compute.
	Indeed by \cite[Exposé IV, Proposition 5.12]{MR50:7130}, for any $ k $-scheme $ p \colon \fromto{X}{\Spec(k)} $, the functor
	\begin{equation*}
		\Hom_{k}(\Xet,\Spec(L)_{\et}) \to 
		\Hom_{k}(X,\Spec(L))
	\end{equation*}
	that sends a geometric morphism $ (\tlowerstar,\alpha) $ to the morphism of schemes given by the composite
	\begin{equation*}
		\begin{tikzcd}[sep=3.5em]
			X \cong t^*(Y) \arrow[r, "t^*(\Delta)"] & \displaystyle t^*(Y \cross_{\Spec(k)} Y) = \tupperstar \qupperstar(Y) \isomorphic \pupperstar(Y) = X \cross_{\Spec(k)} Y \arrow[r, "\pr_2"] & Y
		\end{tikzcd}
	\end{equation*}
	is an equivalence of groupoids.
	Furthermore it is easy to check that this isomorphism defines an inverse to the obvious map
	\begin{equation*}
		\Hom_{k}(X,\Spec(L)) \to \Hom_{k}(\Xet,\Spec(L)_{\et}) \period
	\end{equation*}
\end{recollection}

\begin{recollection}[(closed subtopoi)]
	Let $ \Xcal $ be a topos and $ U \in \Open(\Xcal) $ a subterminal object.
	Recall that the \textit{closed complement} $ \Zcal \subset \Xcal $ of the open subtopos $ \Xcal_{/U} $ is the full subcategory of $ \Xcal $ spanned by those $ \Fcal \in \Xcal $ such that the projection $ \pr_2 \colon \Fcal \times U \to U $ is an isomorphism.
	Furthermore if $ \elowerstar \colon \Wcal \to \Xcal $ is a geometric morphism such that $ \eupperstar(U) = \emptyset $ is the intial object, then there is a unique factorization
	\begin{equation*}
		\begin{tikzcd}
			\Wcal \arrow[r,"\elowerstar"] \arrow[dr, "\elowerstar'"', dotted] & \Xcal \\
			& \Zcal \arrow[u, "i_*"', hooked]
		\end{tikzcd}
	\end{equation*}
\end{recollection}

\begin{remark}\label{remark:closed_subtopoi_of_schemes}
	Let $ X $ be a scheme and let $ Z \subset X $ a closed subscheme with open complement $ U \subset X $.
	Then the closed complement of the open subtopos $ \Uet \subset \Xet $ is equivalent to the topos $ \Zet $.
	(Note that this statement is true for any closed subscheme structure on $ |Z| $.) 
	Furthermore a geometric morphism $ \elowerstar \colon \Wet \to \Xet $ factors through the closed complement $ \Zet $ if and only if the map of topological spaces $ |\elowerstar| \colon |W| \to |X| $ factors through $ |Z| \subset |X| $.
\end{remark}

\begin{lemma}\label{lemma:topologically_finite_type_has_finite_residue_fields}
	Let $ k $ be a perfect field and $ X $ a scheme locally topologically of finite type over $ k $.
	Then for any closed point $ x \in X $, the field extension $ \upkappa(x) \supset k $ is finite.
\end{lemma}

\begin{proof}
	By definition, there is a universal homeomorphism of $ k $-schemes $ f \colon \fromto{X}{X'} $ where $ X' $ is locally of finite type over $ k $.
	Then $ f(x) \in X' $ is closed and \Cref{prop:characterization_of_Jacobson_schemes} shows that the field extension $ \upkappa(f(x)) \supset k $ is finite.
	Since $ k $ is perfect, the finite extension $ \upkappa(f(x)) $ is also perfect.
	Since $ f $ is a universal homeomorphism, the induced extension on residue fields $ \upkappa(f(x)) \supset \upkappa(x) $ is a universal homeomorphism.
	Since $ \upkappa(f(x)) $ is perfect, \stacks{0CNC} shows that the field extension $ \upkappa(f(x)) \supset \upkappa(x) $ is an isomorphism.
\end{proof}

\begin{proposition}\label{prop:reconstruction_for_fields}
	Let $ k $ be a perfect field, $ L \supset k $ an algebraic field extension, and $ Y $ a scheme topologically of finite type over $ k $.
	Then the natural map
	\begin{equation*}
		Y(L) \to \Hompink(\Spec(L)_{\et},\Yet)
	\end{equation*}
	is bijective.
\end{proposition}

\begin{proof}
	Consider the commutative triangle
	\begin{equation*}
		\begin{tikzcd}
			Y(L) \arrow[rr] \arrow[dr] & & \Hompink(\Spec(L)_{\et},\Yet) \arrow[dl] \\
			& \{\text{closed points of } Y \} & \phantom{\Hompink(\Spec(L)_{\et},\Yet)} \period
		\end{tikzcd}
	\end{equation*}
	It suffices to see that the horizontal map becomes a bijection after taking fibers over each closed point $ y_0 \in Y $.
	By \cref{remark:closed_subtopoi_of_schemes}, the map on fibers is the natural map
	\begin{equation*}
		\Hom_{k}(\Spec(L),\Spec(\upkappa(y_0))) \to 	\Hom_{k}(\Spec(L)_{\et},\Spec(\upkappa(y_0))_{\et}) \period
	\end{equation*}
	By \Cref{lemma:topologically_finite_type_has_finite_residue_fields}, the residue field $ \upkappa(y_0) $ is a finite field extension of $ k $; thus the claim follows from \cref{recollection:claim_is_obvious_for_domain_etale}.
\end{proof}

\begin{observation}\label{observation:reduce_to_perfect_fields}
	Let $ S $ be a scheme and $ f,g \colon \fromto{X}{Y} $ morphisms of $ S $-schemes.
	Let $ \fromto{T}{S} $ be a faithfully flat map.
	Then by fpqc descent, $ f = g $ if and only if $ f_T = g_T $ as morphisms of $ T $-schemes $ \fromto{X_T}{Y_T} $.

	In particular, since every extension of fields $ K \supset k $ is faithfully flat, given morphisms of $ k $-schemes  $ f,g \colon \fromto{X}{Y} $, we have $ f = g $ if and only if $ f_K = g_K $.
\end{observation}


We now prove injectivity in general.

\begin{notation}
	Given a groupoid $ \Gcal $, we write $ \uppi_0(\Gcal) $ for the set of isomorphism classes of objects of $ \Gcal $.
\end{notation}

\begin{theorem}\label{thm:injectivity_part_of_reconstruction}
	Let $ S $ be a qcqs scheme and let $ X $ and $ Y $ be schemes topologically of finite type over $ S $.
	If $ X $ is reduced and all fibers of $ X $ over $ S $ are reduced, then the natural map
	\begin{equation*}
		\fromto{\Hom_{S}(X,Y)}{\uppi_0(\Hom_{S}(\Xet,\Yet))}
	\end{equation*}
	is injective.
\end{theorem}

\begin{proof}
	Since $ X $ is reduced, the natural morphism $ \coprod_{s \in S} X_s \to X $ is an epimorphism.
	Notice that we have a commutative square
	\begin{equation*}
		\begin{tikzcd}
			\Hom_{S}(X,Y) \arrow[r] \arrow[d, hooked] & \uppi_0(\Hom_{S}(\Xet,\Yet)) \arrow[dd] \\
			\Hom_{S}(\coprod_{s \in S} X_s,Y) \arrow[d, "\wr"{xshift=-0.25ex}] & & \\
			\displaystyle \prod_{s \in S}\Hom_{\upkappa(s)}(X_s,Y_s) \arrow[r] & \displaystyle \prod_{s \in S}\uppi_0 (\Hom_{\upkappa(s)}(X_{s,\et},Y_{s,\et})) \comma
		\end{tikzcd}
	\end{equation*}
	where the right-hand vertical map is given by pulling back along the morphisms $ \Spec(\upkappa(s))_{\et} \to S_{\et} $, using \cref{rec:taking_topoi_preserves_fiber_products}.
	Since the left-hand vertical map is injective, to show that the top horizontal map is injective, it suffices to show that the bottom horizontal map is injective.
	Since all fibers of $ X $ over $ S $ are reduced, we may reduce to the case where $ S = \Spec(k) $ is the spectrum of a field.

	By \cref{observation:reduce_to_perfect_fields}, we may furthermore assume that $ k $ is perfect.
	Since $ X $ is reduced and topologically of finite type over $ k $, the natural morphism
	\begin{equation*}
		\coprod_{x \in X_{\cl}} \Spec(\upkappa(x)) \to X
	\end{equation*}
	is an epimorphism (\Cref{lemma:map_from_points_is_epi_for_Jacobson_schemes}).
	Thus \Cref{lemma:topologically_finite_type_has_finite_residue_fields} shows that we can reduce to the case where $ X = \Spec(L) $ for some finite field extension $ L \supset k $.
	In this case the claim follows from \Cref{prop:reconstruction_for_fields}.
\end{proof}

If $ S $ is furthermore Jacobson, \Cref{prop:rigidity,lem:every_morphism_between_schemes_locally_topologically_of_finite_type_is_pinned} imply the following variant:

\begin{corollary}\label{cor:injectivity_part_of_reconstruction_for_Jacobson_schemes}
	Let $ S $ be a qcqs Jacobson scheme and let $ X $ and $ Y $ be schemes topologically of finite type over $ S $.
	If $ X $ is reduced and all fibers of $ X $ over $ S $ are reduced, then the natural map
	\begin{equation*}
		\fromto{\Hom_{S}(X,Y)}{\Hompin_{S}(\Xet,\Yet)}
	\end{equation*}
	is injective.
\end{corollary}

\begin{remark}
	\Cref{cor:injectivity_part_of_reconstruction_for_Jacobson_schemes} shows that in the equivalent formulations of étale reconstruction in \Cref{prop:equivalent_conditions_for_etale_reconstruction}, we can replace the word `bijective' by `surjective' and these statements are still equivalent to étale reconstruction.
\end{remark}


\section{Reduction to regular source and target \texorpdfstring{$\GGm$}{𝔾ₘ}}\label{sec:reduction_to_regular_source_and_target_GGm}

Let $ k $ be a field.
The goal of this section is to prove that $ k $ satisfies étale reconstruction in the sense of \Cref{def:etale_reconstruction} if and only if for each affine, regular, connected, finite type $ k $-scheme $ X $ such that $ X(k) \neq \emptyset $, the natural map 
\begin{equation*}
	\Hom_{k}(\Xawn,\GGm) \to \Hompink(\Xetawn,\GGmet)
\end{equation*}
is bijective.
We accomplish this as follows.
In \cref{subsec:h-descent_for_morphisms_from_absolutely_weakly_normal_schemes}, we prove that for a fixed $ k $-scheme $ Y $, the functor
\begin{equation*}
	X \mapsto \Hom_k(\Xawn,Y) 
\end{equation*}
satisfies descent for Voevodsky's \textit{\htopology} \cites{MR2687724}{MR1403354}.
In \cref{subsec:h-descent_for_the_etale_topos}, we prove that the étale topos also satisfies \hdescent.
This implies that the functor
\begin{equation*}
	X \mapsto \Hompink(\Xetawn,\Yet)
\end{equation*}
is an \hsheaf (\Cref{cor:pinned_homomorphisms_are_an_h-sheaf}).
By the theory of alterations, every scheme of finite type over a field admits an \hhypercover by regular finite type schemes; this allows us to reduce reconstruction to the case where $ X $ an affine, regular, finite type $ k $-scheme.

Since one can extract the Zariski topological space of $ X $ from the étale topos $ \Xet $, in order to reconstruct $ X $, it suffices to reconstruct the structure sheaf $ \Ocal_X $, i.e., to prove reconstruction when $ Y = \AA_{k}^1 $.
With a little more work, we show that it actually suffices to treat the case $ Y = \GGm $ and $ X $ is connected and has a rational point. 
\Cref{subsec:reduction_to_the_regular_case} makes this informal argument precise.
There we also show the more refined statement that $ k $ satisfies étale reconstruction if and only if for each a pinned geometric morphism \smash{$ \philowerstar \colon \fromto{\Xetawn}{\GGmet} $}, there is a morphism of $ k $-schemes \smash{$ \fromto{\Xawn}{\GGm} $} that agrees with $ \philowerstar $ on $ \kbar $-points.

We note that the idea that it suffices to prove reconstruction for $ X $ affine and $ Y = \GGm $ appeared in Voevodsky's work \cite{MR1098621}.


\subsection{\hdescent for morphisms from absolutely weakly normal schemes}\label{subsec:h-descent_for_morphisms_from_absolutely_weakly_normal_schemes}

The purpose of this subsection is to prove the following:

\begin{theorem}\label{thm:h-descent_for_morphisms_from_Xawn}
	Let $ S $ be a scheme and let $ Y $ be an $ S $-scheme.
	Then the presheaf $ \Sch_{S}^{\op} \to \Set $ given by
	\begin{equation*}
		X \mapsto \Hom_S(\Xawn,Y)
	\end{equation*}
	is an \hsheaf.
\end{theorem}

For finite type schemes over a noetherian scheme of characteristic $ 0 $ this was proven in \cite[Theorem 3.2.9]{zbMATH00912445}.
The general statement is an easy consequence of the case $ Y = \AA_S^1 $:

\begin{proposition}[{\stacks{0EVU}}]\label{prop:h-descent_for_A1}
	Let $ S $ be a scheme.
	Then the presheaf $ \Sch_{S}^{\op} \to \Set $ given by
	\begin{equation*}
		X \mapsto \Hom_S(\Xawn,\AA_S^1)
	\end{equation*}
	is an \hsheaf.
\end{proposition}

\begin{proof}[Proof of \Cref{thm:h-descent_for_morphisms_from_Xawn}]
	We want to show that for any \hcover $ W \to X $, the natural map
	\begin{equation*}
		\Hom_S(\Xawn,Y) \to \lim\big\lparen \Hom_S(\Wawn,Y) \rightrightarrows \Hom_S((W \times_X W)^{\awn},Y) \big\rparen
	\end{equation*}
	is an isomorphism.
	Using \Cref{lem:absolute_weak_normalization_of_open_immersions}, we may work Zariski locally on $ Y $ and therefore assume that $ Y $ and $ S $ are both affine.
	Furthermore, since absolute weak normalization preserves Zariski covers (\Cref{lem:absolute_weak_normalization_of_open_immersions}), the functor
	\begin{equation*}
		X \mapsto \Hom_S(\Xawn,Y)
	\end{equation*}
	is a Zariski sheaf.
	Hence we may assume that $ X = \Spec(A) $ and $ W = \Spec(B) $ are affine.
	Now \Cref{prop:h-descent_for_A1} implies that the natural ring homomorphism
	\begin{equation*}
		A^{\awn} \to \lim \big\lparen B^{\awn} \rightrightarrows (B \tensor_A B)^{\awn} \big\rparen
	\end{equation*}
	is an isomorphism, which implies our claim.
\end{proof}

\begin{example}\label{ex:Hom_Xawn_GGm_is_an_h-sheaf}
	Let $ S $ be a scheme.
	Then the presheaf $ \Sch_{S}^{\op} \to \Set $ given by $ X \mapsto \Hom_S(\Xawn,\GGm) $ is an \hsheaf.
\end{example}


\subsection{\hdescent for the étale topos}\label{subsec:h-descent_for_the_etale_topos}

We now prove that the étale topos satisfies \hdescent.
We make use of the following general result.

\begin{notation}
	We write $ \Cat $ for the $ (2,1) $-category with objects locally small categories, $ 1 $-mor\-phisms functors, and $ 2 $-morphisms natural isomorphisms.
	Write $ \Gpd \subset \Cat $ for the full subcategory spanned by the groupoids.
	We write $ \LTop $ for the $ (2,1) $-category with objects topoi, $ 1 $-morphisms left exact left adjoints, and $ 2 $-morphisms natural isomorphisms.
\end{notation}

\begin{recollection}
	The inclusion $ \fromto{\LTop}{\Cat} $ preserves limits.
	See \HTT{Proposition}{6.3.2.3} for the proof of this result in the setting of \topoi.
	In the setting of ordinary topoi, the proof is verbatim the same.
\end{recollection}

\begin{proposition}\label{prop:HA.4.7.5.3}
	Let $ \Ccal^{\bullet} \colon \fromto{\Deltaplus}{\Cat} $ be an augmented cosimplicial category.
	Let $ G \colon \fromto{\Ccal^{-1}}{\Ccal^0} $ denote the agumentation.
	Assume that:
	\begin{enumerate}[label=\stlabel{prop:HA.4.7.5.3}, ref=\arabic*]
		\item\label{prop:HA.4.7.5.3.1} The category $ \Ccal^{-1} $ admits limits of $ G $-split equalizer diagrams, and those equalizers are preserved by $ G $.

		\item\label{prop:HA.4.7.5.3.2} For every morphism $ \alpha \colon \fromto{[m]}{[n]} $ in $ \Deltaplus $, the square
		\begin{equation*}
			\begin{tikzcd}
				\Ccal^m \arrow[r, "d^0"] \arrow[d, "\alphaupperstar"'] & \Ccal^{m+1} \arrow[d, "{([0] \star \alpha)\upperstar}"] \\
				\Ccal^n \arrow[r, "d^0"'] & \Ccal^{n+1}
			\end{tikzcd}
		\end{equation*}
		is horizontally right adjointable.
		(In particular, this requires that the coface functors $ d^0 $ be left adjoints.)
	\end{enumerate}
	Then the natural functor \smash{$ \theta \colon \fromto{\Ccal^{-1}}{\lim_{[n] \in \DDelta} \Ccal^n} $} admits a fully faithful right adjoint.
	Moreover, if $ G $ is conservative, then $ \theta $ is an equivalence of categories.
\end{proposition}

\begin{proof}
	The \categorical version of this statement is the content of \HA{Corollary}{4.7.5.3}.
	Since the inclusion $ \incto{\Cat}{\Catinfty} $ of ordinary categories into \categories is a right adjoint, the statement here is an immediate consequence of the \categorical statement.
\end{proof}

\begin{remark}
	\Cref{prop:HA.4.7.5.3} is also a consequence of old results of Beck (ubpublished) and Bénabou--Roubaud \cite{MR255631}.
	See \cite[p. 270]{MR1285884}.
	However, these works do not give the explicit statement we use; Lurie's result is the \categorical generalization of the exact result we use.
\end{remark}

\begin{proposition}\label{prop:X_et_is_an_h-sheaf}
	The functor $ (-)_{\et} \colon \fromto{\Sch^{\op}}{\LTop} $ with pullback functoriality is an \hsheaf.
\end{proposition}

\begin{proof}
	Since the assignment $ \goesto{X}{\Xet} $ is a Zariski sheaf, all that remains to be verified is that $ \goesto{X}{\Xet} $ satisfies descent for proper surjections.
	Let $ f \colon \fromto{X}{Y} $ be a proper surjective morphism of schemes; we need to show that the induced augmented cosimplicial diagram of categories
	\begin{equation}\label{diagam:et_of_Cech_nerve}
		\begin{tikzcd}[sep=1.5em]
		    \Yet \arrow[r, "\fupperstar"] & \Xet \arrow[r, shift left=0.75ex] \arrow[r, shift right=0.75ex] & (X \cross_Y X)_{\et} \arrow[l] \arrow[r] \arrow[r, shift left=1.5ex] \arrow[r, shift right=1.5ex] &  \cdots  \arrow[l, shift left=0.75ex] \arrow[l, shift right=0.75ex] 
		\end{tikzcd}
	\end{equation}
	is a limit diagram.
	We prove this by showing that \eqref{diagam:et_of_Cech_nerve} satisfies the hypotheses of \Cref{prop:HA.4.7.5.3}.
	For \enumref{prop:HA.4.7.5.3}{1}, note that as a pullback functor, $ \fupperstar $ is left exact.
	Moreover, since $ f $ is surjective, the functor $ \fupperstar $ is conservative, so all that remains to be shown is \enumref{prop:HA.4.7.5.3}{2}.

	For \enumref{prop:HA.4.7.5.3}{2}, notice that relevant square of categories is given by applying $ (-)_{\et} $ to the pullback square of schemes 
	\begin{equation*}
		\begin{tikzcd}
			X \crosslimits_{Y} X^{\cross_Y [n]} \arrow[r, "\pr_2"] \arrow[d, "\id{X} \cross \alpha"'] & X^{\cross_Y [n]} \arrow[d, "\alpha"] \\
			X \crosslimits_{Y} X^{\cross_Y [m]} \arrow[r, "\pr_2"'] & X^{\cross_Y [m]} \period
		\end{tikzcd}
	\end{equation*}
	Here the vertical morphisms are induced by $ \alpha \colon \fromto{[m]}{[n]} $.
	Also notice that since $ f $ is proper and proper morphisms are stable under pullback, for each $ \el \geq -1 $, the projection morphism
	\begin{equation*}
		\pr_2 \colon \fromto{X \cross_{Y} X^{\cross_Y [\el]}}{X^{\cross_Y [\el]}}
	\end{equation*}
	is proper.
	By the proper basechange theorem for étale sheaves of \textit{sets} \cite[Exposé XII, Théorème 5.1(i)]{MR50:7131}, the induced square of étale topoi
	\begin{equation*}
		\begin{tikzcd}
			(X^{\cross_Y [m]})_{\et}  \arrow[r, "\prupperstar_2"] \arrow[d, "(\id{X} \cross \alpha)\upperstar"'] & (X \crosslimits_{Y} X^{\cross_Y [m]})_{\et} \arrow[d, "\alphaupperstar"] \\
			(X \crosslimits_{Y} X^{\cross_Y [n]})_{\et} \arrow[r, "\prupperstar_2"'] & (X^{\cross_Y [n]})_{\et} \period
		\end{tikzcd}
	\end{equation*}
	is horizontally right adjointable, as desired.
\end{proof}

\begin{remark}
	Rydh has proven that the quasicompact étale site and quasicompact separated étale site both satisfy \hdescent \cite[Corollaries 5.15 \& 5.16]{MR2679038}; one can use this to show that the étale topos satisfies \hdescent.
	However, our proof method is quite different and, given \Cref{prop:HA.4.7.5.3}, more direct.
\end{remark}

\begin{remark}[{(\arcdescent for the étale topos)}]
	One can also show that the the étale topos satisfies descent for the finer \textit{\arctopology} introduced by Bhatt and Mathew \cite{MR4278670}.
	Given \Cref{prop:X_et_is_an_h-sheaf}, the proof of this is exactly the same as the proof of the analogous result for the derived \category of constructible étale sheaves \cite[Theorem 5.13]{MR4278670}.
\end{remark}

\begin{corollary}\label{cor:pinned_homomorphisms_are_an_h-sheaf}
	Let $ k $ be a field and let $ Y $ be a $ k $-scheme.
	Then the functor $ \Sch_k^{\op} \to \Gpd $ given by 
	\begin{equation*}
		X \mapsto \Hompink(\Xet, \Yet)
	\end{equation*}
	is an \hsheaf.
	Furthermore, when restricted to schemes topologically of finite type over $ k $, it is an \hsheaf of sets.
\end{corollary}

\begin{proof}
	By \Cref{prop:X_et_is_an_h-sheaf}, it follows that $\Hom_{k}((-)_{\et}, \Yet)$ is an \hsheaf of groupoids.
	Since $\Hompink((-)_{\et}, \Yet)$ is a subfunctor of the former, we only need to show that a geometric morphism $ \flowerstar \colon \Xet \to \Yet $ is pinned if there is an \hcover $p \colon W \to X $ such that $ (f \circ p)_* $ is pinned.
	For this it suffices to see that $ p $ is surjective on closed points.
	This holds since the fiber over every closed point is nonempty an qcqs and therefore has a closed point.
	The `furthermore' part is a consequence of \Cref{prop:rigidity}.
\end{proof}


\subsection{Reduction to the regular case}\label{subsec:reduction_to_the_regular_case}

In this subsection, we show that a field $ k $ satisfies étale reconstruction if and only if for each affine, regular, connected, finite type $ k $-scheme $ X $ with a rational point and pinned geometric morphism \smash{$ \philowerstar \colon \fromto{\Xetawn}{\GGmet} $}, there exists a morphism of $ k $-schemes \smash{$ \fromto{\Xawn}{\GGmet} $} that agrees with $ \philowerstar $ on \smash{$ \kbar $}-points. 

We begin with a technical proposition needed in our proof.
To prove this, we state some facts about the étale topos proven in \Cref{app:morphisms_of_etale_sites_and_etale_topoi}.

\begin{nul}
	Let $ X $ be a scheme.
	We write $ \Etqcs_X \subset \Et_X $ for the full subcategory spanned by those $ p \colon \fromto{X'}{X} $ such that $ p $ is quasicompact and separated.
	If $ X $ is quasiseparated, then restriction along the inclusion defines an equivalence of categories
	\begin{equation*}
		\equivto{\Sh_{\et}(\Et_X)}{\Sh_{\et}(\Etqcs_X)}
	\end{equation*}
	See \Cref{lem:separated_and_separated_finitely_presented_etale_sites_are_bases}.
	Moreover, if $ X $ and $ Y $ are topologically noetherian schemes, then for every geometric morphism $ \flowerstar \colon \fromto{\Xet}{\Yet} $, the functor $ \fupperstar $ sends $ \Etqcs_Y $ to $ \Etqcs_X $.
	Hence $ \flowerstar $ is induced by the morphism of sites \smash{$ \fupperstar \colon \fromto{\Etqcs_Y}{\Etqcs_X} $}.
	See \Cref{cor:every_geometric_morphism_between_topologically_noetherian_schemes_is_coherent}.
\end{nul}

The following is a slight generalization of \cite[Proposition 2.3]{MR1098621}. For convenience of the reader we also spell out a proof.

\begin{proposition}\label{prop:Suffices_to_reconstruct_on_points}
	Let $ k $ be a field and let $ X $ and $ Y $ be schemes topologically of finite type over $ k $; also assume that  $ X $ is reduced.
	Let $ (\flowerstar,\alpha) \colon \Xet \to \Yet $ be a pinned geometric morphism of topoi over $ \Spec(k)_{\et} $.
	Assume that for each $ V \in \Etqcs_Y $, there exists a morphism of $ k $-schemes $ \phi_V \colon \fupperstar(V) \to V $ such that for any morphism of $ k $-schemes $ x \colon \Spec(\kbar) \to \fupperstar(V) $ we have 
	\begin{equation*}
		(\restrict{f}{V})_* \circ \xlowerstar \isomorphic \phi_{V,*} \circ \xlowerstar
	\end{equation*}
	in $ \Hompink(\Spec(\kbar)_{\et}, \Yet) $.
	Then $ \flowerstar \isomorphic \phi_{Y,*} $ in $ \Hompink(\Xet, \Yet) $.
\end{proposition}

\begin{proof}
	First note that, as just explained, $ \fupperstar(V) $ is representable by an étale $ X $-scheme.
	Now observe that by assumption the square
	\begin{equation}\label{sq:desired_naturality}
		\begin{tikzcd}
			\fupperstar(V) \arrow[r, "\phi_V"] \arrow[d] & V \arrow[d] \\
			X \arrow[r, "\phi_Y"'] & Y
		\end{tikzcd}
	\end{equation}
	commutes after precomposing with any geometric point $ \Spec(\kbar) \to \fupperstar(V)$.
	Since $ X $ is reduced, $ \fupperstar(V) $ is also reduced; hence it follows that the square \eqref{sq:desired_naturality} commutes.
	The universal property of the pullback induces a morphism of étale $ X $-schemes
	\begin{equation*}
		\gamma_V \colon \fupperstar(V) \to \phiupperstar_Y(V) \period
	\end{equation*}
	In particular, $ \gamma_V $ is quasicompact and étale.
	The same reasoning shows that \smash{$ (\gamma_V)_{V \in \Etqcs_Y} $} defines a natural transformation \smash{$ \gamma \colon \fupperstar \to \phiupperstar_Y $}.

	We claim that each $ \gamma_V $ is an isomorphism.
	For this it suffices to see that $ \gamma_V $ induces an isomorphism on geometric fibers over closed points of $ X $.
	Indeed, because $ X $ is Jacobson, this implies that $ \gamma_V $ is an étale surjective monomorphism and thus an isomorphism by {\stacks{06NC}}.
	But for a morphism $ x \colon \Spec(\kbar) \to X $ we have by assumption $ \xupperstar \fupperstar \cong (x \circ \phi_Y)^* $ and therefore $ \gamma $ induces an endomorphism of the stalk functor $ (x \circ \phi_Y)^* $.
	But by the description of the category of points given recalled in \Cref{ex:Pt_of_Xet}, any endomorphism of a point of the étale topos is an isomorphism.
	In particular, $ \gamma $ is an isomorphism.
	
	To conclude the proof, we only need to check that the isomorphism $ \gamma $ is compatible with $ \alpha $, i.e., that $ \gamma $ defines an $ 2 $-isomorphism in $ \LTop_k $.
	Explicitly we have to check that for any finite separable field extension $ E $ of $ k $  the triangle of $ X $-schemes
	\begin{equation}\label{triang:desired_compatibility_of_gamma_with_alpha}
		\begin{tikzcd}
			& \fupperstar(Y_E) \arrow[dr, "\gamma_{Y_E}"] & \\
			X_E \arrow[ur, "\alpha_E"] \arrow[rr, "\sim"{yshift=-0.25ex}] & & \phiupperstar_Y(Y_E)
		\end{tikzcd}
	\end{equation}
	commutes.
	It again suffices to check that for any geometric point $ x \colon \Spec(\kbar) \to X $ whose image is a closed point of $ Y $, the triangle \eqref{triang:desired_compatibility_of_gamma_with_alpha} commutes after applying $ \xupperstar $.
	The induced triangle reads
	\begin{equation*}
		\begin{tikzcd}
			& (f \circ x)^*(Y_E) \arrow[dr, "\xupperstar \gamma_{Y_E}"] & \\
			\Spec(\kbar)_E \arrow[rr, "\sim"{yshift=-0.25ex}] \arrow[ur, "\xupperstar\alpha_E"] & & \xupperstar\phiupperstar_Y(Y_E) \period
		\end{tikzcd}
	\end{equation*}
	Therefore, it suffices to see that the natural isomorphism $ \xupperstar \gamma $ defines a natural isomorphism in $ \LTop_k $.
	For this we apply \Cref{lemma:technical_stuff} below to the geometric morphism $ \flowerstar \circ \xlowerstar $, the map of schemes $ \phi_Y \circ x $, the natural isomorphism $ \gamma_x \colon \isomto{\xupperstar \circ \fupperstar}{\xupperstar \phiupperstar_Y} $ that we are given by assumption, and the collection of maps
	\begin{equation*}
		\begin{tikzcd}
			\phitilde_V \colon \xupperstar \fupperstar(V) \arrow[r] & \fupperstar(V) \arrow[r, "\phi_V"] & V \comma
		\end{tikzcd}
	\end{equation*}
	where the left-hand map is the projection from the pullback.
	For this we have to verify that the assumptions of \Cref{lemma:technical_stuff} below are satisfied.
	The only nontrivial point that we have to verify is that there are natural isomorphisms $ \xupperstar \restrict{f}{V}^* \isomorphic \phitilde_V^{*} $ for any $ V \in \Etqcs_Y $.
	But $ \xupperstar \fupperstar(V) $ is just a disjoint union of spectra of separably closed fields, and by assumption we have such a natural isomorphism for every connected component, giving us the desired natural transformation in general.
	Thus \Cref{lemma:technical_stuff} implies that $ \xupperstar \gamma = \gamma_x $.
	In particular $ \xupperstar \gamma $ is a natural isomorphism in $ \LTop_k $, because $ \gamma_x $ is by assumption.
\end{proof}

\begin{lemma}\label{lemma:technical_stuff}
	Let $ k $ be a field and let $ Y $ be scheme topologically of finite type over $ k $.
	Let $ \flowerstar \colon \Spec(\kbar)_{\et} \to \Yet $ be a pinned geometric morphism, $ x \colon \Spec(\kbar) \to Y $ a morphism of $ k $-schemes and $ \gamma \colon \isomto{\xlowerstar}{\flowerstar} $ an isomorphism in \smash{$ \Hompink(\Spec(\kbar)_{\et},\Yet) $}.
	Assume that we are given a natural collection of maps \smash{$ \paren{\phi_V \colon \fupperstar(V) \to V}_{V \in \Etqcs_Y} $} such that:
	\begin{enumerate}[label=\stlabel{lemma:technical_stuff}, ref=\arabic*]
		\item\label{lemma:technical_stuff.1} We have $ \phi_Y = x $.

		\item\label{lemma:technical_stuff.2} We have natural isomorphisms $ \restrict{f}{V}^* \isomorphism \phiupperstar_V$ in $ \Hompink(\fupperstar(V)_{\et},\Vet) $.
	\end{enumerate}
	Then for each $ V \in \Etqcs_Y $, the map $ \gammatilde_V \colon \fupperstar(V) \to \xupperstar(V) $ induced by the naturality square
	\begin{equation*}
		\begin{tikzcd}[column sep=3em]
			\fupperstar(V) \arrow[r, "\phi_V"] \arrow[d] & V \arrow[d] \\
			\Spec(\kbar) \arrow[r, "\phi_Y= x"'] & Y		
		\end{tikzcd}
	\end{equation*} 
	and the universal property of the pullback agrees with $ \gamma_V \colon \isomto{\fupperstar(V)}{\xupperstar(V)} $.
\end{lemma}

\begin{proof}
	By the universal property of the pullback we only have to check that the diagram of schemes
	\begin{equation*}
		\begin{tikzcd}
			\fupperstar(V) \arrow[dr, "\gamma_V"{description}] \arrow[drr, "\phi_V", bend left=20] \arrow[ddr, bend right=20] \\
			& \xupperstar(V) \arrow[r, "\restrict{x}{V}"'] \arrow[d] & V \arrow[d] \\
			& \Spec(\kbar) \arrow[r, "x"'] & Y
		\end{tikzcd}
	\end{equation*}
	commutes.
	The only non-obvious equation we have to check for this is $ \phi_V = \restrict{x}{V}\circ \gamma_V $.
	For this, by \Cref{prop:reconstruction_for_fields} it suffices to show that the associated diagram in $ \RTop_{k} $ commutes.

	By assumption \enumref{lemma:technical_stuff}{2} on the maps $ \phi_V $, it suffices to provide a natural isomorphism
	\begin{equation*}
		\delta \colon \isomto{\restrict{f}{V}^*}{\gamma_V^* \circ \restrict{x}{V}^*}
	\end{equation*}
	in $ \LTop_k $.
	For any quasicompact separated étale morphism $ T \to V $, the commutative square
	\begin{equation*}
		\begin{tikzcd}
			\fupperstar(T) \arrow[r, "\gamma_T", "\sim"'{yshift=0.25ex}] \arrow[d] & \xupperstar(T) \arrow[d] \\
			\fupperstar(V) \arrow[r, "\gamma_V"', "\sim"{yshift=-0.25em}] & \xupperstar(V)
		\end{tikzcd}
	\end{equation*}
	induces an isomorphism
	\begin{equation*}
		\delta_T \colon \fupperstar(T) \to \fupperstar(V) \times_{\xupperstar(V)} \xupperstar(T) \period
	\end{equation*}
	Finally, it is easy to check that the morphisms \smash{$ (\delta_{T})_{T \in \Etqcs_V} $} define a natural isomorphism in $ \LTop_k $, because $ \gamma $ does.
\end{proof}

We are now ready to prove the main result of this section.
Our proof makes use of the \hdescent results of \cref{subsec:h-descent_for_morphisms_from_absolutely_weakly_normal_schemes,subsec:h-descent_for_the_etale_topos} and the following observation.


\begin{theorem}[{(reduction to regular source and target $ \GGm $)}]\label{thm:reduction_to_the_regular_case}
	Let $ k $ be a field. 
	Then the following statements are equivalent:
	\begin{enumerate}[label=\stlabel{thm:reduction_to_the_regular_case}, ref=\arabic*]
		\item\label{thm:reduction_to_the_regular_case.1} The field $ k $ \emph{satisfies étale reconstruction} in the sense of \Cref{def:etale_reconstruction}.
		That is, for all finite type $ k $-schemes $ X $ and $ Y $, the natural map 
		\begin{equation*}
			\Hom_{k}(\Xawn,Y) \to \Hompink(\Xetawn,\Yet)
		\end{equation*}
		is bijective.

		\item\label{thm:reduction_to_the_regular_case.2} For all finite type $ k $-schemes $ X $ and $ Y $, such that $ X $ is affine, regular, connected, and $ X(k) \neq \emptyset $, the natural map 
		\begin{equation*}
			\Hom_{k}(\Xawn,Y) \to \Hompink(\Xetawn,\Yet)
		\end{equation*}
		is bijective.
		
		\item \label{thm:reduction_to_the_regular_case.3} Let $ X $ and $ Y $ be finite type $ k $-schemes, such that $ X $ is affine, regular, connected, and $ X(k) \neq \emptyset $.
		Then for any pinned geometric morphism $ \philowerstar \colon \Xet \to \Yet $ over $ \Spec(k)_{\et} $, there exists a morphism of $ k $-schemes $ f \colon \Xawn \to Y $ such that the induced map $ f(\kbar) \colon \Xawn(\kbar) \to Y(\kbar) $ agrees with $ \philowerstar(\kbar) $.

		\item\label{thm:reduction_to_the_regular_case.4} Let $ X $ be an affine, regular, connected, finite type $ k $-scheme such that $ X(k) \neq \emptyset $.
		Then for any pinned geometric morphism $ \philowerstar \colon \Xetawn \to \AA_{\et}^1 $ over $ \Spec(k)_{\et} $, there exists a morphism of $ k $-schemes $ f \colon \Xawn \to \AA^1 $ such that the induced map
		\begin{equation*}
			f(\kbar) \colon \Xawn(\kbar) \to \kbar
		\end{equation*}
		agrees with $ \philowerstar(\kbar) $.

		\item\label{thm:reduction_to_the_regular_case.5} Let $ X $ be an affine, regular, connected, finite type $ k $-scheme such that $ X(k) \neq \emptyset $.
		Then for any pinned geometric morphism $ \philowerstar \colon \Xetawn \to \GGmet $ over $ \Spec(k)_{\et} $, there exists a morphism of $ k $-schemes $ f \colon \Xawn \to  \GGm $ such that the induced map 
		\begin{equation*}
			f(\kbar) \colon \Xawn(\kbar)\to \kbar^{\units}
		\end{equation*}
		agrees with $ \philowerstar(\kbar) $.
	\end{enumerate}
\end{theorem}

\begin{proof}
	Clearly \enumref{thm:reduction_to_the_regular_case}{1} $ \Rightarrow $ \enumref{thm:reduction_to_the_regular_case}{2}.
	To see that \enumref{thm:reduction_to_the_regular_case}{2} $ \Rightarrow $ \enumref{thm:reduction_to_the_regular_case}{1}, note that by \Cref{thm:h-descent_for_morphisms_from_Xawn,cor:pinned_homomorphisms_are_an_h-sheaf} both the source and target of the map in \enumref{thm:reduction_to_the_regular_case}{1} satisfy \hdescent in $ X $.
	Thus we may work \hlocally and by the theory of alterations \cites[Theorem 1.1]{Illusie:On_Gabbers_refined_uniformization}[Exposé IX, Théorème 1.1]{MR3309086}[Theorem 4.4]{MR3730515}[Theorem 1.2.5]{MR3665001} we may reduce to the case where $ X $ is regular.
	By further working Zariski locally, we may reduce to the case where $ X $ is affine.
	Passing to a finite extension of $ k $ if necessary, we may assume that $ X $ is connected and $ X(k) \neq \emptyset$.
	This completes the proof that \enumref{thm:reduction_to_the_regular_case}{2} $ \Rightarrow $ \enumref{thm:reduction_to_the_regular_case}{1}.

	Clearly \enumref{thm:reduction_to_the_regular_case}{2} implies \enumref{thm:reduction_to_the_regular_case}{3} and the converse follows from \Cref{cor:injectivity_part_of_reconstruction_for_Jacobson_schemes} combined with \Cref{prop:Suffices_to_reconstruct_on_points}.
	Clearly \enumref{thm:reduction_to_the_regular_case}{3} implies \enumref{thm:reduction_to_the_regular_case}{4} and \enumref{thm:reduction_to_the_regular_case}{5}.
	Since we can cover $ \AA^1 $ by two copies of $ \GGm $, it follows that if we manage to construct the desired morphism of $ k $-schemes over both copies of $ \GGm $, we can reconstruct it over $ \AA^1 $: since the two maps agree on $ X(\kbar) $ by assumption, they will automatically agree over the intersection.
	Thus \enumref{thm:reduction_to_the_regular_case}{5} implies \enumref{thm:reduction_to_the_regular_case}{4}.
	
	To finish the proof it thus suffices to show that \enumref{thm:reduction_to_the_regular_case}{4} $ \Rightarrow $ \enumref{thm:reduction_to_the_regular_case}{3}.
	We may work Zariski locally on both the source and target to assume that $ X = \Spec(A) $ and $ Y = \Spec(B) $ are affine.
	Let $ \philowerstar \colon \Xet \to \Yet$ be a pinned geometric morphism.
	Define a map of sets
	\begin{equation*}
		\psi \colon B = \Hom_k(Y,\AA^1) \longrightarrow \Hom_k(\Xawn,\AA^1) = \Aawn
	\end{equation*}
	by sending $ b \colon Y \to \AA_{k}^1 $ to the unique element $ \psi(b) \colon \Xawn \to \AA^1 $ such that $ \psi(b)(\kbar) = \blowerstar \circ  \philowerstar (\kbar) $.
	Indeed, we assumed that such an element $ \psi(b) $ exists and it is clearly unique.
	We claim that $ \psi \colon B \to A^{\awn}  $ is a ring homomorphism.
	In other words we have to see that the square
	\begin{equation*}
		\begin{tikzcd}[column sep=3.5em]
			\Hom_k(Y,\AA^1 \times \AA^1) \arrow[r,"\mult_*"] \arrow[d, "\psi \times\psi"'] & \Hom_k(Y,\AA^1) \arrow[d, "\psi"] \\
			\Hom_k(\Xawn,\AA^1 \times \AA^1) \arrow[r, "\mult_*"'] & \Hom_k(\Xawn,\AA^1)
		\end{tikzcd}
	\end{equation*}
	where $ \mult \colon \AA^1 \times \AA^1 \to \AA^1 $ is the multiplication map, commutes and similarly for the addition map.
	Both cases are exactly the same, so we just explain the case of the multiplication map.

	For this, by \Cref{prop:characterization_of_Jacobson_schemes,lemma:map_from_points_is_epi_for_Jacobson_schemes} we may assume that $ X = \Spec(\kbar) $.
	Now observe that if we show that the square
	\begin{equation}\label{sq:comparing_psi_and_phi}
		\begin{tikzcd}
			\Hom_k(Y,\AA^1\times \AA^1) \arrow[r] \arrow[d, "\psi \times\psi"'] & \Hompink(\Yet,(\AA^1 \times \AA^1)_{\et}) \arrow[d, "- \of \philowerstar"] \\
			\Hom_k(\Spec(\kbar),\AA^1 \times \AA^1) \arrow[r, "\sim"{yshift=-0.25em}] & \Hompink(\Spec(\kbar)_{\et},(\AA^1 \times \AA^1)_{\et})
		\end{tikzcd}
	\end{equation}
	commutes, then the claim easily follows from \Cref{cor:injectivity_part_of_reconstruction_for_Jacobson_schemes}.
	A consequence of \Cref{prop:reconstruction_for_fields} the canonical map
	\begin{equation*}
		\Hompink(\Spec(\kbar)_{\et},(\AA^1 \times \AA^1)_{\et}) \to \Hompink(\Spec(\kbar)_{\et},\AA_{\et}^1) \cross \Hompink(\Spec(\kbar)_{\et},\AA_{\et}^1)
	\end{equation*}
	induced by the projections is an isomorphism.
	Thus it suffices to see that the square \eqref{sq:comparing_psi_and_phi} commutes after composing with the two projections $ \fromto{\AA^1 \times \AA^1}{\AA^1} $.
	This is immediate from the definition of $ \psi $.
	
	It follows that $ \psi $ is indeed a map of rings and therefore induces a morphism of schemes $ f \colon \Xawn \to Y $.
	All that is left to show is that for any morphism of $ k $-schemes $ x \colon \Spec(\kbar) \to \Xawn $. we have $ \philowerstar \circ \xlowerstar \isomorphic \flowerstar \circ \xlowerstar $.
	In other words, we have to show that the square
	\begin{equation*}
		\begin{tikzcd}
			\Hom_k(\Spec(\kbar),\Xawn) \arrow[r, "\sim"{yshift=-0.25em}] \arrow[d, "f \of -"'] & \Hompink(\Spec(\kbar)_{\et},\Xet) \arrow[d, "\philowerstar \of -"] \\
			\Hom_k(\Spec(\kbar),Y) \arrow[r, "\sim"{yshift=-0.25em}] & \Hompink(\Spec(\kbar)_{\et},\Yet)
		\end{tikzcd}
	\end{equation*}
	commutes.
	Note that since $ Y $ is affine it suffices to prove this after composing with any morphism of $ k $-schemes $ b \colon Y \to \AA^1 $.
	Then the claim is clear from the construction of $ \psi $.
\end{proof}

\begin{remark}
	For the proof of \cref{intro_thm:etale_reconstruction}, it is not necessary to use the theory alterations.
	Indeed, using that any finite type $ k $-scheme admits a finite surjection from its normalization, the proof of \Cref{thm:reduction_to_the_regular_case} also shows that it suffices to prove the conclusion of \enumref{thm:reduction_to_the_regular_case}{5} for normal schemes.
 	But by \Cref{prop:finite_generation_of_Pic_of_normal_schemes_over_finitely_generated_fields} the proof given in \cref{sec:proof_of_the_reconstruction_theorem} works more generally for normal schemes over finitely generated fields.
\end{remark}


\section{Amenable fields}\label{sec:amenable_fields}

In \cref{sec:reduction_to_regular_source_and_target_GGm}, we proved that a field $ k $ satisfies étale reconstruction if and only if for each affine, regular, connected, finite type $ k $-scheme $ X $ with a rational point and pinned geometric morphism \smash{$ \philowerstar \colon \fromto{\Xetawn}{\GGmet} $}, there exists a morphism of $ k $-schemes \smash{$ \fromto{\Xawn}{\GGmet} $} that agrees with $ \philowerstar $ on $ \kbar $-points. 
In \cref{sec:proof_of_the_reconstruction_theorem}, we will check this criterion for infinite fields that are finitely generated over their prime fields.
However, the argument we give works over more general fields.
The idea is to relate both sides to the étale cohomology group
\begin{equation*}
	\Het^1(X;\GGm) \isomorphic \Pic(X)
\end{equation*}
and prove that the desired map is surjective by an an obstruction-theoretic argument.
The purpose of this section is to introduce the class of fields that for which our argument works; we call these \textit{amenable} fields.

In order to formulate the condition of amenability properly, we need to avoid some bad behavior that happens at the characteristic.
To do this, in \cref{sec:prime-to-p_Kummer_theory}, we introduce variants of divisibility, profinite completion, Tate modules, and étale cohomology with coefficients in $ \ZZhat(1) $ that are all away from a prime.
In \cref{sec:the_definition_and_stability_properties}, we introduce amenable fields and show that finitely generated fields are amenable.


\subsection{Prime-to-\texorpdfstring{$p$}{p} Kummer theory}\label{sec:prime-to-p_Kummer_theory}

We begin by introducing prime-to-$ p $ divisibility.

\begin{notation}
	Let $ p $ be a prime number or $ 1 $.
	We write $ \NNpprime \subset \NN $ for the set of positive integers $ n $ such that $ \gcd(n,p) = 1 $.
	Throughout, we regard $ \NNpprime $ as a poset ordered by divisibility.
\end{notation}

\begin{definition}
	Let $ p $ be a prime number or $ 1 $, and let $ A $ be an abelian group.
	We say that an element $ a \in A $ is \defn{$ p' $-divisible} (or \defn{prime-to-$ p $-divisible}) if for each $ n \in \NNpprime $, there exists an element $ b \in A $ such that $ nb = a $.
	We say that $ A $ is \defn{$ p' $-divisible} if every element of $ A $ is $ p' $-divisible.
\end{definition}

\begin{observation}
	If $ p = 1 $, then a $ p' $-divisible abelian group is exactly a \textit{divisible} abelian group.
\end{observation}

\begin{notation}\label{ntn:n-torsion_and_Tate_module}
	Let $ A $ be an abelian group (or group scheme).
	\begin{enumerate}[label=\stlabel{ntn:n-torsion_and_Tate_module}, ref=\arabic*]
		\item For each integer $ n \geq 1 $, write $ A[n] \subset A $ for the kernel of multiplication by $ n $, i.e., the \defn{$ n $-torsion} in $ A $.

		\item Let $ p $ be a prime number or $ 1 $.
		The \defn{$ p' $-Tate module} of $ A $ is the limit
		\begin{equation*}
			\Tatep(A) \colonequals \lim_{n \in (\NNpprime)^{\op}} A[n] \period
		\end{equation*}
		For $ m $ dividing $ n $, the transition map $ \fromto{A[n]}{A[m]} $ is given by multiplication by $ n/m $.
	\end{enumerate} 
\end{notation}

\begin{observation}\label{obs:if_there_are_no_p'-torsion_elements_then_the_p'-Tate_module_vanishes}
	Let $ p $ be a prime number or $ 1 $.
	If an abelian group $ A $ has no nonzero $ p' $-divisible torsion elements, then $ \Tatep(A) = 0 $.
\end{observation}

\begin{lemma}\label{lem:inverting_p_and_Tate_away_from_p}
	Let $ p $ be a prime number or $ 1 $ and let $ A $ be an abelian group.
	Then:
	\begin{enumerate}[label=\stlabel{lem:inverting_p_and_Tate_away_from_p}, ref=\arabic*]
		\item\label{lem:inverting_p_and_Tate_away_from_p.1} For each $ n \in \NNpprime $, the natural homomorphism $ \fromto{A}{A[\nicefrac{1}{p}]} $ induces an isomorphism
		\begin{equation*}
			\isomto{A[n]}{(A[\nicefrac{1}{p}])[n]}
		\end{equation*}
		on $ n $-torsion subgroups.

		\item\label{lem:inverting_p_and_Tate_away_from_p.2} The natural homomorphism $ \fromto{A}{A[\nicefrac{1}{p}]} $ induces an isomorphism
		\begin{equation*}
			\isomto{\Tatep(A)}{\Tatep(A[\nicefrac{1}{p}])} \period
		\end{equation*}
	\end{enumerate}
\end{lemma}

\begin{proof}
	For \enumref{lem:inverting_p_and_Tate_away_from_p}{1}, observe that since $ \ZZ[\nicefrac{1}{p}] $ is flat over $ \ZZ $, we have
	\begin{equation*}
		(A[n])[\nicefrac{1}{p}] \isomorphic A[n] \tensor_{\ZZ} \ZZ[\nicefrac{1}{p}] \isomorphic (A[\nicefrac{1}{p}])[n] \period
	\end{equation*}
	So it suffices to show that $ p $ acts invertibly on $ A[n] $.
	For this, note that $ A[n] $ is a $ \ZZ/n $-module and since $ p $ is coprime to $ n $, the class of $ p $ in $ \ZZ/n $ is a unit.
	Item \enumref{lem:inverting_p_and_Tate_away_from_p}{2} is an immediate consequence of \enumref{lem:inverting_p_and_completing_away_from_p}{1} and the definition of the $ p' $-Tate module.
\end{proof}

Now we study completions of abelian groups away from a prime.

\begin{notation}\label{ntn:p'-completion}
	Let $ p $ be a prime number or $ 1 $.
	For an abelian group $ A $, the \defn{$ p' $-completion} of $ A $ is the limit
	\begin{equation*}
		A\pprimecomp \colonequals \lim_{n \in (\NNpprime)^{\op}} A/n \comma
	\end{equation*}
	where the transition maps are the quotient maps.
	That is, \smash{$ A\pprimecomp $} is the maximal prime-to-$ p $ quotient of the profinite completion of $ A $.
	We write
	\begin{equation*}
		\ZZpprime \colonequals \ZZ\pprimecomp \isomorphic \prod_{\substack{\textup{primes} \\ \ell \neq p}} \ZZell \period
	\end{equation*}
\end{notation}

\begin{lemma}\label{lem:inverting_p_and_completing_away_from_p}
	Let $ p $ be a prime number or $ 1 $ and let $ A $ be an abelian group.
	Then:
	\begin{enumerate}[label=\stlabel{lem:inverting_p_and_completing_away_from_p}, ref=\arabic*]
		\item\label{lem:inverting_p_and_completing_away_from_p.1} For each $ n \in \NNpprime $, the natural homomorphism $ \fromto{A}{A[\nicefrac{1}{p}]} $ induces an isomorphism $ \isomto{A/n}{A[\nicefrac{1}{p}]/n} $.

		\item\label{lem:inverting_p_and_completing_away_from_p.2} The natural homomorphism $ \fromto{A}{A[\nicefrac{1}{p}]} $ induces an isomorphism $ \isomto{A\pprimecomp}{A[\nicefrac{1}{p}]\pprimecomp} $.
	\end{enumerate}
\end{lemma}

\begin{proof}
	For \enumref{lem:inverting_p_and_completing_away_from_p}{1}, observe that 
	\begin{equation*}
		A[\nicefrac{1}{p}]/n \isomorphic (A/n) \tensor_{\ZZ} \ZZ[\nicefrac{1}{p}] \isomorphic (A/n)[\nicefrac{1}{p}] \comma
	\end{equation*}
	so it suffices to show that $ p $ acts invertibly on $ A/n $.
	For this, note that $ A/n $ is a $ \ZZ/n $-module and since $ p $ is coprime to $ n $, the class of $ p $ in $ \ZZ/n $ is a unit.
	Item \enumref{lem:inverting_p_and_completing_away_from_p}{2} is an immediate consequence of \enumref{lem:inverting_p_and_completing_away_from_p}{1} and the definition of $ p' $-completion.
\end{proof}

\begin{nul}
	Let $ k $ be a field of exponential characteristic $ p $ and let $ X $ be a $ k $-scheme.
	Then for any $n \in \NNpprime$, the Kummer sequence provides a natural short exact sequence
	\begin{equation}\label{eq:n-Kummer_sequence}
		\begin{tikzcd}
			0 \arrow[r] & \Units(X)/n \arrow[r] & \Het^1(X;\mun) \arrow[r] & \Pic(X)[n] \arrow[r] & 0 \period
		\end{tikzcd}
	\end{equation}
	Here, $\mun$ is the étale sheaf of $n$-th roots of unity. 
\end{nul}

\begin{definition}[{(the $ p' $-Kummer map)}]
	Let $ k $ be a field of exponential characteristic $ p $ and let $ X $ be a $ k $-scheme.
	Write
	\begin{equation*}
		\Het^1(X;\ZZpprime(1)) \colonequals \lim_{n \in (\NNpprime)^{\op}} \Het^1(X;\mun) \period
	\end{equation*}
	We then define the \defn{$ p' $-Kummer map}
	\begin{equation*}
		\Kum \colon \Units(X)\pprimecomp \longrightarrow \Het^1(X;\ZZpprime(1))
	\end{equation*}
	as the limit of the maps $\Units(X)/n \to \Het^1(X;\mun) $ for $n \in \NNpprime $, coming from the Kummer sequence.
\end{definition}

\begin{observation}[{(functoriality of $ \Het^1(X;\ZZpprime(1)) $)}]\label{obs:functoriality_of_Het1}
	Let $ k $ be a field of exponential characteristic $ p $ and let $ X $ and $ Y $ be $ k $-schemes.
	Then for each positive integer $ n $ coprime to $ p $, the sheaves $ \mubf_{n,X} $ and $ \mubf_{n,Y} $ are the pullbacks of $ \mubf_{n,\Spec(k)} $.
	As a consequence, for any geometric morphism $ \flowerstar \colon \fromto{\Xet}{\Yet} $ over $ \Spec(k)_{\et} $ (not necessarily induced by a morphism of schemes), we have $ \fupperstar(\mubf_{n,Y}) \isomorphic \mubf_{n,X} $.
	Hence for each $ i \geq 0 $, the assignment $ \goesto{X}{\Het^i(X;\mun)} $ is functorial in geometric morphisms of étale topoi over $ \Spec(k)_{\et} $.
	Thus the assignment
	\begin{equation*}
		\goesto{X}{\Het^1(X;\ZZpprime(1))}
	\end{equation*}
	is also functorial in geometric morphisms of étale topoi over $ \Spec(k)_{\et} $.  
\end{observation}

\begin{observation}\label{obs:p'-Kummer_exact_sequence}
	Let $ k $ be a field of exponential characteristic $ p $ and let $ X $ be a $ k $-scheme.
	For $ m $ dividing $ n $, the transition map $ \fromto{\Units(X)/n}{\Units(X)/m} $ is surjective.
	Hence the cofiltered system defining the $ p' $-completion of $ \Units(X) $ is a Mittag-Leffler system.
	So \stacks{0598} shows that taking limits over the sequences \eqref{eq:n-Kummer_sequence}, we obtain an exact sequence
	\begin{equation}\label{eq:pprime-Kummer_sequence}
		\begin{tikzcd}[sep=3em]
			0 \arrow[r] & \Units(X)\pprimecomp \arrow[r, "\Kum"] & \Het^1(X;\ZZpprime(1)) \arrow[r] & \Tatep(\Pic(X)) \arrow[r] & 0 \period
		\end{tikzcd}
	\end{equation}
	In particular, the $ p' $-Kummer map is an isomorphism if and only if $ \Tatep(\Pic(X)) = 0 $.
\end{observation}


\subsection{Amenability}\label{sec:the_definition_and_stability_properties}

The following is the main definition of this section.

\begin{definition}\label{def:amenable}
	Let $ k $ be a field of exponential characteristic $ p $.
	We say that $ k $ is \defn{amenable} if the following conditions hold: 
	\begin{enumerate}[label=\stlabel{def:amenable}, ref=\arabic*]
		\item\label{def:amenable.1} For each torus $ T $ over $ k $, the abelian group $T(k)$ has no nonzero $ p' $-divisible elements.

		\item\label{def:amenable.2} For each regular finite type $ k $-scheme $ X $, we have $ \Tatep(\Pic(X)) = 0 $.
	\end{enumerate}
\end{definition}

For future use, we record the following consequences of \Cref{obs:if_there_are_no_p'-torsion_elements_then_the_p'-Tate_module_vanishes,obs:p'-Kummer_exact_sequence}:

\begin{observation}\label{obs:consequences_of_amenable}
	Let $ k $ be of exponential characteristic $ p $.
	\begin{enumerate}[label=\stlabel{obs:consequences_of_amenable}, ref=\arabic*]
		\item If condition \enumref{def:amenable}{1} is satisfied, then for any torus $ T $ over $ k $, the Kummer map
		\begin{equation*}
			T(k)  \to \lim_{n \in (\NNpprime)^{\op}} \Het^1(k;T[n])
		\end{equation*}
		is injective.

		\item If for each regular finite type $ k $-scheme $ X $, the group $ \Pic(X) $ has no $ p' $-divisible torsion elements, then condition \enumref{def:amenable}{2} is satisfied.

		\item Condition \enumref{def:amenable}{2} is satisfied if and only if for each regular finite type $ k $-scheme $ X $, the $ p' $-Kummer map $ \Kum \colon \fromto{\Units(X)\pprimecomp}{\Het^1(X;\ZZpprime(1))} $ is an isomorphism.
	\end{enumerate}
\end{observation}

\noindent We now deduce analogous facts about absolute weak normalizations.

\begin{lemma}\label{lem:units_in_Xawn}
	Let $ k $ be a field of exponential characteristic $ p $ and let $ X $ be a qcqs seminormal $ k $-scheme.
	Then the natural homomorphism $ \fromto{\Units(X)}{\Units(\Xawn)} $ induces an isomorphism 
	\begin{equation*}
		\isomto{\Units(X)[\nicefrac{1}{p}]}{\Units(\Xawn)} \period
	\end{equation*}
\end{lemma}

\begin{proof}
	If $ p = 1 $, then since $ X $ is seminormal, $ \Xawn = X $ and there is nothing to show. 
	If $ p > 1 $, then $ \Xawn $ is the perfection of $ X $, thus applying \stacks{01Z0} we deduce that
	\begin{align*}
		\Units(\Xawn) &\isomorphic \colim \paren{ \Units(X) \xrightarrow{\cdot p} \Units(X) \xrightarrow{\cdot p}  \cdots } \\ 
		&= \Units(X)[\nicefrac{1}{p}] \period \qedhere
	\end{align*} 
\end{proof}

\begin{lemma}[{(perfection and the Picard group \cite[Lemma 3.5]{MR3674218})}]\label{lem:perfection_inverts_p_in_Pic}
	Let $ X $ be a qcqs $ \FFp $-scheme.
	Then pullback along the counit $ \fromto{\Xperf}{X} $ induces an isomorphism
	\begin{equation*}
		\isomto{\Pic(X)[\nicefrac{1}{p}]}{\Pic(\Xperf)} \period
	\end{equation*}
\end{lemma}

\begin{lemma}\label{lem:alternative_description_of_H1(Xawn;ZZpprime(1))_for_amenable_fields}
	Let $ k $ be an amenable field and let $ X $ be a regular finite type $ k $-scheme.
	Then: 
	\begin{enumerate}[label=\stlabel{lem:alternative_description_of_H1(Xawn;ZZpprime(1))_for_amenable_fields}, ref=\arabic*]
		\item\label{lem:alternative_description_of_H1(Xawn;ZZpprime(1))_for_amenable_fields.1} Pullback along the counit $ \fromto{\Xawn}{X} $ induces an isomorphism
		\begin{equation*}
			\isomto{\Tatep(\Pic(X))}{\Tatep(\Pic(\Xawn))} \period
		\end{equation*}

		\item\label{lem:alternative_description_of_H1(Xawn;ZZpprime(1))_for_amenable_fields.2} The $ p' $-Kummer map \smash{$ \Kum \colon \Units(\Xawn)\pprimecomp \to \Het^1(\Xawn;\ZZpprime(1)) $} is an isomorphism.

		\item\label{lem:alternative_description_of_H1(Xawn;ZZpprime(1))_for_amenable_fields.3} We have $ \Tatep(\Pic(\Xawn)) = 0 $.
	\end{enumerate}
\end{lemma}

\begin{proof}
	For \enumref{lem:alternative_description_of_H1(Xawn;ZZpprime(1))_for_amenable_fields}{1}, note that if $ p = 1 $, then $ \Xawn = X $.
	Hence \Cref{lem:perfection_inverts_p_in_Pic} shows that with no conditions on $ p $, pullback along the counit induces an isomorphism 
	\begin{equation*}
		\isomto{\Pic(X)[\nicefrac{1}{p}]}{\Pic(\Xawn)} \period
	\end{equation*}
	By \Cref{lem:inverting_p_and_Tate_away_from_p} we see that pullback along the counit induces isomorphisms 
	\begin{equation*}
		\Tatep(\Pic(X)) \isomorphism \Tatep(\Pic(X)[\nicefrac{1}{p}]) \isomorphism \Tatep(\Pic(\Xawn)) \period
	\end{equation*}
	To conclude, note that \Cref{obs:p'-Kummer_exact_sequence} shows that \enumref{lem:alternative_description_of_H1(Xawn;ZZpprime(1))_for_amenable_fields}{2} $ \Leftrightarrow $ \enumref{lem:alternative_description_of_H1(Xawn;ZZpprime(1))_for_amenable_fields}{3}.
	Moreover, \Cref{obs:consequences_of_amenable} shows that $ \Tatep(\Pic(X)) = 0 $; so item \enumref{lem:alternative_description_of_H1(Xawn;ZZpprime(1))_for_amenable_fields}{1} implies \enumref{lem:alternative_description_of_H1(Xawn;ZZpprime(1))_for_amenable_fields}{3}. 
\end{proof}

The property of being amenable is stable under finite field extensions:

\begin{lemma}\label{lem:amenable_is_stable_under_finite_extension}
	Let $ k $ be an amenable field and let $ L \supset k $ be a finite field extension. 
	Then $ L $ is also amenable.
\end{lemma}

\begin{proof}
	If $ X $ is regular and of finite type over $ L $, it is also regular and of finite type over $ k $.  
	Thus, we only need to check the condition on tori.
	If $L \supset k$ is separable, then this follows by noting that for any torus $T $ over $ L $, the Weil restriction $\Res_{L/k}(T)$ is a torus over $ k $ with the property that
	\begin{equation*}
		\Res_{L/k}(T)(k) = T(L) \period
	\end{equation*}  

	In the case that $L \supset k$ is finite purely inseparable, $\Res_{L/k}$ does not carry tori to tori, so we cannot reason in the same way. 
	In this case, note that by Galois descent, it is enough to check that our statement is true for the torus $\GGm$ over $ L $. 
	Since $L \supset k$ is finite purely inseparable, there is some $n>0$ such that $L^{p^n} \subset k $. 
	If $a \in L^{\units}$ is $ p' $-divisible, then so is $a^{p^n} \in k$.
	By our assumption on $ k $, this implies that $a^{p^n} = 1$.
	Thus $a = 1$, so the statement holds for tori. 
\end{proof}

\begin{definition}
	We say that a field $ k $ is \defn{finitely generated} if $ k $ is a finitely generated field extension of its prime field. 
\end{definition}

\noindent We now show that finitely generated fields are amenable. 
We first recall the the following fundamental result on the finite generation of Picard groups.

\begin{proposition}[{\cite[Proposition 6.1]{MR1399035}}]\label{prop:finite_generation_of_Pic_of_normal_schemes_over_finitely_generated_fields}
	Let $ k $ be a finitely generated field and let $ X $ be normal finite type $ k $-scheme.
	Then $ \Pic(X) $ is finitely generated.
\end{proposition}

\begin{lemma}\label{lem:torus_part_of_amenability_is_stable_under_passing_to_function_fields}
	Let $ L $ be a field of exponential characteristic $ p $.
	If for any torus $ T $ over $ k $, the abelian group $T(L)$ has no nonzero $ p' $-divisible elements, then the function field $ L(t) $ has the same property.
\end{lemma}

\begin{proof}
	It suffices to show the statement for the standard torus $\GG_{\mup,L(t)}$, so assume that $f \in L(t)^{\units} $ is $ p' $-divisible. 
	We then immediately see that $ f $ is constant, and by the assumption on $ L $, this shows that $ f = 1 $. 
\end{proof}

\begin{proposition}\label{prop:finitely_generated_fields_are_amenable}
	If $ k $ is a finitely generated field, then $ k $ is amenable.
\end{proposition}

\begin{proof}
	Given \Cref{prop:finite_generation_of_Pic_of_normal_schemes_over_finitely_generated_fields}, we only have to show that tori over $ k $ have no $ p' $-divisible points. 
	Since any finitely generated field is a finite extension of a purely transcendental extension of its prime field, by \Cref{lem:amenable_is_stable_under_finite_extension}, it is enough to show that if $ L $ has the property that no $ L $-torus has $ p' $-divisible elements, then $L(t)$ has the same property.
	This is the content of \Cref{lem:torus_part_of_amenability_is_stable_under_passing_to_function_fields}. 
\end{proof}



\section{Reconstruction for amenable discretely valued fields}\label{sec:proof_of_the_reconstruction_theorem}

The goal of this section is to prove the main result of this paper: if $ k $ is an amenable field that admits a nontrivial discrete valuation, then $ k $ satisfies étale reconstruction (\Cref{thm:reconstruction_for_amenable_fields}). 
In particular, we deduce that infinite finitely generated fields satisfy étale reconstruction (\Cref{cor:finitely_generated_fields_satisfy_etale_reconstruction}).

In \cref{subsec:the_cohomological_argument} we provide a `cohomological' criterion for amenable fields to satisfy étale reconstruction (\Cref{prop:discreteness_of_chi_top_implies_étale_reconstruction}).
\Cref{subsec:finite_generation_and_units} proves some preparatory lemmas needed to check this criterion in the presence of a discrete valuation.
In \cref{subsec:amenable_discretely_valued_fields_satisfy_etale_reconstruction}, we prove the main result.


\subsection{A cohomological criterion for reconstructing amenable fields}\label{subsec:the_cohomological_argument}

To state our cohomological criterion for étale reconstruction, we begin by fixing some notation.

\begin{notation}[{(the canonical class in $ \Het^1(\GGm;\ZZpprime(1)) $)}]\label{ntn:class_in_GGm}
	Let $ k $ be a field of exponential characteristic $ p $.
	For each positive integer $ n $ coprime to $ p $, note that the Kummer sequence provides a canonical class
	\begin{equation*}
		c_{1,n} \in \Het^1(\GGm;\mun) \period
	\end{equation*}
	Note that for $ m $ dividing $ n $, the homomorphism $ \fromto{\mun}{\mum} $ given by raising to the $ n/m $-th power induces a homomorphism
	\begin{equation*}
		\fromto{\Het^1(\GGm;\mun)}{\Het^1(\GGm;\mum)}
	\end{equation*}
	that sends $ c_{1,n} $ to $ c_{1,m} $.
	We write $ c_1 $ for the class
	\begin{equation*}
		(c_{1,n})_{n \in (\NNpprime)^{\op}} \in \Het^1(\GGm;\ZZpprime(1)) \period
	\end{equation*}
\end{notation}

\begin{notation}
	Let $ k $ be a field of exponential characteristic $ p $ and let $ X $ be a $ k $-scheme.
	Following Voevodsky we write
	\begin{equation*}
		\OTop(X) \colonequals \Hompink(\Xet,\GGmet) \period
	\end{equation*}
	Define a map $ \chitop' $ by the assignment
	\begin{align*}
		\chitop' \colon \OTop(X) &\to \Het^1(X; \ZZpprime(1)) \\ 
		\fupperstar &\mapsto \fupperstar(c_1) \period
	\end{align*}
	Here, note that the class $ \fupperstar(c_1) $ is well-defined by the functoriality of the assignment $ \goesto{X}{\Het^1(X;\ZZpprime(1))} $ in geometric morphisms of étale topoi over $ \Spec(k)_{\et} $ explained in \Cref{obs:functoriality_of_Het1}.
\end{notation}

\begin{nul}
	By construction, the square
	\begin{equation}\label{eq:square_relating_chitop_and_Kum}
		\begin{tikzcd}[column sep=3em]
			\Units(X) \arrow[r, "\can"] \arrow[d] & \Units(X)\pprimecomp \arrow[d, "\Kum"] \\ 
			\OTop(X) \arrow[r, "\chitop'"'] & \Het^1(X;\ZZpprime(1))
		\end{tikzcd}
	\end{equation}
	commutes.
\end{nul}	

\begin{construction}[{($ \chitop $)}]\label{const:chitop_with_desired_target}
	Let $ k $ be an amenable field of exponential characteristic $ p $ and let $ X $ be a regular finite type $ k $-scheme.
	Recall from \Cref{lem:alternative_description_of_H1(Xawn;ZZpprime(1))_for_amenable_fields} that the $ p' $-Kummer map
	\begin{equation*}
		\Kum \colon \Units(\Xawn)\pprimecomp \longrightarrow \Het^1(\Xawn;\ZZpprime(1))
	\end{equation*} 
	is an isomorphism.
	We write $ \chitop $ for the composite
	\begin{equation*}
		\begin{tikzcd}[sep=3em]
			\OTop(\Xawn) \arrow[r, "\chitop'"] & \Het^1(\Xawn; \ZZpprime(1)) \arrow[r, "\Kum^{-1}", "\sim"'{yshift=0.25ex}] & \Units(\Xawn)\pprimecomp \period
		\end{tikzcd}
	\end{equation*}
\end{construction}

Our goal is to show that if the image of $ \chitop $ is contained in the image of the natural homomorphsim $ \fromto{\Units(\Xawn)}{\Units(\Xawn)\pprimecomp} $, then $ k $ satisfies étale reconstruction.
To prove this criterion, we need a construction and a couple of observations.

\begin{construction}\label{const:psi_map}
	Let $ k $ be a field of exponential characteristic $ p $, let $ Y $ be a $ \kawn $-scheme, and let $ E \supset \kawn $ be a finite field extension.
	For any morphism of $ \kawn $-schemes $ x \colon \Spec(E) \to Y $, we obtain a pullback map
	\begin{equation*}
		\xupperstar \colon \Het^1(Y;\ZZpprime(1)) \longrightarrow \Het^1(\Spec(E);\ZZpprime(1)) \isomorphic (E^{\units})\pprimecomp \period
	\end{equation*}
	Thus for any fixed class $ \xi \in \Het^1(Y;\ZZpprime(1)) $ we obtain a map
	\begin{equation*}
		Y(E) \to (E^{\units})\pprimecomp \comma \qquad x \mapsto \xupperstar(\xi) \period
	\end{equation*}
	Passing to the colimit over all finite extensions of $\kawn$, we obtain a map
	\begin{equation*}
		\psi_\xi \colon Y(\kbar) \to \colim_{E} (E^{\units})\pprimecomp \period
	\end{equation*}
	We write
	\begin{equation*}
		\evhat(-,-) \colon \fromto{Y(\kbar) \cross \Het^1(Y;\ZZpprime(1))}{\colim_{E} (E^{\units})\pprimecomp}
	\end{equation*}
	for the pairing corresponding to the assignment that sends a class $ \xi \in \Het^1(Y;\ZZpprime(1)) $ to the map $ \psi_\xi $.
\end{construction}

\begin{notation}
	Let $ k $ be an amenable field of exponential characteristic $ p $ and let $ X $ be a regular finite type $ k $-scheme.
	We also write $ \evhat $ for the composite
	\begin{equation}\label{eq:pairing}
		\begin{tikzcd}[sep=4em]
			\Xawn(\kbar) \cross \Units(\Xawn)\pprimecomp \arrow[r, "\id{} \cross \Kum", "\sim"'{yshift=0.25ex}] & \Xawn(\kbar) \cross \Het^1(\Xawn;\ZZpprime(1)) \arrow[r, "\evhat"] & \colim_{E} (E^{\units})\pprimecomp \period
		\end{tikzcd}
	\end{equation}
\end{notation}

\begin{observation}\label{obs:can_map_injective}
	Let $ k $ be an amenable field of exponential characteristic $ p $.
	Then for any separable field extension $ E \supset k $, the canonical map
	\begin{equation*}
		E^{\units} \to (E^{\units})\pprimecomp
	\end{equation*}
	is injective because it may be identified with the Kummer map.
	Since $ (E^{\awn})^{\units} = E^{\units}[\nicefrac{1}{p}] $ and $ p $ is invertible on the right hand side, we also deduce that the natural map
	\begin{equation*}
		(E^{\awn})^{\units} \to (E^{\awn,\times})\pprimecomp
	\end{equation*}
	is injective.
	Passing to the colimit over all finite separable field extensions, we therefore obtain an injective map
	\begin{equation*}
		i \colon \kbar^{\units} \to \colim_{E} (E^{\awn,\times})\pprimecomp \period
	\end{equation*}
\end{observation}

\begin{observation}
	If $ Y = \GGm $ in \Cref{const:psi_map} and $ c_1 \in \Het^1(\GGm;\ZZpprime(1)) $ is the canonical class, then the induced map
	\begin{equation*}
		\psi_{c_1} = \evhat(-,c_1) \colon \kbar^{\units} = \GGm(\kbar) \to \colim_{E} (E^{\units})\pprimecomp
	\end{equation*}
	is the canonical map $ i $.
\end{observation}
	
\begin{observation}\label{obs:naturality_of_evaluation}
	For any pinned geometric morphism $ \flowerstar \colon \Xet \to \Yet $ between schemes topologically of finite type over $\kawn$, by construction, the square
	\begin{equation*}
		\begin{tikzcd}[row sep=3em, column sep=4.5em]
			X(\kbar) \times \Het^1(Y;\ZZpprime(1)) \arrow[r, "\flowerstar \cross \id{}"] \arrow[d, "\id{} \cross \fupperstar"'] & Y(\kbar) \times \Het^1(Y;\ZZpprime(1)) \arrow[d, "\evhat"] \\
			X(\kbar) \times \Het^1(X;\ZZpprime(1)) \arrow[r, "\evhat"'] & \displaystyle \colim_{E} (E^{\units})\pprimecomp
		\end{tikzcd}
	\end{equation*}
	commutes.
	In particular, taking $ Y = \GGm $ and restricting to
	\begin{equation*}
		\Xawn(\kbar) \cross \{c_1\} \subset \Xawn(\kbar) \cross \Het^1(\GGm;\ZZpprime(1)) \comma
	\end{equation*}
	it follows that for any $ \flowerstar \in \OTop(\Xawn) $ the triangle
	\begin{equation*}
		\begin{tikzcd}[row sep=3em]
			\Xawn(\kbar) \arrow[rr, "\flowerstar"] \arrow[dr, "\psi_{\chitop(\flowerstar)}"'] & &  \kbar^{\units} \arrow[dl, "i"] \\
			& \displaystyle \colim_{E} \, (E^{\units})\pprimecomp &
		\end{tikzcd}
	\end{equation*}
	commutes.
\end{observation}

\begin{proposition}\label{prop:discreteness_of_chi_top_implies_étale_reconstruction}
	For an amenable field $ k $ of exponential characteristic $ p $, the following are equivalent:
	\begin{enumerate}[label=\stlabel{prop:discreteness_of_chi_top_implies_étale_reconstruction}, ref=\arabic*]
		\item\label{prop:discreteness_of_chi_top_implies_étale_reconstruction.1} The field $ k $ satisfies étale reconstruction.

		\item\label{prop:discreteness_of_chi_top_implies_étale_reconstruction.2} For each affine, regular, connected, finite type $ k $-scheme $ X $ with $ X(k) \neq \emptyset $, we have
		\begin{equation*}
			\image\paren{ \chitop \colon \fromto{\OTop(\Xawn)}{\Units(\Xawn)\pprimecomp} } \subset \image\paren{ \can \colon \fromto{\Units(\Xawn)}{\Units(\Xawn)\pprimecomp} } \period
		\end{equation*}
	\end{enumerate}
\end{proposition}

\begin{proof}
	To see that \enumref{prop:discreteness_of_chi_top_implies_étale_reconstruction}{1} $ \Rightarrow $ \enumref{prop:discreteness_of_chi_top_implies_étale_reconstruction}{2}, note that if $ k $ satisfies étale reconstruction, then for each regular finite type $ k $-scheme $ X $, the natural map
	\begin{equation*}
		\fromto{\Units(\Xawn)}{\OTop(\Xawn)}
	\end{equation*}
	is an isomorphism.
	So the commutativity of the square \eqref{eq:square_relating_chitop_and_Kum} shows that $ \image(\chitop) \subset \image(\can) $.

	For the implication \enumref{prop:discreteness_of_chi_top_implies_étale_reconstruction}{2} $ \Rightarrow $ \enumref{prop:discreteness_of_chi_top_implies_étale_reconstruction}{1}, choose a separable closure $ \kbar $ of $ \kawn $, and let \smash{$ \philowerstar \in \OTop(\Xawn) $} be a pinned geometric morphism \smash{$ \fromto{\Xetawn}{\GGmet} $} over $ k $.
	\Cref{thm:reduction_to_the_regular_case} shows that it suffices to construct an element $ f \in \Units(\Xawn) $ that agrees with $ \philowerstar $ on $ \Xawn(\kbar)$.
	By assumption, there exists an element $ f \in \Units(\Xawn) $ mapping to $ \chitop(\philowerstar) $ under the natural map
	\begin{equation*}
		\can \colon \Units(\Xawn) \to \Units(\Xawn)\pprimecomp \period
	\end{equation*}
	It follows from \Cref{obs:naturality_of_evaluation} that for any $ \xbar \in \Xawn(\kbar) $ we have
	\begin{equation*}
		i(f(\xbar)) = \evhat(\xbar,\can(f)) = \evhat(\xbar,\chitop(\philowerstar)) = \psi_{\chitop(\philowerstar)}(\xbar) = i (\philowerstar(\xbar)) \period
	\end{equation*}
	Because $ k $ is amenable, $ i $ is injective (\Cref{obs:can_map_injective}), so $ f(\xbar) = \philowerstar(\xbar) $, as desired.
\end{proof}


\subsection{Finite generation \& units}\label{subsec:finite_generation_and_units}

The rest of this section is dedicated to verifying the hypotheses of \Cref{prop:discreteness_of_chi_top_implies_étale_reconstruction} when $ k $ is amenable and admits a nontrivial discrete valuation.
Our argument relies on understanding the kernel of the homomorphism $ \fromto{\Units(\Xawn)}{(\kawn)^{\units}} $ induced by a rational point $ x \in X(k) $.
It will be useful to know that this kernel is finitely generated and free; this subsection proves this.

\begin{recollection}
	A field extension $ K \supset k $ is \defn{regular} if it is separable and $ k $ is algebraically closed in $ K $.
\end{recollection}

\begin{example}[{\cite[Proposition 5.51]{MR4225278}}]\label{ex:geometrically_integral_implies_regular}
	Let $ k $ be a field and let $ X $ be an integral $ k $-scheme.
	Then the field extension $ \Kup(X) \supset k $ is regular if and only if $ X $ is geometrically integral over $ k $.
\end{example}

\begin{lemma}\label{lem:K/k_torsion-free}
	Let $ \incto{A}{B} $ be an injective ring homomorphism.
	If $ A $ is integrally closed in $ B $, then the abelian group $ B^{\units}/A^{\units} $ is torsion-free.
\end{lemma}

\begin{proof}
	Let $ b \in B^{\units} $ be an element such that $ b^n = a $ for some $ a \in A^{\units} $. 
	Then the polynomial $ t^n - a \in A[t] $ is an integral equation for $ b $.
	Since $ A $ is integrally closed in $ B $, we deduce that $ b \in A $.
	Moreover, since $ b^n $ is a unit in $ A $, we deduce that $ b $ is also a unit in $ A $.
	Hence the class of $ b $ is the identity in $ B^{\units}/A^{\units} $.
\end{proof}

\begin{proposition}\label{prop:R/k_finitely_generated}
	Let $ k $ be a field and $ K \supset k $ a finitely generated regular field extension.
	Let $ R \subset K $ be a $ k $-subalgebra that is finitely generated as a $ k $-algebra.
	Then:
	\begin{enumerate}[label=\stlabel{prop:R/k_finitely_generated}, ref=\arabic*]
		\item\label{prop:R/k_finitely_generated.1} The abelian group $ R^{\units}/k^{\units} $ is finitely generated and free.

		\item\label{prop:R/k_finitely_generated.2} For any $ k $-algebra augmentation $ \epsilon \colon \fromto{R}{k} $, there is an isomorphism $ \ker(\epsilon^{\units}) \isomorphic R^{\units}/k^{\units} $. 
		In particular, $ \ker(\epsilon^{\units}) $ is finitely generated and free.
	\end{enumerate}
\end{proposition}

\begin{proof}
	For \enumref{prop:R/k_finitely_generated}{1}, the claim that $ R^{\units}/k^{\units} $ is finitely generated is proven in \cite[Chapter 2, Corollary 7.3]{MR715605}.
	Hence it suffices to show that the group $ R^{\units}/k^{\units} $ is torsion-free. 
	Notice that the incusion $ R \subset K $ induces an injection $ \incto{R^{\units}/k^{\units}}{K^{\units}/k^{\units}} $.
	By \Cref{lem:K/k_torsion-free}, the group $ K^{\units}/k^{\units} $ is torsion-free; hence the subgroup $ R^{\units}/k^{\units} $ is also torsion-free.

	For \enumref{prop:R/k_finitely_generated}{2}, consider the short exact sequence of abelian groups
	\begin{equation*}
		\begin{tikzcd}
			1 \arrow[r] & \ker(\epsilon^{\units}) \arrow[r, hooked] & R^{\units} \arrow[r, "\epsilon^{\units}"] & k^{\units} \arrow[r] & 1 \period
		\end{tikzcd}
	\end{equation*}
	Since $ \epsilon $ is an augmentation, $ \epsilon^{\units} \colon \fromto{R^{\units}}{k^{\units}} $ admits a section.
	The splitting lemma completes the proof.
\end{proof}

\begin{notation}\label{ntn:Units_1}
	Let $ k $ be a field and let $ X $ be a $ k $-scheme with fixed rational point $ x \in X(k) $.
	Denote by $ \Unitsone(X) \subset \Units(X) $ the subgroup spanned by those $ f $ such that $ f(x) = 1 $.
	That is, $ \Unitsone(X) $ is the kernel of the map $ \fromto{\Units(X)}{k^{\units}} $ induced by the $ k $-point $ x $.
	By functoriality, the $ k $-point $ x $ induces a $ \kawn $-point $x^{\awn}$ of $ \Xawn $.
	We also write $\Unitsone(\Xawn) \subset \Units(\Xawn)$ for the subgroup spanned by those $ f $ such that $f(x^{\awn}) = 1 $.
\end{notation}

\begin{corollary}\label{cor:Units_1_finitely_generated}
	Let $ k $ be a field and let $ X $ be an affine, geometrically integral, finite type $ k $-scheme.
	Then for each $ x \in X(k) $, the abelian group $ \Unitsone(X) $ is finitely generated and free.
\end{corollary}

\begin{proof}
	Since $ X $ is geometrically integral and of finite type over $ k $, \Cref{ex:geometrically_integral_implies_regular} shows that the field extension $ \Kup(X) \supset k $ is finitely generated and regular.
	Since the $ k $-subalgebra $ \Ocal(X) \subset \Kup(X) $ is finitely generated, the claim now follows from \Cref{prop:R/k_finitely_generated}.
\end{proof}

\begin{lemma}\label{lem:Units_1_is_finitely_generated}
	Let $ k $ be a field of exponential characteristic $ p $ and let $ X $ be an affine, regular, connected, finite type $ k $-scheme with $ X(k) \neq \emptyset $.
	Then for each $ x \in X(k) $, the $ \ZZ[\nicefrac{1}{p}] $-module $ \Unitsone(\Xawn) $ is finitely generated and free.
\end{lemma}

\begin{proof}
	First notice that since $ X $ is a connected $ k $-scheme that admits a rational point, $ X $ is geometrically connected \stacks{04KV}.
	Since $ X $ is regular and of finite type over $ k $, we deduce that $ X $ is also geometrically irreducible.

	If $ p = 1 $, then $ X = \Xawn $ and $ X $ is smooth over $ k $.
	Since $ X $ is geometrically irreducible, \stacks{056T} shows that $ X $ is geometrically integral over $ k $.
	Thus \Cref{cor:Units_1_finitely_generated} shows that $ \Unitsone(X) $ is a finitely generated free abelian group. 
	If $ p > 1 $, then since the natural morphism $ \fromto{(X_{\kperf})_{\red}}{X} $ is a universal homeomorphism, \Cref{lem:units_in_Xawn} shows that 
	\begin{equation*}
		\Unitsone(\Xawn) \isomorphic \Unitsone((X_{\kperf})_{\red})[\nicefrac{1}{p}] \period
	\end{equation*}
	Now note that since $ X $ is geometrically irreducible, $ (X_{\kperf})_{\red} $ is also geometrically irreducible.
	Moreover, since $ \kperf $ is perfect and $ (X_{\kperf})_{\red} $ is reduced, $ (X_{\kperf})_{\red} $ is also geometrically reduced.
	Hence $ (X_{\kperf})_{\red} $ is geometrically integral.
	To conclude, note that \Cref{cor:Units_1_finitely_generated} shows that $ \Unitsone((X_{\kperf})_{\red}) $ is a finitely generated free abelian group.
\end{proof}


\subsection{Amenable discretely valued fields satisfy étale reconstruction}\label{subsec:amenable_discretely_valued_fields_satisfy_etale_reconstruction}

We now verify that the hypothesis of \Cref{prop:discreteness_of_chi_top_implies_étale_reconstruction} holds for amenable fields that admit a nontrivial discrete valuation.
This is a slight generalization of Voevodsky's argument in \cite[Proposition 3.4]{MR1098621}.
For this, we need two preparatory results.

\begin{nul}
	Let $ p $ be a prime number or $ 1 $.
	Recall that by \Cref{lem:inverting_p_and_completing_away_from_p}, we have \smash{$ \ZZpprime \isomorphic \ZZ[\nicefrac{1}{p}]\pprimecomp $}.
	Also note that the natural homomorphism \smash{$ \fromto{\ZZ[\nicefrac{1}{p}]}{\ZZ[\nicefrac{1}{p}]\pprimecomp} $} is injective.
	Hence for the rest of this section, we tacitly regard $ \ZZ[\nicefrac{1}{p}] $ as a subgroup of \smash{$ \ZZpprime $}.
\end{nul}

\begin{lemma}\label{lemma:integer_up_to_p-power} 
	Let $ p $ be a prime number or $ 1 $, and let \smash{$ \alpha \in \ZZpprime $}. 
	If there exists a nonzero element \smash{$ n \in \ZZ[\nicefrac{1}{p}] $} such that $ n \alpha \in \ZZ[\nicefrac{1}{p}] \subset \ZZpprime $, then $\alpha \in \ZZ[\nicefrac{1}{p}]$.
\end{lemma}

\begin{proof}
	For $ \ell \neq p $, write $ \pr_{\ell} \colon \ZZpprime \to \ZZell $ for the projection.
	By assumption,
	\begin{equation*}
		\alpha \in \QQ \subset \ZZpprime \tensor \QQ \andeq \pr_\ell(\alpha) \in \ZZell \subset \QQell \period
	\end{equation*}
	Thus the $ \ell $-adic valuation of $ \alpha $ is greater than or equal to $ 0 $.
	So $ \alpha = m/p^k $ for some $ m \in \ZZ $ and $ k \in \NN $.
\end{proof}

\noindent The following result is the only part that seems to break over a finite field:

\begin{proposition}\label{prop:argument_using_discrete_valuation}
	Let $ k $ be an amenable field of exponential characteristic $ p $.
	Let $ X $ be an affine, regular, connected, finite type $ k $-scheme with a fixed rational point $ x \in X(k) $.
	Choose a basis $ g_1,\ldots,g_m $ for the finitely generated free $ \ZZ[\nicefrac{1}{p}] $-module $ \Unitsone(\Xawn) $.
	Let 
	\begin{equation*}
		\xi = g_1^{a_1} \cdots g_m^{a_m} \in \Unitsone(\Xawn)\pprimecomp
	\end{equation*}
	where $ a_1,\ldots,a_m \in \ZZpprime $, be such that $ \evhat(-,\xi) $ factors through
	\begin{equation*}
		i \colon \kbar^{\units} \to \colim_E (E^{\units})\pprimecomp \period
	\end{equation*}
	If $ k $ admits a nontrivial discrete valuation, then $ \xi \in \Unitsone(\Xawn) $.
\end{proposition}

\begin{proof}
	We note that the discrete valuation on $ k $ extends to a nontrivial $\ZZ[\nicefrac{1}{p}]$-valued valuation on $\kawn$. 
	We make the following claim:

	\begin{claim}\label{claim:in_argument_using_discrete_valuations}
		For any $ 1 \leq d \leq m $, we can find a finite extension $ E \supset \kawn $, points $ x_1, \dots, x_d \in \Xawn(E) $ and an extension of the valuation on $\kawn$ to a valuation $\nu \colon E^{\units} \to \ZZ[\nicefrac{1}{p}] $ such that the determinant of the $ \ZZ[\nicefrac{1}{p}] $-matrix
		\begin{equation*}
			A(d) = (\nu(g_s(x_t)))_{s,t}
		\end{equation*}
		is nonzero.
	\end{claim}

	If we can prove \Cref{claim:in_argument_using_discrete_valuations} we are done.
	To see this, note that since $ \Unitsone(\Xawn) $ is a finitely generated free $ \ZZ[\nicefrac{1}{p}] $-module, to prove that $ \xi \in \Unitsone(\Xawn) $, we equivalently need to show that $ a_1,\ldots,a_m \in \ZZ[\nicefrac{1}{p}] $.
	Write \smash{$ \nuhat \colon (E^{\units})\pprimecomp \to \ZZpprime $} for the $ p' $-completion of $ \nu $.
	Note that for $ d = m $ we have
	\begin{equation*}
		A(m) \cdot
		\begin{pmatrix}
			a_1 \\ 
			\vdots \\ 
			a_m
		\end{pmatrix}
		= 
		\begin{pmatrix}
			\nuhat(\evhat(x_1,\xi)) \\ 
			\vdots \\ 
			\nuhat(\evhat(x_m,\xi))
		\end{pmatrix}
		\period
	\end{equation*}
	Since $\nuhat(\evhat(x_d,\xi)) \in \ZZ[\nicefrac{1}{p}]$ by assumption, by multiplying with the adjugate of the $ \ZZ[\nicefrac{1}{p}] $-matrix $ A(m) $, we see that each $a_j$ has the property that 
	\begin{equation*}
		\det(A(m)) a_j \in \ZZ[\nicefrac{1}{p}] \period
	\end{equation*}
	Since $ \det(A(m)) \in \ZZ[\nicefrac{1}{p}] $ is nonzero, \Cref{lemma:integer_up_to_p-power} implies that $ a_1,\ldots,a_m \in \ZZ[\nicefrac{1}{p}] $, as desired. 
\end{proof}

\begin{proof}[Proof of \Cref{claim:in_argument_using_discrete_valuations}]
	We proceed by induction on $ d $. 
	For the base case $ d = 1 $, we claim that there is some finite extension $E \supset \kawn$, and some $x \in \Xawn(E)$ such that $\nu(g_1(x)) \neq 0$.  
	Pick a nontrivial extension of the valuation of $\kawn$ to $E$, and consider the induced map $g_1 \colon \Xawn \to \GGm$. 
	Since $ g_1 $ is not constant, $ g_1 $ is an open map. 
	Since the complement of $g_1(\Xawn)$ is a finite set of closed points, we can pick some $x \in \Xawn(E)$  for some extension $ E \supset \kawn $ such that $\nu(g_1(x)) \neq 0$, . 

	For the induction step, assume that the statement holds for $ d $. 
	Consider the matrix $A(d+1) = (\nu(g_s(x_t)))_{s,t} $ for some points $x_1,\ldots,x_{d+1} $ yet to be chosen, and expand $ \det(A(d+1)) $ formally in minors as
	\begin{equation*}
		\det(A(d+1)) = \sum_{s=1}^{d+1} (-1)^{s+1} \nu(g_s(x_1)) M_s \comma
	\end{equation*}
	where $M_s$ is the determinant we obtain by removing the row and column containing $g_s(x_1)$. 
	By induction, there exist a finite extension $ E' \supset \kawn $ and points $ x_2, \ldots, x_{d+1} \in \Xawn(E')$ such that $M_1 \neq 0$. 
	We now consider the function
	\begin{equation*}
		g' \colonequals g_1^{\pm M_1} \cdots g_{d+1}^{\pm M_{d+1}} \in \Unitsone(\Xawn) \period
	\end{equation*}
	Since the $g_j$ are linearly independent and $M_1 \neq 0$, the function $g'$ is not constant. 
	Thus, $g'$ is open, so there is a finite extension $ E \supset E' $ and a $x_1 \in \Xawn(E')$ such that $\nu(g'(x_1)) \neq 0 $ and the claim follows.
\end{proof}

The following is the main result of this paper:

\begin{theorem}\label{thm:reconstruction_for_amenable_fields}
	If $ k $ is an amenable field that admits a nontrivial discrete valuation, then $ k $ satisfies étale reconstruction.
\end{theorem}

\begin{proof}
	Let $ p $ denote the exponential characteristic of $ k $ and choose a separable closure $ \kbar \supset \kawn $.
	By \Cref{prop:discreteness_of_chi_top_implies_étale_reconstruction}, we need to show that for each affine, regular, connected, finite type $ k $-scheme $ X $ with a fixed rational point $ x \in X(k) $, we have 
	\begin{equation*}
		\image\paren{ \chitop \colon \fromto{\OTop(\Xawn)}{\Units(\Xawn)\pprimecomp} } \subset \image\paren{ \can \colon \fromto{\Units(\Xawn)}{\Units(\Xawn)\pprimecomp} } \period
	\end{equation*}
	Let $ \philowerstar \in \OTop(\Xawn) $ and write $ \xi \colonequals \chitop(\philowerstar) $.
	After scaling $ \philowerstar $, we can, without loss of generality assume that $ \philowerstar(x^{\awn}) = 1 $.
	Thus $ \xi $ lies in the group $ \Unitsone(\Xawn)\pprimecomp $.

	We note that the pairing 
	\begin{equation*}
		\evhat \colon \Xawn(\kbar) \times \Unitsone(\Xawn)\pprimecomp \to \colim_{E} \, (E^{\units})\pprimecomp
	\end{equation*}
	defined by \eqref{eq:pairing} extends the usual evaluation map
	\begin{equation*}
		\begin{tikzcd}
			\Xawn(\kbar) \times \Unitsone(\Xawn) \arrow[r] & \kbar^{\units} \arrow[r, "i", hooked]  & \colim_{E} \, (E^{\units})\pprimecomp \period
		\end{tikzcd}
	\end{equation*}
	Next, we observe that, by construction, $ \evhat(-,\xi) = \psi_{\chitop(\philowerstar)} $. 
	In particular, since $\philowerstar$ is pinned, $ \evhat(-,\xi)$ factors through $ i \colon \incto{\kbar^{\units}}{\colim_E (E^{\units})\pprimecomp} $.
	Thus by \Cref{prop:argument_using_discrete_valuation}, $ \xi $ is in the image of the natural homomorphism
	\begin{equation*}
		\can \colon \Unitsone(\Xawn) \to \Unitsone(\Xawn)\pprimecomp \comma
	\end{equation*}
	as desired.
\end{proof} 

\begin{corollary}\label{cor:finitely_generated_fields_satisfy_etale_reconstruction}
	If $ k $ is an infinite field that is also finitely generated, then $ k $ satisfies étale reconstruction.
\end{corollary}

\begin{proof}
	Since $ k $ is finitely generated, \Cref{prop:finitely_generated_fields_are_amenable} shows that $ k $ is amenable.
	Since $ k $ is also infinite, $ k $ admits a nontrivial discrete valuation.
	Hence $ k $ satisfies the hypotheses of \Cref{thm:reconstruction_for_amenable_fields}.
\end{proof}


\appendix

\section{Morphisms of étale sites and étale topoi}\label{app:morphisms_of_etale_sites_and_etale_topoi}

Given a scheme $ X $, write \smash{$ \Etqcs_X \subset \Et_X $} for the subsite of the étale site spanned by the quasicompact separated morphisms.
If $ X $ is quasiseparated, then \smash{$ \Etqcs_X $} is a basis for the étale topology (see \Cref{lem:separated_and_separated_finitely_presented_etale_sites_are_bases}).
In this appendix, we verify that if $ X $ and $ Y $ are topologically noetherian schemes, then every geometric morphism $ \flowerstar \colon \fromto{\Xet}{\Yet} $ is induced by a morphism of sites $ \fromto{\Etqcs_Y}{\Etqcs_X} $.
See \Cref{cor:every_geometric_morphism_between_topologically_noetherian_schemes_is_coherent}.
This allows us to repackage the étale reconstruction property in site-theoretic terms (see \Cref{cor:reformulation_of_etale_reconstruction_in_terms_of_sites}).
In \cite{MR1098621}, Voevodsky exclusively works with quasicompact separated étale sites; however, Grothendieck's conjecture \cite[p. 7]{MR1483108} is specifically about étale topoi.
So \Cref{cor:every_geometric_morphism_between_topologically_noetherian_schemes_is_coherent} also verifies that Voevodsky's site-theoretic result is equivalent to the topos-theoretic statement appearing in Grothendieck's letter.

We start by introducing a few bases for the étale site.
Before doing so, we remind the reader that quasicompactness and quasiseparatedness are \textit{topological} properties of schemes and morphisms of schemes.

\begin{recollection}[(quasicompactness and quasiseparatedness)]
	A map of topological spaces $ f \colon \fromto{S}{T} $ is \defn{quasicompact} (resp., \defn{quasiseparated}) if for every quasicompact (resp., quasiseparated) open $ V \subset T $, the preimage $ f^{-1}(V) \subset S $ is quasicompact (resp., quasiseparated).
	A morphism of schemes $ f \colon \fromto{X}{Y} $ is \defn{quasicompact} (resp., \defn{quasiseparated}) if the induced map on underlying topological spaces $ |f| \colon \fromto{|X|}{|Y|} $ is quasicompact (resp., quasiseparated).

	If $ S $ is a noetherian space, then every subset of $ S $ is quasicompact and quasiseparated.
	In particular, in this case, every map $ f \colon \fromto{S}{T} $ to an arbitrary topological space $ T $ is qcqs.
\end{recollection}


\begin{notation}[(variants of the étale site)]\label{ntn:variants_of_the_etale_site}
	Let $ X $ be a scheme.
	We write
	\begin{equation*}
		\Etsep_X \subset \Et_X \comma \qquad \Etqcs_X \subset \Et_X \comma \andeq \Etaff_X \subset \Et_X 
	\end{equation*}
	for the full subcategories spanned by those $ p \colon \fromto{X'}{X} $ such that $ p $ is separated, $ p $ is quasicompact and separated, and $ X' $ is affine over $ \Spec(\ZZ) $, respectively.
\end{notation}

\begin{observation}\label{obs:properties_of_variants_of_the_etale_site}
	Let $ X $ be a scheme.
	\begin{enumerate}[label=\stlabel{obs:properties_of_variants_of_the_etale_site}, ref=\arabic*]
		\item By cancellation, every morphism in $ \Etsep_X $ is separated and every morphism in $ \Etqcs_X $ is quasicompact and separated.
		Moreover, since separated and quasicompact morphisms are stable under basechange, the full subcategories
		\begin{equation*}
			\Etqcs_X \subset \Etsep_X \subset \Et_X
		\end{equation*}
		are closed under finite limits.

		\item Since affine schemes are separated and every morphism with separated source is separated, we have
		\begin{equation*}
			\Etaff_X \subset \Etsep_X \period
		\end{equation*}

		\item Assume that $ X $ is quasiseparated.
		Since every morphism from a quasicompact scheme to a quasiseparated scheme is quasicompact, we have
		\begin{equation*}
			\Etaff_X \subset \Etqcs_X \period
		\end{equation*}

		\item If $ X $ is qcqs, then every object of $ \Etqcs_X $ is qcqs.
	\end{enumerate}
\end{observation}

\begin{lemma}\label{lem:separated_and_separated_finitely_presented_etale_sites_are_bases}
	Let $ X $ be a scheme.
	Then:
	\begin{enumerate}[label=\stlabel{lem:separated_and_separated_finitely_presented_etale_sites_are_bases}, ref=\arabic*]
		\item\label{lem:separated_and_separated_finitely_presented_etale_sites_are_bases.1} The full subcategories
		\begin{equation*}
			 \Etaff_X \subset \Etsep_X \subset \Et_X
		\end{equation*}
		are both bases for the étale topology.
		Hence restriction along the inclusions define equivalences of topoi 
		\begin{equation*}
			\Sh_{\et}(\Et_X) \equivalence \Sh_{\et}(\Etsep_X) \equivalence \Sh_{\et}(\Etaff_X)  \period 
		\end{equation*}
		The inverses are given by right Kan extension along the inclusions.

		\item\label{lem:separated_and_separated_finitely_presented_etale_sites_are_bases.2} Assume that $ X $ is quasiseparated.
		Then $ \Etqcs_X \subset \Et_X $ is a basis for the étale topology.
		Hence restriction along the inclusion defines an equivalences of topoi
		\begin{equation*}
			\Sh_{\et}(\Et_X) \equivalence \Sh_{\et}(\Etqcs_X) \period 
		\end{equation*}
		The inverse is given by right Kan extension along the inclusion.
	\end{enumerate}
\end{lemma}

\begin{proof}
	First note that by \cite[Propositions B.6.4 \& B.6.6]{Ultracategories}, the statements about equivalences of topoi follow from the statements about bases.
	To prove \enumref{lem:separated_and_separated_finitely_presented_etale_sites_are_bases}{1}, first note that it suffices to prove that \smash{$ \Etaff_X \subset \Et_X $} is a basis.
	This is immediate from the fact that Zariski covers are étale covers.
	For \enumref{lem:separated_and_separated_finitely_presented_etale_sites_are_bases}{2} note that since $ X $ is quasiseparated, we have 
	\begin{equation*}
		\Etaff_X \subset \Etqcs_X \period
	\end{equation*}
	Hence the claim is a consequence of \enumref{lem:separated_and_separated_finitely_presented_etale_sites_are_bases}{1}.
\end{proof}

In the remainder of this subsection, we explain why every \textit{coherent} geometric morphism between the étale topoi of qcqs schemes is induced by a morphism between quasicompact separated étale sites.

\begin{recollection}[(coherent topoi)]
	Let $ \Xcal $ be a topos.
	An object $X \in \Xcal$ is \defn{quasicompact} if for every epimorphism
	\begin{equation*}
		\surjto{\coprod_{i \in I} X_i}{X}
	\end{equation*}
	there exists a finite subset $J \subset I$ such that $\coprod_{j \in J} X_j \rightarrow X$ is an epimorphism.  
	An object $X \in \Xcal$ is \defn{quasiseparated} if for all diagrams $Y \rightarrow X \leftarrow Y' $ with $ Y $ and $Y'$ quasicompact, the pullback $Y \times_X Y'$ is quasicompact. 
	Finally, an object $ X \in \Xcal$ is \defn{coherent} if $ X $ is both quasicompact and quasiseparated. 
	A geometric morphism $ \flowerstar \colon \Xcal \rightarrow \Ycal$ of topoi is \defn{coherent} if for every coherent object $Y \in \Ycal$, the pullback $\fupperstar(Y)$ is coherent.
	See \cites[Exposé VI]{MR50:7131}[\S C.5]{Ultracategories} for more on coherent topoi.
\end{recollection}

\begin{example}\label{ex:coherent_objects_of_the_etale_topos}
	Let $ X $ be a qcqs scheme.
	Then the étale topos $ \Xet $ is coherent.
	Moreover, \cite[Exposé IX, Proposition 2.7]{MR50:7131} shows that an étale sheaf $ \Fcal \in \Xet $ is coherent if and only if $ \Fcal $ is constructible.
\end{example}

\begin{lemma}\label{lem:coherence_for_geometric_morphisms_between_etale_topoi}
	Let $ X $ and $ Y $ be qcqs schemes and let $ \flowerstar \colon \fromto{\Xet}{\Yet} $ be a geometric morphism.
	The following are equivalent:
	\begin{enumerate}[label=\stlabel{lem:coherence_for_geometric_morphisms_between_etale_topoi}, ref=\arabic*]
		\item The geometric morphism $ \flowerstar $ is coherent.

		\item The induced map of topological spaces $ |\flowerstar| \colon \fromto{|X|}{|Y|} $ is quasicompact.
	\end{enumerate}
\end{lemma}	

\begin{proof}
	If $\flowerstar$ is coherent, then by restricting the morphism $\flowerstar$ to subobjects of the terminal object of $\Yet$ we deduce that induced map $ |\flowerstar| $ is quasicompact.
	Conversely, assume that $ |\flowerstar| $ is quasicompact.  
	Since $ X $ and $ Y $ are qcqs, in light of \Cref{ex:coherent_objects_of_the_etale_topos} we need to show that for each constructible sheaf $\Fcal \in \Yet $, the pullback $ \fupperstar(\Fcal) $ is constructible.
	By assumption, there is a finite stratification $\{Y_i\}_{i \in I}$ of $ Y $ where each $ Y_i \subset Y $ is qcqs and locally closed such that on each stratum, $\restrict{\Fcal}{Y_i}$ is finite locally constant. 
	Since $ |\flowerstar| $ is quasicompact, the stratification pulls back to a stratification $\{|\flowerstar|^{-1}(Y_i)\}_{i \in I}$ of $ X $ by qcqs locally closed subschemes such that on each stratum, $\fupperstar(\Fcal)$ is finite locally constant. 
\end{proof}

\begin{theorem}\label{thm:every_coherent_geometric_morphism_between_etale_topoi_of_qcqs_schemes_is_induced_by_a_morphism_of_sites}
	Let $ X $ and $ Y $ be qcqs schemes and let $ \flowerstar \colon \fromto{\Xet}{\Yet} $ be a coherent geometric morphism.
	Then for each $ V \in \Etqcs_Y $, the étale sheaf $ \fupperstar(V) $ is representable by an object of $ \Etqcs_X $.
	Hence $ \flowerstar $ is induced by the morphism of sites
	\begin{equation*}
		\fupperstar \colon \fromto{\Etqcs_Y}{\Etqcs_X} \period
	\end{equation*}
\end{theorem}

\begin{proof}
    The diagonal $ V \to V \times_Y V  $ is a clopen immersion; let $ V' $ be the complementary clopen subscheme.
    Since $ \flowerstar $ is coherent, by \Cref{lem:coherence_for_geometric_morphisms_between_etale_topoi} the sheaves $ \fupperstar(V) $ and $ \fupperstar(V') $ are constructible.
    Hence $ \fupperstar(V) $ and $ \fupperstar(V') $ are represented by an étale algebraic spaces $ U \to X $ and $ U' \to X $.
    Furthermore, pulling back the isomorphism $ \isomto{V \coproduct V'}{V \times_Y V} $, we obtain an isomorphism
    \begin{equation*}
        U \coproduct U' \isomorphism U \times_X U \period
    \end{equation*}
    Thus the diagonal of $ U $ is a closed immersion, equivalently, $ U \to X $ is separated. 
    Since an étale algebraic space is also locally quasi-finite, it follows from \stacks{03XX} that $ U $ is a scheme and $ U \to X $ is quasicompact and separated, as desired.
\end{proof}

\begin{corollary}\label{cor:every_geometric_morphism_between_topologically_noetherian_schemes_is_coherent}
	Let $ X $ and $ Y $ be schemes.
	If $ X $ is topologically noetherian and $ Y $ is qcqs, then:
	\begin{enumerate}[label=\stlabel{cor:every_geometric_morphism_between_topologically_noetherian_schemes_is_coherent}, ref=\arabic*]
		\item\label{cor:every_geometric_morphism_between_topologically_noetherian_schemes_is_coherent.1} Every geometric morphism $ \flowerstar \colon \fromto{\Xet}{\Yet} $ is coherent. 
		
		\item\label{cor:every_geometric_morphism_between_topologically_noetherian_schemes_is_coherent.2} Every geometric morphism $ \flowerstar \colon \fromto{\Xet}{\Yet} $ is induced by a morphism of quasicompact separated étale sites $ \fromto{\Etqcs_Y}{\Etqcs_X} $.
	\end{enumerate}
\end{corollary}

\begin{proof}
	By \Cref{thm:every_coherent_geometric_morphism_between_etale_topoi_of_qcqs_schemes_is_induced_by_a_morphism_of_sites}, it suffices to prove \enumref{cor:every_geometric_morphism_between_topologically_noetherian_schemes_is_coherent}{1}. 
	For this, note that by \Cref{lem:coherence_for_geometric_morphisms_between_etale_topoi}, it suffices to shows that $ |\flowerstar| \colon \fromto{|X|}{|Y|} $ is quasicompact.
	Since $ |X| $ is noetherian, this is automatic.
\end{proof}

We conclude by reformulating the étale reconstruction property (see \Cref{prop:equivalent_conditions_for_etale_reconstruction,def:etale_reconstruction}) in terms of sites.

\begin{notation}
	Let $ k $ be a field and let $ X $ and $ Y $ be $ k $-schemes.
	Assume that $ X $ is topologically noetherian and $ Y $ is qcqs.
	We write $ \Hom_k(\Etqcs_Y,\Etqcs_X) $ for the groupoid of morphisms $ f^{-1} \colon \fromto{\Etqcs_Y}{\Etqcs_X} $ in the $ (2,1) $-category of sites under the étale site of $ \Spec(k) $.
	We write 
	\begin{equation*}
		\Hompink(\Etqcs_Y,\Etqcs_X) \subset \Hom_k(\Etqcs_Y,\Etqcs_X) 
	\end{equation*}
	for the full subgroupoid spanned by those morphisms of sites $ f^{-1} \colon \fromto{\Etqcs_Y}{\Etqcs_X} $ such that the induced geometric morphism
	\begin{equation*}
		\begin{tikzcd}
			\Xet \equivalent \Sh_{\et}(\Etqcs_X) \arrow[r, "\flowerstar"] & \Sh_{\et}(\Etqcs_Y) \equivalent \Yet
		\end{tikzcd}
	\end{equation*}
	is a pinned geometric morphism in the sense of \Cref{def:pinned_geometric_morphism}.
\end{notation}

\begin{corollary}\label{cor:reformulating_morphisms_of_topoi_in_terms_of_morphisms_of_sites}
	Let $ k $ be a field and let $ X $ and $ Y $ be $ k $-schemes.
	Assume that $ X $ is topologically noetherian and $ Y $ is qcqs.
	Then the natural functors
	\begin{align*}
		\Hom_k(\Etqcs_Y,\Etqcs_X) \to \Hom_k(\Xet,\Yet) \andeq \Hompin_k(\Etqcs_Y,\Etqcs_X) \to \Hompin_k(\Xet,\Yet)
	\end{align*}
	given by $ \goesto{f^{-1}}{\flowerstar} $ are equivalences of groupoids.
\end{corollary}

\begin{proof}
	Immediate from \Cref{cor:every_geometric_morphism_between_topologically_noetherian_schemes_is_coherent} and the definitions.
\end{proof}

\begin{nul}
	In light of \Cref{prop:rigidity,cor:reformulating_morphisms_of_topoi_in_terms_of_morphisms_of_sites}, if $ X $ and $ Y $ are topologically of finite type over $ k $, then the groupoid $ \Hompin_k(\Etqcs_Y,\Etqcs_X) $ is equivalent to a set.
\end{nul}

\begin{corollary}\label{cor:etale_reconstruction_via_sites_for_two_schemes}
	Let $ k $ be a field and let $ X $ and $ Y $ be schemes topologically of finite type over $ k $.
	Then the following are equivalent:
	\begin{enumerate}[label=\stlabel{cor:etale_reconstruction_via_sites_for_two_schemes}, ref=\arabic*]
		\item The natural map 
		\begin{equation*}
			\Hom_k(X,Y) \to \Hompink(\Etqcs_Y,\Etqcs_X) \comma \qquad f \mapsto f^{-1}
		\end{equation*}
		is bijective.

		\item The natural map $ \fromto{\Hom_k(X,Y)}{\Hompink(\Xet,\Yet)} $ is bijective.
	\end{enumerate}
\end{corollary}

\begin{remark}
	\Cref{cor:etale_reconstruction_via_sites_for_two_schemes} shows that Voevodsky's site-theoretic result \cite[Theorem 3.1]{MR1098621} is equivalent to the following statement: for every finitely generated field $ k $ of characteristic $ 0 $ and finite type $ k $-schemes $ X $ and $ Y $ with $ X $ normal, the natural map
	\begin{equation*}
		\fromto{\Hom_k(X,Y)}{\Hompink(\Xet,\Yet)}
	\end{equation*}
	is bijective.
\end{remark}

\begin{corollary}\label{cor:reformulation_of_etale_reconstruction_in_terms_of_sites}
	Let $ k $ be a field.
	Then $ k $ satisfies étale reconstruction if and only if for all schemes $ X $ and $ Y $ topologically of finite type over $ k $ with $ X $ absolutely weakly normal, the natural map
	\begin{equation*}
		\Hom_k(X,Y) \to \Hompink(\Etqcs_Y,\Etqcs_X)
	\end{equation*}
	is bijective.
\end{corollary}	


\DeclareFieldFormat{labelnumberwidth}{#1}
\printbibliography[keyword=alph]
\DeclareFieldFormat{labelnumberwidth}{{#1\adddot\midsentence}}
\printbibliography[heading=none, notkeyword=alph]

\end{document}